\documentclass[12pt,reqno]{amsart}
\usepackage{latexsym}
\usepackage{amssymb}
\usepackage{verbatim}
\usepackage{version}
\newtheorem{theorem}{Theorem}
\newtheorem{proposition}{Proposition}
\newtheorem{lemma}{Lemma}
\newtheorem{corollary}{Corollary}

\newtheorem{te}{Theorem}[section]

\newtheorem{pr}[te]{Proposition}
\newtheorem{rem}[te]{Remark}
\newtheorem{rems}[te]{Remarks}
\newtheorem{co}[te]{Corollary}
\newtheorem{lem}[te]{Lemma}
\newtheorem{prob}[te]{Problem}

\newtheorem{de}[te]{Definition}

\newcommand{\Z}{\mathbb{Z}}

\newcommand{\R}{\mathbb{\R}}

\newcommand{\s}{\mathfrak{s}}

\includeversion{comment}

\def\R{\mathbb{R}}

\def\N{\mathbb{N}}

\def\H{\mathbb{H}}
\def\P{\mathbb{P}}
\def\Z{\mathbb{Z}}

\def\cA{\mathcal{A}}

\def\cE{\mathcal{E}}
\def\cF{\mathcal{F}}
\def\cG{\mathcal{G}}
\def\cH{\mathcal{H}}

\def\cL{\mathcal{L}}

\def\cP{\mathcal{P}}

\def\cR{\mathcal{R}}
\def\cS{\mathcal{S}}
\def\cT{\mathcal{T}}

\def\cX{\mathcal{X}}

\newcommand{\Ker}{{\rm Ker\,}}
\def\be{\begin{equation}}
\def\ee{\end{equation}}

\newcommand{\w}{\mathfrak{w}}

\def\bp{\begin{proof}}
\def\ep{\end{proof}}

\def\ben{\begin{enumerate}}
\def\een{\end{enumerate}}
\def\ba{\begin{eqnarray*}}
\def\ea{\end{eqnarray*}}

\def\lra{\longrightarrow}
\def\Llra{\Longleftrightarrow}
\def\Lra{\Longrightarrow}

\def\w4{\,\widehat{1/4}}
\newcommand{\wh}{\widehat}
\newcommand{\la}{\langle}
\newcommand{\ra}{\rangle}
\newcommand{\wt}{\widetilde}
\newcommand{\sm}{\setminus}
\newcommand{\sse}{\subseteq}
\newcommand{\es}{\varnothing} 
\newcommand{\fG}{\mathfrak{G}}

\newcommand{\fP}{\mathfrak{P}}
\newcommand{\fH}{\mathfrak{H}}
\newcommand{\fX}{\mathfrak{X}}

\newcommand{\fm}{\mathfrak{m}}

\newcommand{\q}{\quad}

\newcommand{\eps}{\varepsilon}
\newcommand{\gam}{\gamma}
\newcommand{\Gam}{\Gamma}
\newcommand{\Lam}{\Lambda}
\newcommand{\lam}{\lambda}
\newcommand{\om}{\omega}
\newcommand{\Om}{\Omega}

\newcommand{\si}{\sigma}

\def\cA{\mathcal{A}}

\def\cE{\mathcal{E}}
\def\cF{\mathcal{F}}
\def\cG{\mathcal{G}}

\def\cR{\mathcal{R}}
\def\cS{\mathcal{S}}

\def\cL{\mathcal{L}}

\def\cH{\mathcal{H}}
\def\cX{\mathcal{X}}

\def\O{\mathcal{O}}
\newenvironment{proofof}[2]{\par{\em Proof of #1}~#2}{{\hfill$\Box$\par}}
\begin{document}
\title[Group actions, polygroup extensions, and group presentations]
{Group actions, deformations, polygroup extensions, and group presentations}
\author[{\c S}.\,A. Basarab and T.\,W. M\"uller]{{\c S}erban Basarab$^\dagger$ and Thomas M\"uller$^*$}

\address{$^\dagger$Institute  of Mathematics of the Romanian Academy, 
 P.O.Box~1-764, RO-70109 Bucharest, Romania}

\address{$^*$School of Mathematical Sciences, Queen Mary
\& Westfield College, University of London,
Mile End Road, London E1 4NS, United Kingdom.}

\thanks{$^*$Research supported by Lise Meitner Grant M1661-N25 of the Austrian Science Foundation FWF}



\begin{abstract}
Generalizing classical extension theory, we solve a Schreier-type extension problem for polygroups by groups. As a consequence, we obtain a method for computing a presentation for a group from its action on a set. The usefulness of this method is illustrated by deriving explicit presentations for the groups $GL_2$ over valuation rings and over valued fields, for the groups $SL_3$ over arbitrary fields, as well as for the five Mathieu groups. Moreover, we sketch some aspects of a new deformation technique for groups, their actions, and presentations,  and apply it to compute presentations for the sharply $3$-transitive Zassenhaus groups $M(q^2)$ (in the notation of Huppert and Blackburn) for any odd prime power $q$. This computation serves to demonstrate how suitable deformation of groups and their actions interacts with, and thereby enhances, the presentation method.
\end{abstract}

\maketitle

\begin{center}
\textit{Dedicated to the memory of John R. Stallings (1935--2008)}
\end{center}

\section*{Introduction}
Ever since Felix Klein's 1872 Erlangen Program \cite{Klein}, the close connection between group theory and geometry has inspired a fair amount of mathematical research. Klein's original idea was to study geometries via the invariants of the group(s) of symmetries naturally associated with this geometry, and to understand the hierarchy of geometries in terms of the corresponding hierarchy of their associated groups. It soon transpired however that this idea is not a one-way road: it may be possible, conversely, to make an abstract group act on some sort of geometric space, and to derive structural information concerning this group from details of such an  action. This strategy is already present -- in somewhat imprecise form -- in Poincar\'e's first Acta memoir \cite{Poin} (see in particular sections III and X), where he explains how to obtain a presentation for a Fuchsian group from its action on a (previously constructed) fundamental domain; in full generality the defining relations of a Fuchsian group were first obtained  by Klein \cite{Klein2}.\footnote{The same approach also underlies the classification given by Klein and Fricke \cite{FrickeKlein} of Fuchsian groups with compact orbit space.} This idea of deriving structural information (usually in the form of a presentation) for a group from its action on a suitable space has, subsequently, been made precise by a number of authors in a variety of contexts; cf.\ Abels~\cite{Abels}, Bass~\cite{Bass}, Behr~\cite{Behr}, Bridson and Haefliger~\cite{BridHaef}, Brown~\cite{Brown}, Gerstenhaber~\cite{Gerstenhaber}, Haefliger~\cite{Haefliger}, Macbeath~\cite{Macbeath}, Serre~\cite{Trees}, Siegel~\cite{CLS1}, Soul\'e~\cite{Soule}, Stallings~\cite{Stallings}, Swan~\cite{Swan}, Tits~\cite[App.~2]{Tits}, and Weil~\cite{Weil} for a nonexhaustive sample.  

The present paper introduces three new (though interrelated)  ideas: (i) a generalisation of group extension theory, which works without the assumption of normality, (ii) a deformation technique for groups, actions, and presentations, and (iii) a method for obtaining presentations from group actions on sets. In this introduction, we shall briefly explain these ideas, and put them into context. 

Firstly, we contribute to the line of research described in the previous paragraph by establishing a method for deriving a presentation of a group from its action on a mere set (finite or infinite, with or without extra structure). The corresponding result, Theorem~\ref{Thm:PresGpAct} in Section~\ref{Subsec:PresGpAct}, is both general and powerful. We illustrate its usefulness in Sections~\ref{Sec:GL2}--\ref{Sec:SL3}, by analysing the groups $GL_2$ over valuation rings and over fields with valuation, and the groups $SL_3$ over arbitrary (discrete) fields, obtaining explicit presentations for these groups; cf.\ Theorems~\ref{Th:PresGL_2(O)}, \ref{The:PresGL_2(K)}, and \ref{Thm:SL3Pres}. In Section~\ref{Sec:Mathieu}, we derive, in a uniform way, new presentations for the five Mathieu groups by combining Theorem~\ref{Thm:PresGpAct} (or rather its consequence Corollary~\ref{Cor:Main2Trans}, where the action is assumed to be multiply transitive) with a version of Witt's construction \cite{Witt}. 

The group actions underlying the results of Sections~\ref{Sec:GL2} and \ref{Sec:SL3} are the natural action of $GL_2(\O)$ on the projective line $\mathbb{P}^1(K)$, where $K$ is the field of fractions of the valuation ring $\O$, and the ($2$-transitive) action of $SL_3(K)$ on the points of the  projective plane $\mathbb{P}^2(K)$, respectively. In Section~\ref{Sec:GL2/VF}, the main result, Theorem~\ref{The:PresGL_2(K)}, is obtained from the action of $GL_2$ over the valued field $(K, v)$ on a suitable $\Lambda$-tree $\mathfrak{X}$, where $\Lambda = v(K^\times)$ is the value group of $v$, whose underlying set is the global residue structure 
\[
\mathfrak{X} = \bigsqcup_{\alpha\in\Lambda} K/\{z\in K: v(z) \geq \alpha\}
\]
of $(K, v)$. This action may be viewed as a deformation, induced by the valuation $v$, of the corresponding projective action.

At present, two approaches to Theorem~\ref{Thm:PresGpAct} are known, each leading further into a fruitful area of research: the theory of group actions on (connected) groupoids, and generalized extension theory. 

In the first case, Theorem~\ref{Thm:PresGpAct} appears as the special case of a structure theorem for such group actions  where the action is of simplicial type (i.e., the corresponding groupoid is a blow-up of a non-empty set); cf.\ \cite{Ba1} and \cite[Chap.~I]{M2}. Incorporating further geometric structure elements, this line of thought eventually leads to a Bass-Serre-type theory for group actions on $\Lambda$-trees, where $\Lambda$ is a totally ordered (or even a lattice-ordered) abelian group. The latter theory, whose presentation requires substantial preparation, will be described in a separate publication \cite{M2}. 

A more direct road to Theorem~\ref{Thm:PresGpAct}, which we adopt in the present paper, is via a suitable generalization of group extension theory. Given a group $G$ and a subgroup $H\leq G$, let $E = H\backslash G/H$ be the corresponding space of $(H,H)$-double cosets. If $H$ is normal in $G$, the set $E$, which carries a natural polygroup structure, is isomorphic to the quotient group $G/H$, and $G$ is an extension of $E$ by $H$. In general, for $H$ not necessarily normal in $G$, the group $G$ may still be viewed as a kind of extension of the polygroup $E$ by the group $H$. One is thus led to formulate a \textit{Schreier-type extension problem} for such generalized extensions, and it is this extension problem which is solved in Section~\ref{Sec:ExtPr}; cf.\ Theorem~\ref{Thm:SolExtP}. Just as Schreier described (equivalence classes of) group extensions via (equivalence classes of) factor systems ($2$-cocycles), Theorem~\ref{Thm:SolExtP} parametrizes (equivalence classes of) group extensions of polygroups by groups via (equivalence classes of) certain objects called \textit{group-like graphs of groups}, which arise from a suitable extension and modification of Serre's graph of groups concept. Theorem~\ref{Thm:PresGpAct} is an immediate consequence of results obtained in Section~\ref{Sec:ExtPr}. 

The theory of generalized group extensions, as initiated in Section~\ref{Sec:ExtPr}, while clearly in need of further research, is of considerable interest in its own right. Even at the present (still very incomplete) state, it has a number of applications beyond a proof of Theorem~\ref{Thm:PresGpAct}: new characterisations for a number of interesting groups (for instance, symmetric or sporadic simple groups) via their corresponding extension space; a general structure theory for $2$-transitive groups (finite or infinite), in the finite case largely without recourse to the classification of finite simple groups; a better understanding of infinite sharply $2$-transitive groups, shedding new light on a recent result by Rips, Segev, and Tent \cite{RST} concerning the existence of infinite sharply $2$-transitive groups without abelian normal subgroups, etc. These results, together with a further development of the underlying extension theory will eventually be presented in a separate publication; cf.\ \cite{M}.

The third idea introduced in this paper concerns a new kind of deformation technique for groups, their associated actions, and presentations. Technically speaking, we use actions $\varphi: G\rightarrow \mathrm{Aut}(G)$ by automorphisms of a group $G$ on itself to twist the extension class of $G$, with $G$ being viewed as an extension of the image $\varphi(G)$ by the kernel of $\varphi$, and to deform associated $G$-actions and presentations. The machinery of group deformations was originally conceived by the first author in the context of profinite and abstract co-Galois theory (see \cite{CGA} and \cite{GFrame}); it was subsequently generalised and developed further by the authors; cf.\ \cite{BM1}. In order to keep the length of this paper within reasonable bounds, 
we confine ourselves in Section~\ref{Sec:Deform} to  describing some of the more elementary (discrete) aspects of deformation theory, demonstrating by means of an example how this technique interacts with (and thereby enhances) the presentation method of Theorem~\ref{Thm:PresGpAct}. The main result of that section, Theorem~\ref{Thm:MPres}, provides a presentation for the sharply $3$-transitive Zassenhaus group $\mathrm{M}(q^2)$ (in the notation of Huppert and Blackburn \cite{HuBl}) for any odd prime power $q=p^m$, the group $\mathrm{M}(q^2)$ being viewed as a deformation (not as a twist!) of the group $\mathrm{PGL}_2(q^2)$. The smallest of these groups, $\mathrm{M}(3^2)$, plays a role in Section~\ref{Subsec:SmallMathieu}, being the stabilizer of a point in the Mathieu group $\mathrm{M}_{11}$.

The contents of this paper represents the first fruit of a, rather intense, collaboration of the authors, which started in 2009, and came to an abrupt end with the unexpected death of {\c S}erban Basarab on 14 July 2014 (TWM). 

\textit{Contents}
\begin{enumerate}
\item[\S1.] Hypergroups and polygroups.
\item[\S2.] Presentations arising from group actions on sets.
\item[\S3.] $GL_2$ over valuation rings.
\item[\S4.] $GL_2$ over valued fields.
\item[\S5.] $SL_3$ over fields.
\item[\S6.] Deformations of groups and presentations.
\item[\S7.] The Mathieu groups
\item[\S8.] Group extensions of polygroups by groups.
\end{enumerate}

\section{Hypergroups and polygroups}
\label{Sec:PolyGps}

For the convenience of the reader, this brief section collects  together some definitions and facts from the theory of hypergroups and polygroups which are needed in the sequel.
\subsection{Definition of a hypergroup} 
The concept of a hypergroup, introduced by F. Marty \cite{M1} in 1934, originally arose as a generalization of the 
concept of an abstract group. Hypergroups play a central role in the theory of algebraic hyperstructures, having connections with 
classical algebraic structures as well as manifold  applications in geometry, topology, combinatorics, probability theory, category theory, 
logic, etc.; cf.\ \cite{Yaacov}--\cite{Dietzmann}, \cite{Harrison}, \cite{Hos/Chv}, \cite{Krasner}, \cite{Mittas1}--\cite{Mittas3}, \cite{Prenowitz1}--\cite{Roth}, and \cite{Tallini}--\cite{Voug}, for a sample.

We state some definitions. Let $E$ be a non-empty set. Denoting by $\cP^\ast(E)$ the collection of all non-empty subsets of $E$, the injective map $E \rightarrow \cP^\ast(E),$ $a \mapsto \{a\}$, identifies $E$ with a subset of $\cP^\ast(E)$. 
A {\em hyperoperation} $\circ$ in $E$ is a map $E \times E \rightarrow \cP^\ast(E),$ denoted by $(a, b) \mapsto a \circ b$. The 
hyperoperation $\circ$ is extended to non-empty subsets of $E$ in a natural way by setting  
$A \circ B:= \bigcup\limits_{(a,b)\in A\times B} a \circ b$.
\begin{de}
{\em A non-empty set $E$ endowed with a hyperoperation $\circ$ is called a {\em hypergroup} if the following are satisfied for all $a, b, c \in E$:
\ben
\vspace{-2mm}
\item[\rm (i)]  $a \circ (b \circ c) = (a \circ b) \circ c$, that is,
$\bigcup_{u \in b \circ c} a \circ u = \bigcup_{v \in a \circ b} v \circ c$\, (Associativity);
\vspace{2mm}
\item[\rm (ii)]  $E \circ a = E = a \circ E$, that is, $\bigcup_{u \in E} u \circ a = E = 
\bigcup_{u \in E} a \circ u$\, (Reproduction Axiom).
\een}
\end{de}

{\footnotesize 
\begin{rems} \em
(1) The associativity condition (i) implies associativity for non-empty subsets of $E$; that is, we have 
$A \circ (B \circ C) = (A \circ B) \circ C$.

\vspace{-2mm}

(2) The reproduction axiom (ii) is equivalent to the following:
\[
b/a:= \{x \in E:\,b \in x \circ a\} \neq \es\,\,\mbox{and}\,\,a \backslash b:= \{x \in E:\,b \in a \circ x\} 
\neq \es\,\,\mbox{ for all}\,\,a, b \in E.
\]
\end{rems}}

\subsection{Polygroups}

We shall need the more restrictive notion of a polygroup introduced by S.\,D. Comer in \cite{Comer}.
\begin{de} 
{\em A {\em polygroup} or {\em quasicanonical hypergroup} is a hypergroup $(E, \circ)$ satisfying the following conditions:
\ben
\vspace{-2mm}
\item[\rm (i)] there exists a {\em scalar identity}, that is, there exists some $e \in E$ such that \\
$a \circ e = a = e \circ a$ for all $a \in E$;

\vspace{2mm}
 
\item[\rm (ii)] for each $a \in E$, there exists some $\bar{a} \in E$ such that, for all $b \in E$, \\
$b/a = b \circ \bar{a}$ and $a \backslash b = \bar{a} \circ b$.  
\een}
\end{de}

{\footnotesize 
\begin{rems} 
\label{Rem:Poly}
\em
(1) The scalar identity $e$ is necessarily unique.

\vspace{-2mm}

(2) For any $a \in E$, the element $\bar{a}$ is unique; moreover, we have $e \in (a \circ \bar{a}) \cap (\bar{a} \circ a)$,\, $e = \bar{e}$,\,
$\bar{\bar{a}} = a$, and $\overline{a \circ b}:= \{\bar{x}\,:\,x \in a \circ b\} = \bar{b} \circ \bar{a}$.

\vspace{-2mm}

(3) A non-empty set $E$, endowed with an associative hyperoperation $\circ$ and a scalar identity $e$, is a polygroup
if, and only if, for each $a \in E$, there exists $\bar{a} \in E$ such that, for all $a, b, c \in E$, 
$a \in b \circ c$ implies $b \in a \circ \bar{c}$ and $c \in \bar{b} \circ a$.  

\vspace{-2mm}

(4) a hypergroup $(E, \circ)$ with a scalar identity is a polygroup if, and only if, 
the following {\em transposition condition} is satisfied:
\[
(b \backslash a) \cap (c/d) \neq \es \Lra (a \circ d) \cap (b \circ c) \neq \es\,\,\mbox{for all}\,\,a, b, c, d \in E;
\] 
cf.\ \cite[Corollary 3]{J}.
\end{rems}}

{\em Commutative polygroups}, that is, polygroups $(E,\circ, e, {}^-)$ satisfying $a\circ b=b\circ a$ for all $a, b\in E$, are also called {\em canonical hypergroups}. Note that, by Part~(2) of Remarks~\ref{Rem:Poly},  polygroups $E$ satisfying $\bar{a} = a$ for all $a \in E$ are commutative.

Given a polygroup $(E;\,\circ,\,e,\,\bar{}\,\,)$, we denote by $\Gam(E)$ its {\em associated graph} consisting of one vertex and $E$ as set of edges, together with the involution\, $\bar{}: E \rightarrow E$, $f \mapsto \bar{f}$. By an {\em orientation} of the 
graph $\Gam(E)$ we mean a subset $E_+ \sse E$ such that 
\[
E = E_+ \amalg \big\{f \in E\,:\,\bar{f} \in E_+\,\mbox{and}\, f \neq \bar{f}\big\};
\] 
in particular, $e = \bar{e} \in E_+$. Note that the graph $\Gam(E)$ does not contain any information concerning the hyperoperation $\circ$.

\subsection{Isomorphisms of polygroups}

An \emph{isomorphism} $\Phi: E_1\rightarrow E_2$ between two polygroups $E_1 = (E_1, \circ,\,\bar{}, e_1)$ and $E_1 = (E_1, \circ,\,\bar{}, e_1)$ is a bijection $E_1\rightarrow E_2$ between the underlying sets, which respects the hyperoperation in the sense that
\[
\Phi(a\circ b) = \Phi(a) \circ \Phi(b),\quad (a, b\in E_1);
\] 
that is, which is a strong homomorphism in the usual terminology of hypergroup theory; cf.\, for instance, Definitition~3.3.15 in \cite{Davvaz}. It can be shown that an isomorphism of polygroups $\Phi: E_1 \rightarrow E_2$ respects inversion, that is, it satisfies 
\[
\Phi(\bar{a}) = \overline{\Phi(a)},\quad a\in E_1,
\] 
and maps the scalar identity of $E_1$ onto that of $E_2$. 

\subsection{Double coset spaces as polygroups}
\label{Subsec:DoubleCosetSpace}
Given a group $G$ and a subgroup $H$, consider the set $E:= H \backslash G /H$ of $(H, H)$-double cosets 
$C(g) = H g H$ for $g \in G$. The set $E$ comes equipped with a canonical polygroup structure  
$(E;\,\circ,\,e,\,\bar{}\,\,)$, with the associative hyperoperation $\circ$ given by 
\[
C(g_1) \circ C(g_2) = \big\{C(g_1 h g_2):\,h \in H\big\}
\] 
(so that $C(g_1)\circ C(g_2)$ results by splitting the complex product $Hg_1H\cdot Hg_2H$ into double cosets modulo $H$), the scalar identity $e = C(1) = H$, and the involution \, $\bar{}: E\rightarrow E$ given by $C(g) \mapsto \overline{C(g)} := C(g^{- 1})$. Double coset spaces as polygroups seem to have first been explicitly discussed in \cite{DO}.

The polygroup $E$ associated with the pair $(G, H)$ is a group if, and only if, $H$ is a normal subgroup of $G$; in this case, $E$ is isomorphic to the quotient group $G/H$, so that $G$ is a group extension of $E$ by $H$. In general, for $H$
not necessarily normal in $G$, the group $G$ may still be viewed  as an extension of the polygroup $E$ by the group $H$; see Section~\ref{Sec:ExtPr}, where the corresponding more general extension problem is formulated and solved.

\section{Presentations arising from group actions on sets}
\subsection{The general result}
\label{Subsec:PresGpAct}
Let $G$ be a group acting transitively from the left on a set $\Omega,$ let $\omega_0$ be a given point of $\Omega$, and let $H:=G_{\omega_0}$ be the stabiliser of $\omega_0$ in $G$. The bijection $\varphi: G/H \rightarrow \Om$, induced by the
surjective map $G \rightarrow \Om = G \om_0$, $g \mapsto g \om_0$ is $H$-equivariant with respect to the action of $H$ by left multiplication on the set $G/H$ and the action of $H$ on $\Om$ induced by restriction from the given transitive action of $G$.

The bijection $\varphi: G/H \lra \Om$ thus induces a bijection
\[
\Phi: H \backslash G / H \lra H \backslash \Om,\quad C(g) \mapsto H \cdot g \om_0,   
\]
from the set $H \backslash G / H$ of double cosets $C(g):= H g H$ with $g \in G$ onto the set $H \sm \Om$ of $H$-orbits $H \om$ 
with $\om \in \Om$. Moreover the sets $H \backslash G / H$ and $H \backslash \Om$ are equipped with canonical structures of polygroups, and the map $\Phi$ is an isomorphism of polygroups. Explicitly, the associative hyperoperation $\circ$ on $H\backslash\Omega$ is defined by
\begin{equation*}
H \om_1 \circ H \om_2 = \big\{H \cdot g h \om_2\,:\,h \in H\big\},\quad(\omega_1, \omega_2\in \Omega,\,\omega_1 = g\omega_0),  
\end{equation*}
the scalar identity is $H \om_0 = \{\om_0\}$, while inversion is defined by $\overline{H \om} = 
H \cdot g^{- 1} \om_0$ for $\om = g \om_0 \in \Om$. (The polygroup structure of the double coset space $H\backslash G/H$ has already been discussed in Section~\ref{Subsec:DoubleCosetSpace}.) 

Note that the polygroup $E:= H \sm G/ H \cong H \sm \Om$ is a group
if, and only if, $H$ is a normal subgroup of $G$; in this case, the stabilizer $H = G_{\om_0}$ is the kernel of the 
transitive action of $G$ on the set $\Om$, $E$ is the quotient group $G/H$, thus $G$ is an extension of the group 
$E = G/H$ by the group $H$. In the general case, when the subgroup $H$ is not necessarily normal in $G$, the group $G$ 
may still be seen as an extension of the polygroup $E$ by the group $H$. The more general form of the classical 
Schreier group extension problem concerning group extensions of polygroups by groups is discussed in 
Section~\ref{Sec:ExtPr}. In this more general framework,  Schreier's factor systems ($2$-cocycles) are replaced by certain {\em group-like graphs of groups}.\footnote{The concept of a group-like graph of groups was originally introduced in \cite[Section 4]{Ba1} in the context of group actions on connected groupoids; cf.\ also \cite[Chap.~1]{M2}.} In particular, we obtain a presentation of the group $G$ in terms of 
the subgroup $H$ and the polygroup $E = H \backslash G / H \cong H \backslash \Om$ as follows.

We fix a section $\si: E \rightarrow G$ of the projection map $G \rightarrow E = H \backslash G / H$ satisfying 
$\si(e) = 1$, $\si(\bar{f}) = \si(f)^{- 1}$ for $\bar{f} \neq f$, and $\theta_f:= \si(f)^2 \in H$ for 
$\bar{f} = f$. For $f \in E$, set $H_f:= H \cap \si(f) H \si(f)^{- 1}$, and define an isomorphism
$\iota_f: H_f \rightarrow H_{\bar{f}}$ by $\iota_f(h) = \si(f)^{- 1} h \si(f)$. Note that $H_e = H$, that $\iota_e = 1_H$, and that 
$\iota_{\bar{f}} = \iota_f^{- 1}$ for $\bar{f} \neq f$, while for $\bar{f} = f$, we have $\theta_f \in H_f$,
$\iota_f \in {\rm Aut}(H_f)$, $\iota_f(\theta_f) = \theta_f$, and $\iota_f^2 = \iota_f \circ \iota_f$ is the 
inner automorphism of $H_f$ given by $h \mapsto \theta_f^{- 1} h \theta_f$.

Next, for each $f \in E$, we choose a left transversal $P_f$ for $H$ modulo $H_f$ with $1 \in P_f$. Consequently,
we obtain a normal form for the elements $g \in G$
\begin{equation}
\label{Eq:NormForm}
g = \rho_g\, \si(C(g))\, \lam_g\quad(\rho_g \in P_{C(g)}, \lam_g \in H).
\end{equation}
On the other hand, for each pair $(f, f') \in E \times E$, we choose a set $Q_{f, f'} \sse H$ of 
pairwise inequivalent representatives of the double cosets $H_f h H_{f'}$ in the space $H_f \sm H/ H_{f'}$ with
$1 \in Q_{f, f'}$, and $Q_{f', f} = Q_{f, f'}^{- 1}$ provided $f \neq f'$. 

Next, we choose an \textit{orientation} on $(E,{}^-)$; that is, a subset $E_+\subseteq E$ such that 
\[
E = E_+ \amalg \big\{\bar{f}: f\in E_+\mbox{ and }f\neq \bar{f}\big\};
\]
in particular, $e=\bar{e}\in E_+$. 

We denote by $F$ the free group with basis $X = \{x_f\,:\,f \in E_+ - \{e\}\}$, in one-to-one correspondence with 
the elements of the set $E_+ - \{e\}$, defining $x_e := 1$ and $x_f := x_{\bar{f}}^{- 1}$ for $f \in E - E_+$. 
Finally, we fix a total order $\leq$ on the set $E - \{e\}$.

With this notation and the above conventions, our basic result concerning group actions on sets is the following.
\begin{te}
\label{Thm:PresGpAct}
The group $G$ is generated by the 
free product $H \ast F$ modulo the type {\em (I)} relations
\begin{equation}
\label{Eq:GenType(I)Rels}
h x_f = x_f \iota_f(h)\quad(f\in E_+-\{e\},\, h\in H_f),
\end{equation}
together with the type {\em (II)} relations
\begin{equation}
\label{Eq:GenType(II)Rels}
x_f h x_{f'} = \rho_g x_{C(g)} \lam_g\quad(f, f' \in E - \{e\}, \bar{f} \leq f', h \in Q_{\bar{f}, f'}, g = \si(f) h \si(f')),
\end{equation}
where $\rho_g$ and $\lambda_g$ are defined as in {\em (\ref{Eq:NormForm})}.
\end{te}
A proof of Theorem~\ref{Thm:PresGpAct} is given in Section~\ref{Sec:ExtPr} in the context of group extensions of polygroups by groups; cf.\ Section~\ref{Subsec:ProofMT}.

\subsection{Multiply transitive actions}
In the context of Theorem~\ref{Thm:PresGpAct}, suppose that the action of $G$ on the set $\Omega$, of cardinality at least $2$,
is $2$-transitive. Then the action of the stabiliser $H:= G_{\om_0}$ on the set $\Omega-\{\omega_0\}$ is transitive,
thus $H \sm \Om = \{\{\om_0\}, \Om - \{\om_0\}\}$ and $H \sm G/ H = \{H, G - H\}$. The polygroup $E = H \sm G/ H \cong H \sm \Om$ of
cardinality $2$ is a group (cyclic of order $2$) if, and only if, $|\Om| = (G:H) = 2$, while, for $|\Om| \geq 3$, $E$ is, 
up to isomorphism, the unique polygroup $\fP_2 = \{0, 1\}$ of cardinality $2$, with the scalar identity $0$, 
$\bar{1} = 1$, and $1 \circ 1 = \{0, 1\}$ (that is, the smallest non-trivial polygroup). In particular, $E_+ = E$ in both cases.

To obtain a section $\si: E \lra G$ of the projection map $G \lra E = H \sm G/ H$, in accordance with the requirements of 
\ref{Subsec:PresGpAct}, we choose a point $\om_1 \in \Om - \{\om_0\}$ and some element $\tau \in G - H$ which interchanges
the points $\om_0$ and $\om_1$ (such an element exists by $2$-transitivity), and set $\si(H) = 1$, $\si(G - H) = \tau$. Thus $G_{\om_1} = \tau H \tau^{- 1}$ is the 
stabilizer of $\om_1$ in $G$, $\theta:= \tau^2 \in H_1:= H \cap \tau H \tau^{- 1}$, and the automorphism $\iota$ of
$H_1$ is given by $h \mapsto \tau^{- 1} h \tau$, in particular, $\iota(\theta) = \theta$. Next, we choose a left
transversal $P \sse H$ for $H$ modulo $H_1$ and a set $Q \sse H$ of pairwise inequivalent representatives of the double 
cosets in the space $H_1 \sm H/ H_1$ with $1 \in P \cap Q$. Note that $|Q| = 1$, that is $H = H_1$, if, and only if,
$|\Om| = 2$; while $|Q| = 2$, that is $H - H_1$ is the unique nontrivial double coset in $H_1 \sm H/ H_1$, if, and only 
if, $|\Om| \geq 3$ and the action of $G$ on the set $\Om$ is $3$-transitive. 

For any $q \in Q$ with $q \neq 1$, the element $\rho_q':= \rho_{\tau q \tau} \in P$ is uniquely determined by the condition that 
$\tau q \tau \in \rho_q' \tau H$, that is $\tau q \om_1 = \rho_q' \om_1$. Set $\lam_q':= \lam_{\tau q \tau} = 
((\rho_q')^{- 1} \tau q)^\tau$ for $q \in Q - \{1\}$. 
                     
Applying Theorem~\ref{Thm:PresGpAct}, we obtain the following.
\begin{co}
\label{Cor:Main2Trans}
Let $G$ be a group acting $2$-transitive from the left on a set $\Omega,$ and let $\omega_0, \omega_1$ be two 
distinct points of $\Omega,$ with stabilisers $H:=G_{\omega_0}$ and $G_{\omega_1} = \tau H \tau^{- 1},$ respectively,
where $\tau \in G$ is an element interchanging $\omega_0$ and $\omega_1$. Let $\{h_i\}_{i\in I}$ be a system of 
generators for the subgroup $H_1:= H \cap \tau H \tau^{- 1}$ of $G$. Then, given the above notation and conventions, 
$G$ is generated by the group $H$ plus one  new generator $x,$ subject to the relations
\begin{align}
h_i\, x &= x\, \iota(h_i)\quad(i\in I),\label{Eq:Gen2TransRel1}\\[1mm]
x^2 &= \theta\,,\label{Eq:Gen2TransRel2}\\[1mm]
x\,q\,x &= \rho_q'\,x\,\lam_q'\quad(q \in Q - \{1\}).\label{Eq:Gen2TransRel3}
\end{align}
\end{co}
Note that Condition (\ref{Eq:Gen2TransRel3}) is empty if $|\Om| = 2$, while it consists of only one identity if
$|\Om| \geq 3$ and the action is $3$-transitive.

\section{$GL_2$ over valuation rings}

\label{Sec:GL2}
Let $K$ be a valued field with valuation $v$. We denote by 
$\mathcal{O}$ the valuation ring with maximal ideal $\fm$, residue field $k = \mathcal{O}/\fm$,
and the totally ordered abelian value group $\Lam = v(K^\times) \cong K^\times/\mathcal{O}^\times$,
where $K^\times = K \sm \{0\}$ is the multiplicative group of $K$, 
$\mathcal{O}^\times =\linebreak \mathcal{O} \sm \fm$ is the multiplicative group of invertible elements of $\mathcal{O}$, and $v(x) \leq v(y) \Llra \mathcal{O} y \sse \mathcal{O} x$ for $x, y \in K$. 

Let $\Lam_+ = \{\alpha \in \Lam:\,\alpha \geq 0\} = v(\O \sm \{0\})$,
$\overline{\Lam}_+ = \Lam_+ \cup \{\infty\}$, with $\infty + \alpha = \infty$ and $\alpha \leq \infty$ for
all $\alpha \in \overline{\Lam}_+$. Note that $\overline{\Lam}_+$ is a totally ordered commutative monoid with
$v(1) = 0$ as the neutral and also least element, while $v(0) = \infty$ is its zero and last element. Given  
 $\alpha \in \overline{\Lam}_+$, fix an element $t_\alpha \in \O$ such that $v(t_\alpha) = \alpha$; thus $t_\infty = 0$,
and we may assume that $t_0 = 1$. 

Let $G = GL_2(\O)$ be the group of all $2\times2$-matrices $A = (a_{i, j})$ over $\O$ which are invertible, that is,
$v({\rm det}(A)) = 0$. We denote by $B = B_2(\O)$ the subgroup of $G$ consisting of the upper-triangular 
matrices, and by $D = D_2(\O)$ the subgroup of diagonal matrices with entries in $\O^\times$. 
For any $a \in \O^\times$, we denote by $Z(a)$ the (scalar) diagonal matrix
$a I_2$, and by $R(a)$ the diagonal matrix $A$ with $a_{1, 1} = a$ and $a_{2, 2} = 1$; in particular, $Z(1) = R(1) = I_2$.
For every $A \in D$, we obtain the decomposition
\begin{equation} 
\label{Eq:DDecomp}
A = Z(a_{2, 2}) R(a_{1, 1} a_{2, 2}^{- 1}) = R(a_{1, 1} a_{2, 2}^{- 1}) Z(a_{2, 2}).
\end{equation}
It follows that the abelian group $D \cong \O^\times \times \O^\times$ is generated by the matrices $Z(a), R(a)$ for
$a \in \O^\times \sm \{1\}$ subject to the defining relations\footnote{Throughout this paper, the commutator $[x,y]$ of two group elements $x,y$ is taken in the form $[x,y] = xyx^{-1}y^{-1}$.}
\begin{align}
R(a) R(b) &= R(ab),\label{Eq:DRel1}\\
Z(a) Z(b) &= Z(ab),\label{Eq:DRel2}\\
[R(a), Z(b)] &= 1,\label{Eq:DRel3}
\end{align}
where $a, b \in \O^\times \sm \{1\}; R(1) = Z(1) = 1$.

In what follows, we shall apply Theorem~\ref{Thm:PresGpAct} to obtain presentations for the groups $B$ and $G$.

\subsection{A presentation for $B_2(\O)$}
\label{SubSec:B}

The group $B = B_2(\O)$ acts transitively from the left on the affine line over the valuation ring $\O$ according to the rule
\begin{equation}
\label{Eq:ActB}
A \cdot z := \frac{a_{1, 1} z + a_{1, 2}}{a_{2, 2}}\,\,{\rm for}\,A \in B \mbox{ and } z \in \O.
\end{equation}
The kernel of the action is the center 
\[
Z = Z_2(\O) = \big\{Z(a)\,|\,a \in \O^\times\big\} \cong \O^\times,
\]
while $B_0 = D \cong \O^\times \times \O^\times$ is the stabilizer of the point $0$. In order to be able to apply Theorem~\ref{Thm:PresGpAct} to the pair
$(B, D)$ we need to describe explicitely the space of double cosets $D \sm B/D$ and its polygroup structure. To this end, we set $S_\infty = I_2$ and, for $\alpha\in\Lambda_+$, 
\[
S_\alpha = 
\begin{pmatrix}-1&t_\alpha\\0&1\end{pmatrix}.
\]
We note that the matrices $S_\alpha$ are involutions; in particular, $S_0 = \begin{pmatrix}-1&1\\0&1\end{pmatrix}$. 
Next, define a hyperoperation on the set $\overline{\Lambda}_+$ via
\[
\alpha \circ \beta = \begin{cases} \{{\rm min}(\alpha, \beta)\},& \alpha \neq \beta,\\[1mm]
                      [\alpha, \infty] = \{\gam \in \overline{\Lam}_+:\,\alpha \leq \gam\},& \alpha = \beta \mbox{ and } {\rm card}(k) \geq 3,\\[1mm]
(\alpha, \infty] = \{\gam \in \Lam_+:\,\alpha < \gam\} \cup \{\infty\},& \alpha = \beta \mbox{ and } k \cong GF(2).
                     \end{cases}  
\]
The resulting polygroup $\mathcal{P}_B(\overline{\Lambda}_+)$ on $\overline{\Lambda}_+$ is clearly commutative, has  scalar identity $\infty$, and its inverse map $\alpha\mapsto \alpha^{-1}$ is the identity on $\overline{\Lambda}_+$. One immediately verifies that, 
for $\alpha\in\Lambda_+$ and any natural number $n\geq2$, 
\begin{equation}
\label{Eq:alpha^n}
\alpha^n := \underbrace{\alpha \circ \cdots \circ \alpha}_{n\,\mathrm{times}} = \begin{cases} \alpha^2 = [\alpha, \infty],& {\rm card}(k) \geq 3,\\[1mm]
\alpha^2 = (\alpha, \infty],& k \cong GF(2)\mbox{ and }2\,|\,n,\\[1mm]
\{\alpha\},& k \cong GF(2)\mbox{ and }2 \not |\,n;                                                        \end{cases}
\end{equation}
in particular, 
\begin{equation}
\label{Eq:0^2}
0^2 := 0 \circ 0 = \begin{cases} \overline{\Lam}_+,& {\rm card}(k) \geq 3,\\[1mm]
                    \overline{\Lam}_+ \sm \{0\},& k \cong GF(2).
                   \end{cases}
 \end{equation}
Given these definitions, we can now state the following.
\begin{lem}
\label{Lem:DBD}
\ben
\item[(i)] The surjective map $B \lra \overline{\Lam}_+$ given by $A\mapsto v(a_{1, 2})$ induces a bijection $\varphi_B: D \sm B/D \lra \overline{\Lam}_+,$
whose inverse sends $\alpha \in \overline{\Lam}_+$ to the double coset $D S_\alpha D$.

\vspace{2mm}
 
\item[(ii)] The bijection $\varphi_B$ yields an isomorphism between the polygroup on the space of double cosets $D \sm B/D$ 
and the commutative polygroup $\mathcal{P}_B(\overline{\Lambda}_+)$.

\vspace{2mm}

\item[(iii)] The double coset $D S_0 D$ is the unique element of the space $D \sm B/D$ which generates its associated hypergroup.
\een 
\end{lem}

\bp
(i) It is straightforward to see that the assignment $A\mapsto v(a_{1,2})$ induces a well-defined map $\varphi_B: D\backslash B/D \rightarrow \overline{\Lambda}_+$. To see that the map $\overline{\Lambda}_+\rightarrow D\backslash G/D$ given by $\alpha\mapsto DS_\alpha D$ is inverse to $\varphi_B$, it suffices to note that every matrix $A \in B \sm D$ admits the normal form
\begin{equation}
\label{Eq:BNormalForm}
A =  \rho_A \cdot \eps(A) \cdot \lam_A,
\end{equation} 
where
\begin{align}
\label{NormFormFactors}
\rho_A = &\, R\big(\frac{a_{1, 2}}{t_\alpha a_{2, 2}}\big) \in D,\\[1mm]
\eps(A) = &\, S_\alpha \in B \sm D,\\[1mm]
\lam_A = &\, R\big(- \frac{t_\alpha a_{1, 1}}{a_{1, 2}}\big) \cdot Z(a_{2, 2}) \in D,\\[1mm]
\alpha =&\, v(a_{1, 2}) \in \Lam_+.
\end{align}

(ii) Let $A, A' \in B\setminus D$. Setting $A'' = A \cdot A'$, we have $a''_{1, 2} = a_{1, 1} a'_{1, 2} + a_{1, 2} a'_{2, 2}$,
therefore 
\begin{align*}
v(a''_{1, 2}) & = {\rm min}(v(a_{1, 2}), v(a'_{1, 2})),\,{\rm if}\,v(a_{1, 2}) \neq v(a'_{1, 2}),\\
v(a''_{1, 2}) & \geq v(a_{1, 2}),\,{\rm if}\,v(a_{1, 2}) = v(a'_{1, 2})\,\,{\rm and}\,\,{\rm card}(k) \geq 3,\\
v(a''_{1, 2}) & > v(a_{1, 2}),\,{\rm if}\,v(a_{1, 2}) = v(a'_{1, 2})\,\,{\rm and}\,\,k \cong GF(2).
\end{align*}
Moreover it follows easily that for any $\alpha \in \Lam_+$ and for any $\beta \geq \alpha$ provided ${\rm card}(k) \geq 3$
and $\beta > \alpha$ provided $k \cong GF(2)$ respectively, there exist matrices $A, A' \in B$ such that $v(a_{1, 2}) = v(a'_{1, 2}) = \alpha,
v(a''_{1, 2}) = \beta$. Thus, the map $\varphi_B: D\backslash B/D\rightarrow \overline{\Lambda}_+$ is an isomorphism of polygroups, as claimed.

(iii) This follows from (\ref{Eq:alpha^n}) and  (\ref{Eq:0^2}).
\ep

{\footnotesize 
\begin{rems} \em
(i) $B$ is a group extension of the commutative polygroup $\mathcal{P}(\overline{\Lambda}_+)$ (defined before 
Lemma~\ref{Lem:DBD}) by the commutative group $D \cong \O^\times \times \O^\times$. Note that the
hyperoperation $\circ$ on $\overline{\Lam}_+$ depends on whether or not the residue field $k$ is isomorphic to 
the field $GF(2)$ of cardinality $2$.

(ii) For any $a \in \O$, let
\[
E(a) := E_{1, 2}(a) = {\begin{pmatrix} 1 & a\\
                        0 & 1
                       \end{pmatrix};}
\]
in particular, $E(0) = I_2$.

Let $a \in \O \sm \{0\}$. Using the normal form (\ref{Eq:BNormalForm}), we obtain the decomposition
\[
E(a) = R\big(\frac{a}{t_\alpha}\big) \cdot S_\alpha \cdot R\big(- \frac{t_\alpha}{a}\big), 
\]
where $\alpha = v(a) \in \Lam_+$. In particular, $E(a) = R(a) \cdot S_0 \cdot R(- a)^{- 1}$ provided $a \in O^\times$, 
while $E(t_\alpha) = S_\alpha \cdot R(- 1)$ for all $\alpha \in \Lam_+$.
\end{rems}}

It remains to apply Theorem~\ref{Thm:PresGpAct} to obtain a presentation for the group $B$ in terms of the stabiliser $B_0 = D$ and the chosen representatives 
$S_\alpha \in B$ for  $\alpha \in \Lam_+$ of the non-trivial double cosets modulo $D$.

To get the relations of type (I), we note that for any $\alpha \in \Lam_+$, the associated partial
automorphism $\iota_\alpha$ of $D$, $x \mapsto S_\alpha  x S_\alpha^{- 1}$ is the identical
automorphism of $D \cap S_\alpha D S_\alpha^{- 1} = Z \cong O^\times$, the center of $B$. Consequently, 
the relations of type (I) are
\begin{equation}
\label{Eq:RelI}
[S_\alpha, Z(a)] = 1,\q (\alpha \in \Lam_+, \,a \in \O^\times \sm \{1\}).
\end{equation}

To obtain the relations of type (II), we note that, by (\ref{Eq:DDecomp}), $\{R(a):\,a \in \O^\times\}$ is a set of representatives for the space 
$Z \sm D/Z = D/Z \cong \O^\times$, so it remains only to use the normal form (\ref{Eq:BNormalForm}) to compute the invariants  
$\rho_A, \eps(A), \lam_A$ for the matrices $A = S_\alpha \cdot R(a) \cdot S_\beta$ with $\alpha, \beta \in \Lam_+,\,\alpha\leq \beta$, and $a \in \O^\times$.
It follows that the relations of type (II) are
\begin{equation}
\label{Eq:RelII1}
S_\alpha^2 = 1\q (\alpha \in \Lam_+),
\end{equation}
\begin{equation}
\label{Eq:RelII2}
S_\alpha \cdot R(a) \cdot S_\alpha = R(\frac{t_\alpha(1 - a)}{t_\gam}) \cdot S_\gam \cdot R(\frac{t_\gam a}{t_\alpha(a - 1)})
\q (\alpha \in \Lam_+, a \in \O^\times \sm \{1\}),
\end{equation}
where $\gam = \alpha + v(a - 1)$, and
\begin{equation}
 \label{Eq:RelII3}
S_\alpha \cdot R(a) \cdot S_\beta = R(\frac{t_\alpha - a t_\beta}{t_\alpha}) \cdot S_\alpha \cdot R(\frac{a t_\alpha}{a t_\beta - t_\alpha})
\q (\alpha, \beta \in \Lam_+,\, \alpha < \beta,\, a \in \O^\times).
\end{equation}

Note that the relations (\ref{Eq:RelII3}) are consequences of the relations (\ref{Eq:DRel1}), (\ref{Eq:RelII1}), 
and (\ref{Eq:RelII2}). Moreover, in accordance with the fact that the double coset $D S_0 D$ generates the 
commutative polygroup with support $D \sm B/D$ (cf.\ Lemma ~\ref{Lem:DBD}(iii)), we also find 
defining relations for $S_\alpha,\,\alpha \in \Lam_+ \sm \{0\}$ in terms of $S_0$ and diagonal matrices:
\begin{equation}
\label{Eq:DefRel}
S_\alpha = S_0 \cdot R(1 - t_\alpha) \cdot S_0 \cdot R(t_\alpha - 1)^{- 1} = [S_0, R(1 - t_\alpha)] \cdot R(- 1)\q
(\alpha \in \Lam_+ \sm \{0\}).
\end{equation}
Consequently, we obtain the following presentation of the group $B = B_2(\O)$.
\begin{pr}
\label{Pr:PresB(O)}
The group $B = B_2(\O)$ of invertible upper-triangular matrices with entries in the valuation ring $\O$ is
generated by the matrices $R(a), Z(a)$ for $a \in \O^\times \sm \{1\}$ and $S:= S_0,$ subject to the relations 
{\em (\ref{Eq:DRel1}), (\ref{Eq:DRel2}), (\ref{Eq:DRel3})}, as well as 
\begin{align}
[S, Z(a)] & = 1\q (a \in \O^\times \sm \{1\}), {\label{Eq:SI}}\\
S^2 & = 1, {\label{Eq:SII1}}\\
[S, R(1 - a t_\alpha)] & = R(a) \cdot [S, R(1 - t_\alpha)] \cdot R(a)^{- 1}\, (0 \neq \alpha \in \Lam_+, 1 \neq a \in \O^\times),
{\label{Eq:SII2}}\\ 
S \cdot R(a) \cdot S & = R(1 - a) \cdot S \cdot R(\frac{a}{a - 1})\q (a \in \mathfrak{k}), {\label{Eq:SII3}}
\end{align}
where $\mathfrak{k} \sse \O^\times \sm (1 + \fm)$ is a set of representatives modulo the maximal ideal $\fm$ of the
elements in $k^\times \sm \{1\}$.  
\end{pr}

\bp
In the presence of (\ref{Eq:DRel1}), (\ref{Eq:DRel2}), and (\ref{Eq:DRel3}), the relations (\ref{Eq:SI}), (\ref{Eq:SII1}),
(\ref{Eq:SII2}), and (\ref{Eq:SII3}) are obvious consequences of the relations (\ref{Eq:RelI}), (\ref{Eq:RelII1}), and
(\ref{Eq:RelII2}). 

Conversely, for any $\alpha \in \Lam_+ \sm \{0\}, a \in \O^\times \sm \{1\}$, we use the defining relation (\ref{Eq:DefRel}) for $S_\alpha$ 
to deduce (\ref{Eq:RelI}) from (\ref{Eq:DRel3}) and (\ref{Eq:SI}), while (\ref{Eq:RelII1}) is a consequence of 
(\ref{Eq:SII1}) and (\ref{Eq:SII2}) taking $a = (t_\alpha - 1)^{- 1}$. 

It remains to deduce (\ref{Eq:RelII2}). Let $\alpha \in \Lam_+, a \in \O^\times \sm \{1\}$. We denote by $\cL, \cR$ 
the left and right side of the identity (\ref{Eq:RelII2}), respectively. We distinguish the following two cases.

(i) $\gam := \alpha + v(a - 1) > 0$. By (\ref{Eq:DefRel}) and (\ref{Eq:SII2}) we obtain
\[
\cR = R\big(\frac{t_\alpha (1 - a)}{t_\gam}\big) \cdot [S, R(1 - t_\gam)] \cdot R\big(\frac{t_\gam a}{t_\alpha (1 - a)}\big) =
[S, R(1 - t_\alpha (1 - a)] \cdot R(a). 
\]
If $\alpha = 0$, so $t_\alpha = 1$, it follows by (\ref{Eq:SII1}) that $\cR = \cL$ as desired. Assuming $\alpha > 0$, we obtain again by 
(\ref{Eq:DefRel}), (\ref{Eq:SII1}), and (\ref{Eq:SII2})
\begin{align*}
\cR & = [S, R(1 - t_\alpha)] \cdot \big(R(1 - t_\alpha) \cdot [S, 1 - \frac{a}{t_\alpha - 1} t_\alpha] \cdot R(1 - t_\alpha)^{- 1}\big)
\cdot R(a)\\[1mm]
& = S_\alpha \cdot R(t_\alpha - 1) \cdot \big(R\big(\frac{a}{t_\alpha - 1}\big) \cdot [S, 1 - t_\alpha] \cdot R\big(\frac{a}{t_\alpha - 1}\big)^{- 1}\big)
\cdot R\big(\frac{a}{1 - t_\alpha}\big)\\[1mm]
& = S_\alpha \cdot R(a) \cdot S_\alpha = \cL
\end{align*}
as desired. 

(ii) $\gam = 0$, that is, $\alpha = v(a - 1) = 0$, so $a\,{\rm mod}\,\fm \in k^\times \sm \{1\}$. Let $b \in \mathfrak{k}$ 
be such that $\beta := v(a - b) > 0$. By (i), (\ref{Eq:DRel1}), (\ref{Eq:SII1}) and (\ref{Eq:SII3}), we obtain
\begin{align*}
S \cdot R(a) \cdot S & = (S \cdot R(b) \cdot S) \cdot (S \cdot R(a b^{- 1}) \cdot S)\\[1mm]
& = \big(R(1 - b) \cdot S \cdot R\big(\frac{b}{b - 1}\big)\big) \cdot \big(R\big(\frac{b - a}{b t_\beta}\big) \cdot S_\beta \cdot R\big(\frac{a t_\beta}{a - b}\big)\big)\\[1mm]
& = R(1 - b) \cdot S \cdot \big(R\big(\frac{1 - \frac{1 - a}{1 - b}}{t_\beta}\big) \cdot S_\beta \cdot R\big(\frac{\frac{1 - a}{1 - b} t_\beta}{\frac{1 - a}{1 - b} - 1}\big)\big) \cdot R\big(\frac{a}{a - 1}\big)\\[1mm]
& = R(1 - b) \cdot S \cdot \big(S \cdot R\big(\frac{1 - a}{1 - b}\big) \cdot S\big) \cdot R\big(\frac{a}{a - 1}\big) \\[1mm]
& = R(1 - a) \cdot S \cdot R\big(\frac{a}{a - 1}\big),
\end{align*}
as required. 
\ep

In particular, if the valuation $v$ is trivial, that is, $\O = K = k$, we obtain the following.

\begin{co}
\label{Cor:PresB(K)}
Let $K$ be an arbitrary field.Then the group $B_2(K)$ of  invertible upper-triangular matrices with entries in $K$ is generated by the matrices
$R(a), Z(a)$ with $a \in K^\times \sm \{1\})$ and the matrix $S,$ subject to the relations {\em (\ref{Eq:DRel1}), (\ref{Eq:DRel2}), (\ref{Eq:DRel3}),}
where $a, b \in K^\times \sm \{1\},$ {\em (\ref{Eq:SI})} with $a \in K^\times \sm \{1\},$ {\em (\ref{Eq:SII1}),} and {\em (\ref{Eq:SII3})}
with $a \in K^\times \sm \{1\}$. 
\end{co}

{\footnotesize 
\begin{rems}
 \label{Rem:B(O)} \em
(i) If the residue field $k$ is isomorphic to $GF(2),$ then $\O^\times = 1 + \fm, \mathfrak{k} = \es$, therefore the 
relation (\ref{Eq:SII3}) does not occur in the presentation of $B_2(\O)$. In particular, 
$B_2(GF(2)) \cong \la S\,|\,S^2 = 1 \ra \cong \Z/2 \Z$, and the polygroup with support $D \sm B/D$ is a group 
(isomorphic to $\Z/2 \Z$) if, and only if, $K \cong GF(2)$.

(ii) Assume that the valuation $v$ is discrete, that is, the totally ordered abelian group $\Lam = v(K^\times)$ has a
smallest positive element $1$. Thus, the ordered group $\Z$ of integers is identified with the smallest proper convex
subgroup of $\Lam$; in particular, $\Lam \cong \Z$ provided the discrete valuation $v$ is of rank $1$. Setting $t:= t_1$,
so $v(t) = 1$, we may replace (\ref{Eq:SII2}) with  
\begin{equation}
\label{Eq:SII2'}
[S, R(1 - at)] = R(a) [S, R(1 - t)] R(a)^{- 1}\quad(a \in O^\times - \{1\}).
\end{equation}
Indeed, using (\ref{Eq:DRel1}) and (\ref{Eq:SII2'}), we obtain for all $1 < \alpha \in \Lam, a \in O^\times$,
\begin{align*}
[S, R(1 - a t_\alpha)] &= [S, R(1 - a t) R(1 - \frac{a(\frac{t_\alpha}{t} - 1)}{1 - a t} t)] \\
&= [S, R(1 - at)] R(1 - at) [S, R(1 - \frac{a (\frac{t_\alpha}{t} - 1)}{1 - a t}t)] R(1 - a t)^{- 1}\\
&= R(a) ( [S, R(1 - t)] R(\frac{t_\alpha}{t} - 1)
[S, R(1 - t)] R(\frac{t_\alpha}{t} - 1)^{- 1}) R(a)^{- 1}\\
&= R(a) [S, R(1 - t_\alpha)] R(a)^{- 1},
\end{align*}
as desired.

(iii) An alternative presentation of $B := B_2(\O)$ can be obtained using the natural semi-direct product decomposition  
$B \cong \O^+ \ltimes (\O^\times \times \O^\times)$ which holds more generally over an arbitrary commutative ring $\O$ with identity element $1$.
It follows that $B$ is generated by the matrices $R(a), Z(a)\, (a \in \O^\times \sm \{1\})$ and $E(a)\,(a \in \O \sm \{0\})$
subject to the relations (\ref{Eq:DRel1}), (\ref{Eq:DRel2}), (\ref{Eq:DRel3}) from the presentation of $D \cong \O^\times
\times \O^\times$ together with the relations 
\[
E(a) \cdot E(b) = E(a + b)\q (a, b \in \O \sm \{0\}); E(0) := 1,
\]
induced by the additive structure of $\O$, and the relations describing the action of $\O^\times \times \O^\times$ on $\O^+$
\[
[Z(a), E(b)] = 1,\, R(a) \cdot E(b) \cdot R(a)^{- 1} = E(a b)\q (a \in \O^\times \sm \{1\}, b \in \O \sm \{0\}). 
\]
The presentation above and the presentation provided by Proposition ~\ref{Pr:PresB(O)} are related through the substitutions
\begin{align*}
S \mapsto & E(1) \cdot R(- 1);\\
E(a) \mapsto & R(a) \cdot S \cdot R(- a)^{- 1}\q (a \in \O^\times),\\
E(a) \mapsto & R(a t_\alpha^{- 1}) \cdot [S, R(1 - t_\alpha)] \cdot R(a^{- 1} t_\alpha)\q (a \in \fm \sm \{0\}, \alpha := v(a)).
\end{align*}

(iv) The relations (\ref{Eq:SII2}) and (\ref{Eq:SII3}) can be written in the form
\[
[S, R(a)] = E(1 - a)\q (a \in \O^\times),
\]
whence the following cohomological interpretation.

The map $[S, R(-)] : \O^\times \lra B, a \mapsto [S, R(a)]$ is a $1$-cocycle
with respect to the action of the multiplicative group $\O^\times$ on the
non-commutative group $B$
\[
(a, A) \in \O^\times \times B \mapsto R(a) A R(a)^{- 1}.
\]
On the other hand, $\O^\times$ acts canonically on the abelian group
$(\O, +), \O^\times \times \O\ni(a, b) \mapsto a b$. Consequently, the relations above express the following two facts.

(a) the map $E : (\O, +) \lra B$ given by $a \mapsto E(a)$ is a monomorphism, and

(b) the $1$-cocycle $[S, R(-)] : \O^\times \lra B$ is obtained by composing 
the $1$-coboundary\\ $\O^\times \lra \O^+$ given by $a \mapsto 1 - a$ with the monomorphism $E : \O^+ \lra B$.
\end{rems}}

\subsection{A presentation for $GL_2(\O)$}
\label{Subsec:GL_2}

The group $G:= GL_2(\O)$ acts transitively from the left on the projective line $\P^1(K) = K \cup \{\infty\}$
over $K$, the fraction field of the valuation ring $\O$, according to the rule
\begin{equation}
\label{Eq:ActG}
A \cdot z := \frac{a_{1, 1} z + a_{1, 2}}{a_{2, 1} z + a_{2, 2}},\quad
(A \in G,\, z \in K \cup \{\infty\}).
\end{equation}
The kernel of the action is the center 
\[
Z = Z_2(\O) = \big\{aI_2:\,a\in\O^\times\big\} \cong \O^\times, 
\]
while $G_\infty = B = B_2(\O)$ is the stabilizer of $\infty$.
In order to be able to apply Theorem~\ref{Thm:PresGpAct} to the pair $(G, B)$, we need to describe explicitely the 
polygroup structure on the space of double
cosets $B \sm G/B$, which is in canonical bijection with the orbit space 
$B \sm \P^1(K) = \bigsqcup_{\alpha \in \overline{\Lam}_+} B \cdot t_\alpha^{- 1}$, where 
\[
B \cdot t_\alpha^{- 1} = \begin{cases} 
                          \O,\,&{\rm if}\,\alpha = 0,\\[1mm]
                          v^{- 1}(- \alpha) = t_\alpha^{- 1} \O^\times,\,&{\rm if}\,\alpha \in \Lam_+ \sm \{0\},\\[1mm]
                         \{\infty\},\,&{\rm if}\,\alpha = v(0) = \infty.
\end{cases}
\]
To this end, set $J_\infty=I_2$, and define
\[
J_\alpha = \begin{cases} {\begin{pmatrix} 0 & - 1\\
                            -1 & 0
                           \end{pmatrix}}\,&{\rm if}\,\alpha = 0,\\
                          {\begin{pmatrix} 1 & 0\\
                             t_\alpha & - 1
 \end{pmatrix}}\,&{\rm if}\,\alpha \in \Lam_+ \sm \{0\},
\end{cases}
\]
noting that the $J_\alpha$ are involutions. Moreover, define a second polygroup structure on the set $\overline{\Lambda}_+$ via
\[
\alpha \circ \beta = \begin{cases} \{{\rm min}\,(\alpha, \beta)\},& \alpha \neq \beta,\\[1mm]
                        [\alpha, \infty] = \{\gam \in \overline{\Lam}_+:\,\alpha \leq \gam\},& \alpha = \beta.
                       \end{cases}
\]
The resulting polygroup $\mathcal{P}_G(\overline{\Lambda}_+)$ is again commutative, has neutral element (scalar identity) $\infty$, and satisfies $\alpha^{-1} = \alpha$ for all $\alpha\in\overline{\Lambda}_+$. With these preliminaries, we can now state the following.

\begin{lem}
\label{Lem:BGB}
\ben
\item[(i)] The surjective map $G \lra \overline{\Lam}_+$ given by $A \mapsto v(a_{2, 1})$ induces a bijection
$\varphi_G: B \sm G/B \lra \overline{\Lam}_+,$ whose inverse sends $\alpha \in \overline{\Lam}_+$ to the double coset $B J_\alpha B$.

\vspace{2mm}
  
\item[(ii)] The bijection $\varphi_G$ of Part~{\em (i)}  yields an isomorphism between the polygroup on the space of double cosets $B \sm G/B$ and the commutative polygroup
 $\mathcal{P}_G(\overline{\Lambda}_+)$. 

\vspace{2mm}

\item[(iii)] We have $0^2=0 \circ 0 = \overline{\Lam}_+,$ so that the double coset $B J_0 B$ is the unique element of the space $B \sm G/B$ which generates
its associated hypergroup.
\een
\end{lem}

\bp
(i) As in the proof of Lemma ~\ref{Lem:DBD}, it suffices to note that every matrix $A \in G \sm B$ admits the normal form
\begin{equation}
\label{Eq:GNormalForm}
A = \rho_A \cdot \eps(A) \cdot \lam_A
\end{equation}
with
\begin{align}
 \label{NormalForm}
\rho_A = &\begin{cases} E_{1, 2}(\frac{a_{1, 1}}{a_{2, 1}})\,&{\rm if}\,\alpha = 0,\\[1mm]
             R(\frac{a_{1, 1} t_\alpha}{a_{2, 1}})\,&{\rm if}\,\alpha > 0;
            \end{cases}\\[3mm]
\eps(A) = &J_\alpha;\\[3mm]
\lam_A = &\begin{cases} 
\begin{pmatrix} - a_{2, 1} & - a_{2, 2}\\
  0 & \frac{\Delta}{a_{2, 1}}
 \end{pmatrix}\,&{\rm if}\,\alpha = 0,\\[6mm]
\begin{pmatrix} \frac{a_{2, 1}}{t_\alpha} & \frac{a_{1, 2} a_{2, 1}}{a_{1, 1} t_\alpha}\\
  0 & - \frac{\Delta}{a_{1, 1}}
 \end{pmatrix}\,&{\rm if}\,\alpha > 0,
\end{cases}
\end{align}
where $\alpha = v(a_{2, 1}) \in \Lam_+$ and $\Delta = \,{\rm det}\,(A)$. 

The proof of (ii) is similar to the proof of Lemma ~\ref{Lem:DBD}(ii), while (iii) is immediate from the definition of the hypergroup operation. 
\ep

{\footnotesize 
\begin{rem} \em
The group $G$ is a group extension  of the commutative polygroup $\mathcal{P}_G(\overline{\Lambda}_+)$ by the group $B$. Note that this polygroup is the same as the polygroup $\mathcal{P}_B(\overline{\Lambda}_+)$ from Lemma ~\ref{Lem:DBD}
corresponding to the case where ${\rm card}(k) \geq 3$. 
\end{rem}}

It remains to apply Theorem~\ref{Thm:PresGpAct} to obtain a presentation for the group $G = GL_2(\O)$ in terms of the 
stabiliser $G_\infty=B$ and the chosen representatives $J_\alpha \in G \sm B\, (\alpha \in \Lam_+)$ of the non-trivial 
double cosets modulo $B$. 

To get the relations of type (I), we look at the associated partial automorphisms $\iota_\alpha$ of $B$, 
$A \mapsto J_\alpha \cdot A \cdot J_\alpha^{- 1}\,(\alpha \in \Lam_+)$ inducing automorphisms of
the subgroups $B_\alpha := B \cap J_\alpha B J_\alpha^{- 1}\,(\alpha \in \Lam_+)$, respectively.

For $\alpha = 0$, we obtain $B_0 = D \cong \O^\times \times \O^\times$, and the automorphism $\iota_0$ of $D$ is
defined by $\iota_0(Z(a)) = Z(a), \iota_0(R(a)) = R(a)^{- 1} \cdot Z(a)$, for $a \in \O^\times$. Consequently, we
obtain the relations
\begin{equation}
\label{Eq:GRelI}
(J_0 \cdot R(a))^2 = Z(a)\q(a \in \O^\times).
\end{equation}
Note that the relation $J_0^2 = 1$ is included in (\ref{Eq:GRelI}) for $a = 1$. Note also that the relations
$(J_0 \cdot R(a))^2 = (R(a) \cdot J_0)^2\,(a \in \O^\times)$, and hence the commuting relations $[J_0, Z(a)] = 1\,(a \in \O^\times)$,
are consequences of (\ref{Eq:GRelI}).

For $\alpha > 0$, we obtain $B_\alpha \cong Z \times R_\alpha \cong \O^\times \times (1 + t_\alpha \O)$, where $R_\alpha$
is the abelian subgroup of $B$ consisting of the matrices
\[
R_\alpha(u) := \begin{pmatrix} u & \frac{1 - u}{t_\alpha}\\[1mm]
                0 & 1
               \end{pmatrix}\q (v(u - 1) \geq \alpha > 0), 
\]
generated by matrices $R_\alpha(u)$ with $v(u - 1) = \alpha$.
The automorphism $\iota_\alpha$ of $B_\alpha$ is defined by 
\[
\iota_\alpha(Z(a)) = Z(a)\,(a \in \O^\times),\, \iota_\alpha(R_\alpha(u)) = R_\alpha(u)^{- 1} \cdot Z(u)\q (v(u - 1) \geq \alpha > 0).   
\]
Consequently, we add to (\ref{Eq:GRelI}) the relations of type (I) for $\alpha > 0$
\begin{equation}
\label{Eq:GRelI'} 
(J_\alpha \cdot R_\alpha(u))^2 = Z(u)\q (v(u - 1) = \alpha > 0),
\end{equation}
and
\begin{equation}
\label{Eq:GRelI''}
[J_\alpha, Z(u)] = 1\q (u \in \O^\times, v(u - 1) < \alpha).
\end{equation}
Note that, as in the case $\alpha = 0$, the relations $J_\alpha^2 = 1$ and $[J_\alpha, Z(u)] = 1$ for $v(u - 1) \geq \alpha > 0$
are consequences of (\ref{Eq:GRelI'}).

To obtain the relations of type (II), we have to choose representatives $M \in B$ for the double cosets of 
$B_\gam \sm B /B_\delta\, (\gam, \delta \in \Lam_+)$, and use the normal form (\ref{Eq:GNormalForm}) for matrices of
the form $J_\gam \cdot M \cdot J_\delta$. First assume $\gam = \delta = 0$, so $B_0 = D$, and put $J := J_0$. It follows by 
Lemma ~\ref{Lem:DBD} that the matrices $S_\alpha\,(\alpha \in \overline{\Lam}_+)$ are representatives for the
double cosets in $D \sm B /D$. As $S_\infty = I_2$, we obtain again the relation $J^2 = 1$, while for $\alpha = 0$, 
setting $S := S_0$, we obtain the relation of type (II) : $J \cdot S \cdot J = S \cdot J \cdot S$. Since $S$ and $J$ are
involutions, it follows that
\begin{equation}
\label{Eq:GRelII}
(J \cdot S)^3 = 1. 
\end{equation}

Next, we show that, in agreement with the fact that the double coset $B \cdot J \cdot B$ generates the hypergroup
on the space $B \sm G / B$ (cf.\ Lemma ~\ref{Lem:BGB}), (\ref{Eq:GRelII}) is essentially the only relation of type (II). 
First note that, for $\alpha > 0$, we get the defining relation for $J_\alpha$
\begin{equation}
\label{Eq:Defrel}
J_\alpha = J \cdot S_\alpha \cdot J\q (\alpha \in \Lam_+ \sm \{0\}). 
\end{equation}
In particular, (\ref{Eq:GRelI''}) is a consequence of (\ref{Eq:RelI}) and (\ref{Eq:GRelI}).

Next note that, for all $\alpha \in \Lam_+ \sm \{0\}$, the pair 
\[
(D/Z \cong \O^\times, B_\alpha/Z \cong R_\alpha
\cong (1 + t_\alpha \O)) 
\]
of subgroups of $B/Z \cong \O \ltimes \O^\times$ is a factorization of $B/Z$; that is,
$(D/Z) \cap (B_\alpha/Z) = 1$, and $B/Z = (D/Z) \cdot (B_\alpha/Z)$. Indeed, for all $a \in \O^\times, b \in \O$, 
the matrix
\[
M_{b, a} := \begin{pmatrix} a & b\\
             0 & 1
            \end{pmatrix}
\]
admits the canonical decompositions
\[
M_{b, a} = R_\alpha(1 - b t_\alpha) \cdot R\big(\frac{a}{1 - b t_\alpha}\big) = R(a + b t_\alpha) \cdot R_\alpha\big(1 - \frac{b t_\alpha}{a + b t_\alpha}\big).
\]
Consequently, the space $B_0 \sm B / B_\alpha$ is a singleton for all $\alpha \in \Lam_+ \sm \{0\}$, with $I_2$ as
representative, whence we obtain the relation $J \cdot J_\alpha = S_\alpha \cdot J$, which is equivalent to (\ref{Eq:Defrel}).
On the other hand, for $\alpha, \beta \in \Lam_+ \sm \{0\}$, the map $\O^\times \lra B_\alpha \sm B / B_\beta$ given by $a \mapsto
B_\alpha \cdot R(a) \cdot B_\beta$ is surjective, inducing a bijection of the space of double cosets $B_\alpha \sm B / B_\beta$
onto the orbit space of the set $\O^\times$ under the action 
\[
(1 + t_\alpha \O) \times \O^\times \lra \O^\times, (u, a) \mapsto u a + \frac{1 - u}{t_\alpha} \cdot t_\beta. 
\]
The action above is free for $\alpha \neq \beta$, while, for $\alpha = \beta$, $1$ is the unique fixed point, and
the action on $\O^\times \sm \{1\}$ is also free.

Let $Q_{\alpha, \beta} \sse \O^\times$ be such that $\{R(a):\,a \in Q_{\alpha, \beta}\}$ is a set of representatives
for the double cosets in $B_\alpha \sm B / B_\beta$ for $\alpha, \beta \in \Lam_+ \sm \{0\}$. The corresponding
relations of type (II)
\begin{equation}
J_\alpha \cdot R(a) \cdot J_\beta = R\big(\frac{a t_\gam}{a t_\alpha - t_\beta}\big) \cdot J_\gam \cdot R\big(\frac{t_\beta - a t_\alpha}
{t_\gam}\big) \cdot Z(- 1),
\end{equation}
where 
\[
\alpha, \beta \in \Lam_+ \sm \{0\}, a \in Q_{\alpha, \beta}, a \neq 1\,{\rm if}\,\alpha = \beta,\,{\rm and}\,
\gam = \begin{cases} \alpha + v(a - 1) & {\rm if}\,\alpha = \beta\\
        {\rm min}\,(\alpha, \beta) & {\rm if}\,\alpha \neq \beta,
       \end{cases}
\]
are consequences of the relations above, in particular of (\ref{Eq:Defrel}), (\ref{Eq:GRelI}), (\ref{Eq:RelII2}),
(\ref{Eq:RelII3}). Indeed, we obtain
\begin{align*}
J_\alpha \cdot R(a) \cdot J_\beta &= (J \cdot S_\alpha \cdot J) \cdot R(a) \cdot (J \cdot S_\beta \cdot J) =
J \cdot S_\alpha \cdot (J \cdot R(a) \cdot J) \cdot S_\beta \cdot J\\[1mm]
&= J \cdot S_\alpha \cdot (R(a)^{- 1} \cdot Z(a)) \cdot S_\beta \cdot J =
J \cdot (S_\alpha \cdot R(a^{- 1}) \cdot S_\beta) \cdot J \cdot Z(a)\\[1mm]
&= J \cdot \big(R\big(\frac{t_\alpha - a^{- 1} t_\beta}{t_\gam}\big) 
\cdot S_\gam \cdot R\big(\frac{a^{- 1} t_\gam}{a^{- 1} t_\beta - t_\alpha}\big)\big) \cdot J \cdot Z(a)\\[1mm]
&= \big(J \cdot R\big(\frac{a t_\alpha - t_\beta}{a t_\gam}\big) \cdot J\big) \cdot (J \cdot S_\gam \cdot J) \cdot \big(J \cdot R\big(\frac{t_\gam}
{t_\beta - a t_\alpha}\big) \cdot J\big) \cdot Z(a)\\[1mm]
&= \big(R\big(\frac{a t_\gam}{a t_\alpha - t_\beta}\big) \cdot Z\big(\frac{a t_\alpha - t_\beta}{a t_\gam}\big)\big) \cdot J_\gam \cdot 
\big(R\big(\frac{t_\beta - a t_\alpha}{t_\gam}\big) \cdot Z\big(\frac{t_\gam}{t_\beta - a t_\alpha}\big)\big) \cdot Z(a)\\[1mm]
&= R\big(\frac{a t_\gam}{a t_\alpha - t_\beta}\big) \cdot J_\gam \cdot R\big(\frac{t_\beta - a t_\alpha}{t_\gam}\big) \cdot Z(- 1),
\end{align*}
as desired.

Moreover, the relations (\ref{Eq:GRelI'}) of type (I) for $\alpha > 0$ are also consequences of the relations 
(\ref{Eq:DRel1}), (\ref{Eq:DRel2}), (\ref{Eq:DRel3}), (\ref{Eq:RelI}), (\ref{Eq:RelII2}), (\ref{Eq:RelII3}), 
(\ref{Eq:GRelI}), (\ref{Eq:GRelII}). Indeed, let $\alpha \in \Lam_+ \sm \{0\},\,u \in \O^\times$ be such that 
$v(u - 1) = \alpha$. 
Then we obtain
\begin{align*}
&J_\alpha R_\alpha(u) J_\alpha = (J S_\alpha) \cdot J \cdot \big(R\big(\frac{1 - u}{t_\alpha}\big) S R\big(\frac{t_\alpha u}{u - 1}\big)\big) \cdot
J \cdot (S_\alpha J) \\[1mm]
&= J \cdot \big(S_\alpha R\big(\frac{t_\alpha}{1 - u}\big) S\big) \cdot J \cdot \big(S R\big(\frac{u - 1}{t_\alpha u}\big) S_\alpha\big) \cdot J 
\cdot Z(- u)\\[1mm]
&= J \cdot \big(R\big(\frac{t_\alpha u}{u - 1}\big) S R(u)^{- 1}\big) \cdot J \cdot \big(R(u)^{- 1} S R\big(\frac{1 - u}{t_\alpha}\big)\big) \cdot J
\cdot Z(- u)\\[1mm]
&= J \cdot R\big(\frac{t_\alpha u}{u - 1}\big) \cdot (S \cdot (R(u)^{- 1} J R(u)^{- 1}) \cdot S) \cdot R\big(\frac{1 - u}{t_\alpha}\big) 
\cdot J \cdot Z(- u)\\[1mm]
&= J \cdot R\big(\frac{t_\alpha u}{u - 1}\big) \cdot (J S J) \cdot R\big(\frac{1 - u}{t_\alpha}\big) \cdot J \cdot Z(- u)\\[1mm]
&= \big(R\big(\frac{u - 1}{t_\alpha u}\big) S R\big(\frac{t_\alpha}{1 - u}\big)\big) \cdot Z(u) = R_\alpha(u)^{- 1} Z(u), 
\end{align*}
whence $(J_\alpha R_\alpha(u))^2 = Z(u)$ as desired.

Summarising the discussion above, we obtain the following presentation for the group $G:= GL_2(\O)$ over a valuation 
ring $\O$.

\begin{te}
\label{Th:PresGL_2(O)}
Let $\O$ be a valuation ring. Then the group $GL_2(\O)$ is generated by the matrices
\[
R(a) = \begin{pmatrix} a & 0\\
0 & 1   
       \end{pmatrix}\,\,(a \in \O^\times \sm \{1\}),\,\,\,S = S_0 = \begin{pmatrix} - 1 & 1\\
0 & 1
\end{pmatrix},\mbox{ and }\,J = J_0 = \begin{pmatrix} 0 & - 1\\
- 1 & 0
\end{pmatrix},
\]
subject to the relations {\em (\ref{Eq:DRel1}), (\ref{Eq:DRel2}), (\ref{Eq:DRel3}), (\ref{Eq:SI}), (\ref{Eq:SII1}),
(\ref{Eq:SII2}), (\ref{Eq:SII3}),} where 
\[
Z(a):= (J \cdot R(a))^2\q (a \in \O^\times \sm \{1\}), 
\]
together with the relations $J^2 = 1$ and $(J \cdot S)^3 = 1$. 
\end{te}
In particular, if the valuation $v$ is discrete, with $v(t) = 1$, then, according to Remarks~\ref{Rem:B(O)}(ii), 
(\ref{Eq:SII2}) can be replaced with (\ref{Eq:SII2'}), while if $v$ is trivial, that is, $\O = K = k$, we obtain

\begin{co}
\label{Cor:GL_2(K)}
Let $K$ be a field. Then the group $GL_2(K)$ is generated by the matrices $R(a)\,\,(a \in K^\times \sm \{1\}),\,\, S,$
and $J,$ subject to the relations {\em (\ref{Eq:DRel1}), (\ref{Eq:DRel2}), (\ref{Eq:DRel3}),} where $a, b \in K^\times \sm \{1\},$
{\em (\ref{Eq:SI})} with $a \in K^\times \sm \{1\},$ {\em (\ref{Eq:SII1}), (\ref{Eq:SII3})} with $a \in K^\times \sm \{1\},$
where $Z(a):= (J \cdot R(a))^2\, (a \in K^\times \sm \{1\}),$ together with the relations 
\begin{equation}
\label{Eq:J,JSOrder}
J^2 = (J \cdot S)^3 = 1.
\end{equation}
\end{co}

In particular, if $K = GF(q)$ is a finite field of cardinality $q = p^f,\,p$ a prime number, $f \geq 1$, we obtain
the following presentation of $GL_2(q)$.

\begin{co}
\label{Co:GL_2(q)}
Let $\zeta$ be a generator of the cyclic multiplicative group $K^\times$ of order $q - 1$. Then
$GL_2(q)$ is generated by the matrices
\[r:= R(\zeta) = \begin{pmatrix} \zeta & 0\\
                    0 & 1
                   \end{pmatrix},\, S:= \begin{pmatrix} - 1 & 1\\
0 & 1 \end{pmatrix},\, J:= \begin{pmatrix} 0 & - 1\\
- 1 & 0 \end{pmatrix}, 
\]
subject to the relations
\begin{align}
r^{q - 1} = S^2 = J^2 = (J \cdot S)^3 = 1,\\
z^{q - 1} = [r, z] = [S, z] = 1,\,{\rm where}\,\,z:= (J \cdot r)^2,\\
S \cdot r^i \cdot S = r^{\gam(i)} \cdot S \cdot r^{- \gam(q - i - 1)}\q (i = 1, \dots, q - 2),
\end{align}
with $\gam(i) \in \{1, \dots, q - 2\}$ defined by $\zeta^{\gam(i)} = 1 - \zeta^i\,\,(i = 1, \dots, q - 2)$.
\end{co}

Using the semi-direct product $B_2(K) \cong K^+ \ltimes (K^\times \times K^\times)$, we obtain the following
alternative presentation of $GL_2(q)$, which can be also deduced from a more general result concerning the
deformations of $GL_2(q)$.

\begin{co}
\label{Co:GL_2(q)'}
With the notation above, let 
\[
P(X):= X^f - \sum_{0 \leq i \leq f - 1}\,a_i X^i \in GF(p)[X]
\] 
be the minimal polynomial of the primitive element $\zeta$ over the prime field $GF(p)$. Then the group $GL_2(q)$ is 
generated by the matrices
\[
r = \begin{pmatrix} \zeta & 0\\
         0 & 1
        \end{pmatrix},\,\, e:= E_{1, 2}(1) = \begin{pmatrix} 1 & 1\\
0 & 1 \end{pmatrix},\,\, J = \begin{pmatrix} 0 & - 1\\
- 1 & 0                          
                          \end{pmatrix},
\]
subject to the relations
\begin{align}
r^{q - 1} = e^p = J^2 = 1,\\
(J \cdot S)^3 = 1,\,{\rm where}\,\,S:= \begin{cases}
e & {\rm if}\, p = 2,\\
e \cdot r^{\frac{q - 1}{2}} & {\rm if}\, p \neq 2,
\end{cases}\\
z^{q - 1} = [r, z] = [e, z] = 1,\,\,{\rm where}\,\,z:= (J \cdot r)^2,\\
[e, e_i] = 1\,\,(i = 1, \dots, f - 1),\,e_f = \prod_{0 \leq i \leq f - 1}\,e_i^{a_i},
\end{align}
where $e_i:= r^i \cdot e \cdot r^{- i}\,\,(i = 0, \dots, f)$.
\end{co}

\subsection{Presentations for $PGL_2(\O)$ and $PGL_2(K)$}

The action (\ref{Eq:ActG}) of $GL_2(\O)$ induces a transitive faithful action of $PGL_2(\O) = GL_2(\O)/Z_2(\O)$ on the
projective line $\P^1(K) = K \cup \{\infty\}$ over $K$, the fraction field of the valuation ring $\O$.
Using the presentation of $GL_2(\O)$ provided by Theorem ~\ref{Th:PresGL_2(O)}, we obtain the following presentation
of $PGL_2(\O)$. 

\begin{co}
\label{Co:PresPGL_2(O)} 
The group $PGL_2(\O)$ is generated by the transformations 
\[
\wt{R}(a): z \mapsto a z\,\,(a \in \O^\times \sm \{1\});\,\,\wt{S}: z \mapsto 1 - z;\,\,\wt{J}: z \mapsto \frac{1}{z} 
\]
subject to the relations
\begin{align}
\wt{R}(a b) = \wt{R}(a) \circ \wt{R}(b)\,\, (a, b \in \O^\times \sm \{1\};\,\, \wt{R}(1) := 1,\label{1}\\[1mm]
\wt{S}^2 = \wt{J}^2 = (\wt{J} \circ \wt{S})^3 = 1,\label{2}\\[1mm]
[\wt{S}, \wt{R}(1 - a t_\alpha)] = \wt{R}(a) \circ \wt{S}_\alpha \circ \wt{R}(- a)^{- 1}\,\,
(\alpha \in \Lam_+ \sm \{0\},\, a \in \O^\times \sm \{1\}),\\[1mm]
\wt{S} \circ \wt{R}(a) \circ \wt{S} = \wt{R}(1 - a) \circ \wt{S} \circ \wt{R}\big(\frac{a}{a - 1}\big)\,\,
(a \in \mathfrak{k}),\label{3}\\[1mm]
(\wt{J} \circ \wt{R}(a))^2 = 1\,\,(a \in \O^\times \sm \{1\})\label{4},
\end{align}
where $\mathfrak{k} \sse \O^\times \sm (1 + \fm)$ is a set of representatives modulo the maximal ideal $\fm$ 
for the elements in $k^\times \sm \{1\},$ and $\wt{S}_\alpha:= [\wt{S}, \wt{R}(1 - t_\alpha)] \circ \wt{R}(- 1)$. 
\end{co}
In particular, if the valuation $v$ is trivial, we obtain the following presentation of $PGL_2(K)$.
\begin{co}
\label{Co:PresPGL_2(K)}
The group $PGL_2(K)$ is generated by $\wt{R}(a)\,(a \in K^\times \sm \{1\},\,\,\wt{S},$ and $\wt{J},$ subject to the
relations {\em (\ref{1})\, ($a, b \in K^\times \sm \{1\}),$ (\ref{2}), and (\ref{3}, \ref{4})} with  
$a \in K^\times \sm \{1\}$.   
\end{co}
Note that $PGL_2(\O)$ and $PGL_2(K)$ are suitable quotients of the free products $\O^\times \ast \mathrm{Sym}_3$ and
$K^\times \ast \mathrm{Sym}_3$, respectively.

We end this subsection with two presentations of the group $PGL_2(q)$ over the finite field $GF(q)$ derived from
Corollaries ~\ref{Co:GL_2(q)} and \ref{Co:GL_2(q)'}.

\begin{co}
\label{PresPGL_2(q)}
The group $PGL_2(q)$ is generated by $\wt{r}:= \wt{R}(\zeta),\, \wt{S},$ and $\wt{J},$ subject to the relations
\begin{align}
\wt{r}^{q - 1} = \wt{S}^2 = \wt{J}^2 = (\wt{J} \circ \wt{S})^3 = (\wt{J} \circ \wt{r})^2 = 1,\\[1mm]
\wt{S} \circ \wt{r}^i \circ \wt{S} = \wt{r}^{\gam(i)} \circ \wt{S} \circ \wt{r}^{- \gam(q - i - 1)}\,\,(i = 1, \dots, q - 2),
\end{align}
with $\gam(i) \in \{1, \dots, q - 2\}$ defined by $\zeta^{\gam(i)} = 1 - \zeta^i\,\,(i = 1, \dots, q - 2)$.
\end{co}

\begin{co}
\label{Co:PresPGL_2(q)'}
The group $PGL_2(q)$ is generated by $\wt{r},\,\,\wt{e}: z \mapsto z + 1,$ and $\wt{J},$ subject to the relations
\begin{align}
\wt{r}^{q - 1} = \wt{e}^p = \wt{J}^2 = (\wt{J} \circ \wt{r})^2 = (\wt{J} \circ \wt{S})^3 = 1,\\[1mm]
[\wt{e}, \wt{e}_i] = 1\,(i = 1, \dots, q - 2),\, \wt{e}_f = \prod_{0 \leq i \leq f - 1}\,\wt{e}_i^{\,a_i},
\end{align}
where
\[
\wt{e}_i:= \wt{r}^i \circ \wt{e} \circ \wt{r}^{- i}\,\,(i = 0, \dots, f) \mbox{ and }
\wt{S}:= \begin{cases} \wt{e} & {\rm if}\,p = 2,\\
          \wt{e} \circ \wt{r}^{\frac{q - 1}{2}} & {\rm if}\,p \neq 2.
         \end{cases}
\]
\end{co}

\section{$GL_2$ over valued fields}
\label{Sec:GL2/VF}

Let $(K, v)$ be a non-trivial valued field, with $\O \neq K, \fm, k = \O/\fm, \Lam = v(K^\times),$ as in Section~\ref{Sec:GL2}. 
We apply Theorem~\ref{Thm:PresGpAct} to obtain a presentation of the group $GL_2(K)$ using its transitive action on a suitable $\Lam$-tree seen as a deformation of the canonical action of $GL_2(K)$ on the projective line $\P^1(K)$,
induced by the valuation $v$.\footnote{A similar approach also works for the group $SL_2$ over a valued field.}

For any $\alpha \in \Lam$, let $K/\alpha:= K/\O_\alpha$ be the factor $\O$-module of $K$ by its submodule 
$\O_\alpha:= \{z \in K\,:\,v(z) \geq \alpha\}$; in other words, for $z_1, z_2 \in K$, we have 
\[
z_1 \equiv z_2\,\,\,{\rm mod}\,\,\,
\alpha \Llra v(z_1 - z_2) \geq \alpha. 
\]
We denote by $[z]_\alpha$ the elements of $K/\alpha$,  noting that 
$[0]_\alpha = [1]_\alpha$ provided $\alpha \leq 0$. Let $\fX:= \bigsqcup_{\alpha \in \Lam} K/\alpha$ be the disjoint
union of the $\O$-modules $K/\alpha$. 
As a global residue structure of the valued field $(K, v)$, $\fX$ has a rich arithmetic and geometric structure, which
can be interpreted as a deformation of the original structure of the valued field $(K, v)$. In particular, 
$\fX$ becomes a $\Lam$-tree with respect to the distance $d: \fX \times \fX \lra \Lam_+$ given by
\[
d([x]_\alpha, [y]_\beta):= |\alpha - \beta| + 2\,\mbox{max}\,(0,\,\mbox{min}\,(\alpha, \beta) - v(x - y)). 
\]
Note that the above construction of the residue structure $\fX$ works in the more general case when $\O$ is a Pr\" ufer domain, and $K$ is its field of fractions, where the totally ordered abelian group $\Lam$ is replaced by the abelian lattice-ordered group of non-zero fractional ideals of finite type of $K$. Still more generally, it works for the larger category of Pr\" ufer ring extensions; see \cite{A-B}, \cite{Prufer}, \cite{Axiomatic}, and \cite{MS} for more details.

To define the desired transitive action of $GL_2(K)$ on $\fX$, we need the following, obvious but useful, lemma.
\begin{lem}
\label{Lem:Beta}
Let $s = \begin{pmatrix} a&b\\c&d\end{pmatrix} \in GL_2(K),$ $\Delta = {\rm det}\,(s), \alpha \in \Lam, z \in K$.
Then the following assertions hold.
\ben
\item[\rm (i)] $\delta:= {\rm min}\,(\alpha + v(c), v(c z + d)) \in \Lam$ does not depend on the choice of the 
representative $z \in [z]_\alpha$.

\vspace{2mm}

\item[\rm (ii)] $\delta = \begin{cases} v(c z + d) & {\rm if}\, [z]_\alpha \neq [- \frac{d}{c}]_\alpha\,,\\[1mm]
                          \alpha + v(c) & {\rm if}\, [z]_\alpha = [- \frac{d}{c}]_\alpha.
                         \end{cases}$

\vspace{2mm}

In particular, $\delta = v(d)$, provided $c = 0$.

\vspace{2mm}

\item[\rm (iii)] Let $\beta:= v(\Delta) + \alpha - 2 \delta$. Then
\[
\beta = \begin{cases} v(\Delta) + \alpha - 2 v(c z + d), & {\rm if}\,[z]_\alpha \neq [- \frac{d}{c}]_\alpha\,,\\[1mm]
         v(\Delta) - \alpha - 2 v(c) & {\rm if}\,[z]_\alpha = [- \frac{d}{c}]_\alpha\,.
        \end{cases}
\]

\vspace{2mm}

In particular, $\beta = \alpha + v(a) - v(d),$ provided $c = 0$.
\een
\end{lem}
Consider the map $GL_2(K) \times \fX \lra \fX,\,(s, [z]_\alpha) \mapsto s \cdot [z]_\alpha$, given by
\begin{equation}
\label{Eq:ActionOnX}
\begin{pmatrix}a&b\\
 c&d
\end{pmatrix}
 \cdot [z]_\alpha:= \begin{cases} [\frac{a z + b}{c z + d}]_\beta\,\mbox{for some (for all)}\,z \in [z]_\alpha, & 
\mbox{if}\,[z]_\alpha \neq [- \frac{d}{c}]_\alpha\,,\\[2mm]
[\frac{a}{c}]_\beta, & \mbox{if}\,[z]_\alpha = [- \frac{d}{c}]_\alpha\,,
                     \end{cases}
\end{equation}
with $\beta \in \Lam$ as defined in Part~(iii) of Lemma~\ref{Lem:Beta}.

{\footnotesize 
\begin{rem} \em
Let $s = \begin{pmatrix}a&b\\
          c&d
         \end{pmatrix} \in GL_2(K), c \neq 0, \Delta:= {\rm det}\,(s), \alpha \in \Lam$. Then
\begin{align*}
A&:= \left\{\frac{a z + b}{c z + d}\,:\,[z]_\alpha = \big[- \frac{d}{c}\big]_\alpha,\,z \neq - \frac{d}{c}\right\}\\[1mm]
&= \left\{\frac{a}{c} - \frac{\Delta}{\lam c^2}\,:\,\lam \in K^\times,\, v(\lam) \geq \alpha\right\}\\[1mm]
&= \left\{x \in K\,:\,v\big(x - \frac{a}{c}\big) \leq \beta:= v(\Delta) - 2 v(c) - \alpha\right\}\,,
\end{align*}
and hence 
\[
A \cap \big[\frac{a}{c}\big]_\beta = \Big\{x \in K\,:\,v\big(x - \frac{a}{c}\big) = \beta\Big\}\, \neq\, \es, 
\]
justifying the
definition $s \cdot [- \frac{d}{c}]_\alpha = [\frac{a}{c}]_\beta$. 
\end{rem}}

For any $s = \begin{pmatrix}a&b\\
              c&d
             \end{pmatrix} \in GL_2(K)$ with $\Delta:= {\rm det}\,(s),$ set $|s|:= v(\Delta) - 2 v(s) \in \Lam_+$, 
where $v(s):= {\rm min}\,(v(a), v(b), v(c), v(d))$.
\begin{pr}
\label{Pr:ActionOnX}
$GL_2(K)$ acts transitively by $\Lam$-isometries on the $\Lam$-tree $\fX,$ according to the rule 
{\em (\ref{Eq:ActionOnX}),} and the following assertions hold.
\ben
\item[\rm (1)] The kernel of the action is $Z:= \{a I_2\,:\,a \in K^\times\} \cong K^\times;$ thus, $PGL_2(K)$ is 
identified with a group of automorphisms of the $\Lam$-tree $\fX$.

\vspace{2mm}

\item[\rm (2)] The stabilizer of the point $[0]_\O = [1]_\O$ is 
\[
H:= \big\{s \in GL_2(K)\,:\,|s| = 0\big\} = GL_2(\O) \cdot Z;
\] 
hence, the surjective map $GL_2(K) \lra \fX,\,s \mapsto 
s \cdot [0]_\O,$ identifies $\fX$ with the quotient of $GL_2(K)$ by the equivalence relation 
\[
s \sim t \Llra |s^{- 1} t| = 0 \Llra s^{- 1} t \in GL_2(\O) \cdot Z,
\]
such that $d(s \cdot [0]_O, t \cdot [0]_O) = |s^{- 1} t|$ for $s, t \in GL_2(K)$.

\vspace{2mm}

\item[\rm (3)] The surjective map $GL_2(K) \lra \Lam_+,\,s \mapsto |s| = d([0]_\O, s \cdot [0]_\O)$ {\em (the composition
of the surjective maps $GL_2(K) \lra \fX,\,s \mapsto s \cdot [0]_\O,$ and $\fX \lra \Lam_+,\,x \mapsto d(x, [0]_\O)$)}
induces a bijection 
\[
PGL_2(\O) \sm PGL_2(K)/ PGL_2(\O) \cong H \sm GL_2(K)/ H \cong H \sm \fX \lra \Lam_+.
\]
\een
\end{pr}
For the proof cf.\ \cite{MS}, \cite{A-B}, and \cite{Prufer}.

For any $\alpha \in \Lam_+$, choose $t_\alpha \in \O$ with $v(t_\alpha) = \alpha$, with $t_0 = 1$, and put 
$u_{\alpha, \beta}:= \frac{t_\alpha t_\beta}{t_{\alpha + \beta}} \in \O^\times$, so that $u_{\alpha, \beta} = u_{\beta, \alpha}$ and 
$u_{\alpha, 0} = 1$. Set
\[
T_\alpha:= \begin{cases} \begin{pmatrix}0&-1\\
                          -t_\alpha&0
                         \end{pmatrix}, &\mbox{if}\,\,\alpha \in \Lam_+ - \{0\},\\[2mm]
\quad I_2, &\mbox{if}\,\,\alpha = 0.
\end{cases}
\]
Note that $\theta_\alpha:= T_\alpha^2 = Z(t_\alpha) = t_\alpha I_2 \in H$.
Our next lemma, whose proof is left to the reader (being rather similar to previous arguments),  furnishes an explicit description of the polygroup with support $H \sm GL_2(K)/ H$.
\begin{lem}
{\em (i)} Define a hyperoperation $\circ$ on $\Lambda_+$ via
\begin{align*}
\alpha \circ \beta&:= \{\alpha + \beta - 2 \gam\,:\,0 \leq \gam \leq {\rm min}\,(\alpha, \beta)\}\\
&= \alpha + \beta - 2 [0, {\rm min}\,(\alpha, \beta)]\\ 
&= |\alpha - \beta| + 2 [0, {\rm min}\,(\alpha, \beta)].
\end{align*}
Then $\Lambda_+$ becomes a commutative polygroup with scalar identity $0$ and associated map \,$\bar{}$ the identity on $\Lambda_+$. Moreover, the canoncial projection of $\Lambda_+$ onto the quotient group $\Lambda/2\Lambda$ is a morphism of polygroups, and we have
\[
\alpha\circ \beta \subseteq (\alpha+\gamma) \circ (\beta+\gamma) \mbox{ for all } \alpha, \beta, \gamma\in\Lambda_+.
\]
{\em (ii)} The map $\Lam_+ \lra H \sm GL_2(K)/ H$ given by $\alpha \mapsto C(T_\alpha):= H T_\alpha H$ is an isomorphism of polygroups.
\label{Lem:Polygroup}  
\end{lem}

Thus, $GL_2(K)$ (respectively $PGL_2(K)$) is a group extension of the polygroup with support $\Lam_+$, as defined in Part~(i) of 
Lemma~\ref{Lem:Polygroup}, by the group $H = GL_2(\O) \cdot Z$ (respectively $PGL_2(\O)$). It remains to apply 
Theorem~\ref{Thm:PresGpAct} to obtain a presentation for $GL_2(K)$ in terms of its subgroup $H$ and the chosen
representatives $T_\alpha \in GL_2(K) \setminus H$ for $\alpha \in \Lam_+ - \{0\}$ of the non-trivial double cosets modulo $H$.

Let $\alpha \in \Lam_+ - \{0\}$. We have $H_\alpha:= H \cap T_\alpha H T_\alpha^{- 1} = GL_2(\O)_\alpha \cdot Z$,
where 
\[
GL_2(\O)_\alpha:= \left\{\begin{pmatrix}a&b\\
                           c&d
                          \end{pmatrix} \in GL_2(\O)\,:\,v(c) \geq \alpha\right\};
\]
in particular, $B_2(\O) = \bigcap_{\alpha \in \Lam_+} GL_2(\O)_\alpha$.
 
The automorphism $\iota_\alpha$ of $H_\alpha$, given by $s \in H_\alpha \mapsto T_\alpha^{- 1} s T_\alpha,$ is the 
identity on $Z$, and sends $\begin{pmatrix}a&b\\
                                                           c&d
                                                          \end{pmatrix} \in GL_2(\O)_\alpha$ to 
$\begin{pmatrix}d&t_\alpha^{- 1} c\\
  t_\alpha b&a
 \end{pmatrix} \in GL_2(\O)_\alpha.$

\begin{lem}
\label{Lem:Halpha}
Let $\alpha \in \Lam_+ - \{0\}$. Then the subgroup $H_\alpha \sse H$ is generated by the matrices 
$R(a) = \begin{pmatrix}a&0\\
         0&1
        \end{pmatrix}$ and  $Z(a) = a I_2$ for $a \in \O^\times - \{1\},$ together with the matrices $Z(t_\gam) = t_\gam I_2$ for $\gam \in \Lam_+ - \{0\}$ 
and the involutory matrices $S = \begin{pmatrix}- 1&1\\
                                                   0&1
                                                  \end{pmatrix}$ and $J_\alpha = \begin{pmatrix}1&0\\
t_\alpha&- 1
\end{pmatrix}.$
\end{lem}

\bp
Obviously, all the matrices above are contained in $H_\alpha$, and the matrices $Z(a)$ for $a \in \O^\times - \{1\}$ together with the matrices $Z(t_\gam)$ for $\gam \in \Lam_+ - \{0\}$ generate the subgroup $Z \cong K^\times$. To show that the above matrices are enough to 
generate $H_\alpha$, let $s:= \begin{pmatrix}a&b\\
      c&d
     \end{pmatrix} \in GL_2(\O)_\alpha$, whence $v(a) = v(d) = 0, v(c):= \beta \geq \alpha > 0$. If $c = 0$, that is, if $s \in B_2(\O)$, we are done by Proposition~\ref{Pr:PresB(O)}. Assuming $c \neq 0$, it follows by Lemma~\ref{Lem:BGB}
that $s \in B_2(\O) J_\beta B_2(\O)$, while $J_\beta \in J_\alpha B_2(\O) J_\alpha$. The conclusion of the lemma is now obvious.
\ep
Using Lemma~\ref{Lem:Halpha}, we obtain the following relations of type (I):
\begin{equation}
\label{Eq:relI1}
[T_\alpha, Z(a)] = 1,\quad(\alpha >0,\, a \in \O^\times - \{1\}),
\end{equation}
\begin{equation}
\label{Eq:relI2}
[T_\alpha, Z(t_\gam)] = 1,\quad(\alpha, \gam > 0),
\end{equation}
\begin{equation}
\label{Eq:relI3}
R(a)^{T_\alpha} = R(a)^{- 1} Z(a),\quad(\alpha > 0,\, a \in \O^\times - \{1\}),
\end{equation}
\begin{equation}
\label{Eq:relI4}
S^{T_\alpha} = J_\alpha,\quad(\alpha > 0),
\end{equation}
\begin{equation}
\label{Eq:relI5}
J_\alpha^{T_\alpha} = S,\quad(\alpha > 0).
\end{equation}
Next, by Lemma~\ref{Lem:BGB}, (\ref{Eq:GNormalForm}), it follows that, for all $\alpha, \beta \in \Lam_+ - \{0\}$, the map 
\[
GL_2(\O) \lra [0, \mbox{min}(\alpha, \beta)],\,\begin{pmatrix}a&b\\
                                           c&d
                                          \end{pmatrix} \mapsto \mbox{min}(\alpha, \beta, v(c)),
\] 
induces a bijection $GL_2(\O)_\alpha \sm GL_2(\O)/ GL_2(\O)_\beta \lra [0, {\rm min}(\alpha, \beta)],$ and hence
$Q_{\alpha, \beta}:= \{J_\gam\,:\,0 \leq \gam < {\rm min}(\alpha, \beta)\} \cup \{I_2\}$ is a set of pairwise 
inequivalent representatives for the double cosets of the space
\[
H_\alpha \sm H/ H_\beta \cong GL_2(\O)_\alpha \sm GL_2(\O)/ GL_2(\O)_\beta,\quad (\alpha, \beta \in \Lam_+ - \{0\}).
\]
The following relations of type (II), corresponding to the relations $0 \in \alpha \circ \alpha,\,\alpha + \beta \in
\alpha \circ \beta\,(\alpha, \beta >0)$ in the polygroup $(\Lam_+, \circ)$, 
\begin{equation}
\label{Eq:relII1}
T_\alpha^2 = Z(t_\alpha)\quad(\alpha > 0),
\end{equation}
\begin{equation}
\label{Eq:relII2}
T_\alpha J T_\beta = T_{\alpha + \beta} R(u_{\alpha, \beta})\quad(\alpha, \beta  > 0).
\end{equation}
are immediate, while the rest of the relations of type (II) are consequences of the relations 
(\ref{Eq:relI1})--(\ref{Eq:relII2}) modulo the relations satisfied in $H$. Indeed, assuming $0 < \alpha < \beta$, we 
obtain the relation of type (II), corresponding to $\beta - \alpha \in \alpha \circ \beta$,
\begin{equation}
\label{Eq:relII3}
T_\alpha T_\beta = T_\alpha (T_\alpha J T_{\beta - \alpha} R(u_{\alpha, \beta - \alpha})^{- 1}) = J T_{\beta - \alpha}
R(u_{\alpha, \beta - \alpha})^{- 1} Z(t_\alpha)\quad(\alpha < \beta), 
\end{equation}  
as a consequence of (\ref{Eq:relII2}), (\ref{Eq:relII1}), and (\ref{Eq:relI2}). Similarly, assuming $\alpha > \beta$, we
obtain
\begin{equation}
\label{Eq:relII4}
T_\alpha T_\beta = T_{\alpha - \beta} R(u_{\beta, \alpha - \beta})^{- 1} J Z(t_\beta)\quad(\alpha > \beta).
\end{equation}
Assuming $0 < \gam < {\rm min}(\alpha, \beta)$, we deduce the relation of type (II), corresponding to 
$\alpha + \beta - 2 \gam \in (\alpha \circ \beta) \cap ((\alpha - \gam) \circ (\beta - \gam))$, using successively
(\ref{Eq:relI4}), (\ref{Eq:relII1}), (\ref{Eq:relI2}), (\ref{Eq:relII4}), (\ref{Eq:relII3}), (\ref{Eq:GRelII}),
(\ref{Eq:relI3}), (\ref{Eq:relI5}), (\ref{Eq:relII2}), (\ref{Eq:relI1}):
\begin{align*}
T_\alpha J_\gam T_\beta &= (T_\alpha T_\gam) (T_\gam^{- 1} J_\gam T_\gam) (T_\gam^{- 1} T_\beta)\\
&= (T_\alpha T_\gam) S (T_\gam T_\beta) Z(t_\gam)^{- 1}\\
&= (T_{\alpha - \gam} R(u_{\alpha - \gam, \gam})^{- 1}) (J S J) (T_{\beta - \gam} R(u_{\gam, \beta - \gam})^{- 1}) Z(t_\gam)\\
&= (T_{\alpha - \gam} R(u_{\alpha - \gam, \gam})^{- 1}) (S J S) (T_{\beta - \gam} R(u_{\gam, \beta - \gam})^{- 1}) Z(t_\gam)\\ 
&= (R(u_{\alpha - \gam, \gam})^{- 1} S)^{T_{\alpha - \gam}} (T_{\alpha - \gam} J T_{\beta - \gam}) (S^{T_{\beta - \gam}})
(R(u_{\gam, \beta - \gam})^{- 1} Z(t_\gam))\\
&= (R(u_{\alpha - \gam, \gam}) Z(u_{\alpha - \gam, \gam})^{- 1} J_{\alpha - \gam}) (T_{\alpha + \beta - 2 \gam} 
R(u_{\alpha - \gam, \beta - \gam})) J_{\beta - \gam} R(u_{\gam, \beta - \gam})^{- 1} Z(t_\gam)\\
&= (R(u_{\alpha - \gam, \gam}) J_{\alpha - \gam}) T_{\alpha + \beta - 2 \gam} (R(u_{\alpha - \gam, \beta - \gam}) 
J_{\beta - \gam} R(u_{\gam, \beta - \gam})^{- 1} Z(t_\alpha) Z(t_{\alpha - \gam})^{- 1}.
\end{align*}
Summarising the discussion above, and using the presentation of $GL_2(\O)$ provided by Theorem~\ref{Th:PresGL_2(O)}, we
obtain the following presentation of $GL_2(K)$ over the valued field $(K, v)$.
\begin{te}
\label{The:PresGL_2(K)}
The group $GL_2(K)$ is generated by the matrices
\[
R(a) = \begin{pmatrix}a&0\\
        0&1
       \end{pmatrix}\quad(a \in \O^\times - \{1\}),\quad S = \begin{pmatrix}-1&1\\
0&1
\end{pmatrix},\quad J = \begin{pmatrix}0&- 1\\
- 1&0
\end{pmatrix}, 
\]
and $T_\alpha = \begin{pmatrix}0&- 1\\
                 -t_\alpha&0
                \end{pmatrix}\quad(\alpha \in \Lam_+ - \{0\}),$ subject to the relations {\em (\ref{Eq:DRel1})--(\ref{Eq:DRel3}),
(\ref{Eq:SI})--(\ref{Eq:SII3}), (\ref{Eq:J,JSOrder}), (\ref{Eq:relI1})--(\ref{Eq:relI5})}, and {\em (\ref{Eq:relII2})}, 
where
\[
Z(a) = (J R(a))^2\quad (a \in \O^\times - \{1\}),
\] 
\[
Z(t_\alpha) = T_\alpha^2,\,J_\alpha = J S_\alpha J =
\begin{pmatrix}1&- 1\\
 t_\alpha&0
\end{pmatrix}\quad (\alpha \in \Lam_+ - \{0\}),
\]
and $u_{\alpha, \beta} = \frac{t_\alpha t_\beta}{t_{\alpha + \beta}}$ for $\alpha, \beta \in \Lam_+ - \{0\}$.
\end{te}

{\footnotesize
\begin{rems} 
\label{Rems:PresGL_2(K)} \em
(i) The presentation given by Theorem~\ref{The:PresGL_2(K)} can be also obtained directly from Corollary~\ref{Cor:GL_2(K)}
using the substitutions $R(a t_\alpha) = R(a) J T_\alpha$ for $a \in \O^\times$ and $\alpha \in \Lam_+ - \{0\}$.

(ii) In the presentation provided by Theorem~\ref{The:PresGL_2(K)}, we may replace the relations (\ref{Eq:SII2}) and
(\ref{Eq:relI4}) with 
\begin{equation}
\label{Eq:relI4'}
[S, R(1 - a t_\alpha)]^{R(a)} R(- 1) = S^{T_\alpha J}\quad(\alpha > 0, a \in \O^\times),
\end{equation}
and the relations (\ref{Eq:relI5}) with
\begin{equation}
\label{Eq:relI5'}
[S, T_\alpha^2] = 1\quad(\alpha > 0),
\end{equation}
with the rest of the relations remaining unchanged.
\end{rems}}

The action (\ref{Eq:ActionOnX}) induces a transitive faithful action of $PGL_2(K) = GL_2(K)/Z$ on the 
$\Lam$-tree $\fX$. Using the presentation of $GL_2(K)$ provided by Theorem~\ref{The:PresGL_2(K)} and 
Remarks~\ref{Rems:PresGL_2(K)}(ii), we obtain the 
following presentation of $PGL_2(K)$ over the valued field $(K, v)$.
\begin{co}
\label{Cor:PresPGL2}
The group $PGL_2(K)$ is generated by the $\Lam$-isometries of the $\Lam$-tree $\fX$
\begin{align*}
\wh{R}(a):&\, [z]_\gam \mapsto [a z]_\gam\,\,\,(a \in \O^\times - \{1\}),\\
\wh{S}:&\, [z]_\gam \mapsto [1 - z]_\gam,\\
\wh{J}:&\, [z]_\gam \mapsto \begin{cases}[0]_{- \gam}&\mbox{if}\,\, [z]_\gam = [0]_\gam\\
                          [\frac{1}{z}]_{\gam - 2 v(z)}&\mbox{if}\,\, [z]_\gam \neq [0]_\gam,
                         \end{cases}\\
\wh{T}_\alpha:&\,[z]_\gam \mapsto \begin{cases}[0]_{- \alpha - \gam}&\mbox{if}\,\, [z]_\gam = [0]_\gam\\
                                   [\frac{1}{t_\alpha z}]_{- \alpha + \gam - 2 v(z)}&\mbox{if}\,\,[z]_\gam \neq [0]_\gam
                                  \end{cases}\,\,(\alpha \in \Lam_+ - \{0\}),
\end{align*}
subject to the following relations
\begin{align}
\wh{R}(ab) = \wh{R}(a) \wh{R}(b)\quad(a, b \in \O^\times - \{1\})\,; \wh{R}(1):= 1,\, \label{a}\\
\wh{S}^2 = \wh{J}^2 = (\wh{J} \wh{S})^3 = 1,\\
(\wh{J} \wh{R}(a))^2 = 1\quad(a \in \O^\times - \{1\}),\\
[\wh{S}, \wh{R}(1 - a)] = \wh{R}(a) \wh{S} \wh{R}(- a)^{- 1}\quad(a \in \mathfrak{k}),\, \label{b}\\
[\wh{S}, \wh{R}(1 - a t_\alpha)]^{\wh{R}(a)} \wh{R}(- 1) = \wh{S}^{\wh{T}_\alpha \wh{J}}\quad (\alpha \in 
\Lam_+ - \{0\}, a \in \O^\times),\,\label{c}\\
(\wh{T}_\alpha \wh{R}(a))^2 = 1\quad(\alpha \in \Lam_+ - \{0\}, a \in \O^\times),\,\label{d}\\
\wh{T}_\alpha \wh{J} \wh{T}_\beta = \wh{T}_{\alpha + \beta} \wh{R}(u_{\alpha, \beta})\quad(\alpha, \beta \in 
\Lam_+ - \{0\}),\,\label{e}
\end{align}
where $u_{\alpha, \beta}:= \frac{t_\alpha t_\beta}{t_{\alpha + \beta}}$ for $\alpha, \beta \in \Lam_+ - \{0\}$, and 
$\mathfrak{k} \sse \O^\times - (1 + \fm)$ is a set of representatives modulo the maximal ideal $\fm$ of the valuation 
ring $\O$ for the elements in $k^\times - \{1\}$.                        
\end{co}

{\footnotesize 
\begin{rem} \em
The presentation given by Corollary~\ref{Cor:PresPGL2} can be obtained directly from Corollary~\ref{Co:PresPGL_2(K)}
using the substitutions 
\[
\wt{R}(a) \mapsto \wh{R}(a)\,(a \in \O^\times - \{1\}),\, \wh{R}(a t_\alpha) \mapsto \wh{R}(a) \wh{J} \wh{T}_\alpha\,
(a \in \O^\times,\, \alpha \in \Lam_+ - \{0\}),\, \wt{S} \mapsto \wh{S},\, \wt{J} \mapsto \wh{J}.
\]
\end{rem}}

In particular, assuming that the valuation $v$ is discrete of rank $1$, so that $\Lam \cong \Z$, we obtain the following.

\begin{co}
If the valuation $v$ is discrete of rank $1,$ and $t:= t_1$ is a local uniformizer, then the group $PGL_2(K)$ is
generated by the isometries $\wh{R}(a)\,(a \in \O^\times - \{1\}),$ $\wh{S}, \wh{J}, \wh{T}:= \wh{T}_1$ of the $\Z$-tree
$\fX$, subject to the relations {\em (\ref{a})-(\ref{b})} together with the relations
\begin{align}
[\wh{S}, \wh{R}(1 - at)]^{\wh{R}(a)} \wh{R}(- 1) = \wh{S}^{\wh{T} \wh{J}}\quad(a \in \O^\times),\,\label{c'}\\
(\wh{T} \wh{R}(a))^2 = 1\quad(a \in \O^\times)\,\label{d'}.
\end{align} 
\end{co}
\bp
Taking $t_n:= t^n$ for $n \in \N$, we obtain $\wh{T}_n = \wh{T}_{n - 1} \wh{J} \wh{T} = (\wh{T} \wh{J})^{n - 1} \wh{T}$
for $n \geq 2$ by (\ref{e}), thus $PGL_2(K)$ is generated by $\wh{R}(a)\,(a \in \O^\times - \{1\}), \wh{S}, \wh{J},$ 
and $\wh{T}$. It remains to show that the relations (\ref{c}) and (\ref{d}) are consequences of the relations 
(\ref{a})-(\ref{b}), (\ref{c'}), and (\ref{d'}). 

Since $\wh{J}^2 = \wh{T}^2 = 1$, it follows that $\wh{T}_n = (\wh{T} \wh{J})^{n - 1} \wh{T} = \wh{T} (\wh{J} \wh{T})^{n - 1} =
\wh{T}_n^{- 1}$, that is, $\wh{T}_n^2 = 1$ for all $n \geq 1$. As $\wh{J} \wh{R}(a) \wh{J} = \wh{T} \wh{R}(a) \wh{T} =
\wh{R}(a)^{- 1}$, we deduce that $(\wh{T}_n \wh{R}(a))^2 = 1$ for all $n \geq 1, a \in \O^\times$, that is, (\ref{d}) is satisfied.

To check (\ref{c}), note that (\ref{c'}) implies $[\wh{S}, \wh{R}(1 - a t)]^{\wh{R}(a)} = 
[\wh{S}, \wh{R}(1 - t)]\,(a \in \O^\times)$, and hence $[\wh{S}, \wh{R}(1 - a t^n)]^{\wh{R}(a)} = 
[\wh{S}, \wh{R}(1 - t^n)]\,(a \in \O^\times)$ for all $n \geq 1$, according to Part~(ii) of Remarks~\ref{Rem:B(O)}. Thus, it 
remains to prove that the identity 
\[
[\wh{S}, \wh{R}(1 - t^n)] = \wh{S}^{(\wh{T} \wh{J})^n} \wh{R}(- 1)
\] 
holds for all $n \geq 1$. We obtain by induction
\begin{align*}
[\wh{S}, \wh{R}(1 - t^{n + 1})] &= [\wh{S}, \wh{R}(1 - t^n)] \wh{R}(1 - t^n)[\wh{S}, \wh{R}(1 - \frac{t - 1}{1 - t^n} t^n)] \wh{R}(1 - t^n)^{- 1}\\
&= [\wh{S}, \wh{R}(1 - t^n)] \wh{R}(t - 1) [\wh{S}, \wh{R}(1 - t^n)] \wh{R}(t - 1)^{- 1} \\
&= \wh{S}^{(\wh{T} \wh{J})^n} \wh{R}(1 - t) \wh{S}^{(\wh{T} \wh{J})^n} \wh{R}(1 - t)^{- 1}\\
&= [\wh{S}, \wh{R}(1 - t)]^{(\wh{T} \wh{J})^n}\\
&= (\wh{S}^{\wh{T} \wh{J}} \wh{R}(- 1))^{(\wh{T} \wh{J})^n}\\
&= \wh{S}^{(\wh{T} \wh{J})^{n + 1}} \wh{R}(- 1), 
\end{align*}
as desired.
\ep

\section{$SL_3$ over fields}
\label{Sec:SL3}

Let $K$ be a field, and let $G=SL_3(K)$ be the $3$-dimensional special linear group over $K$. We look for a presentation
of $G$ induced by its natural action on the projective plane $\mathbb{P}^2(K)$. 
Let ${\bf 0}:= [0,0,1]^t$ be the origin of the affine plane
\[
\mathbb{A}^2(K)_{x,y} = \big\{[a,b,1]^t:\, a,b\in K\big\}.
\]
The stabiliser $G_{\bf 0}$ consists of all matrices $A=(a_{i,j})\in G$ with $a_{1,3}=a_{2,3}=0$. 
Since the action is $2$-transitive, the space $G_{\bf 0}\backslash G/G_{\bf 0}$ contains exactly one non-trivial double
coset $G_{\bf 0} A G_{\bf 0}$ for some (for all) $A \in G \setminus G_{\bf 0}$. As $G \setminus G_{\bf 0}$ contains 
involutions, we may choose one of them, say
\begin{equation}
\label{Eq:T}
T:= \begin{pmatrix}0&0&-1\\0&-1&0\\-1&0&0\end{pmatrix},
\end{equation}
and use Corollary~\ref{Cor:Main2Trans} to obtain a presentation of $G$ in terms of $G_{\bf 0}$ and the  
extra generator $T$.

The involution $T$ moves the point $\bf 0$ to the origin ${\bf 0'}:= [1, 0, 0]^t$ of the affine plane 
$\mathbb{A}^2(K)_{y,z}$, whose stabilizer $G_{\bf 0'} = T G_{\bf 0} T$ consists of all matrices $A \in G$ with 
$a_{2, 1} = a_{3, 1} = 0$. Consequently, the intersection $G_{\bf 0, 0'} = G_{\bf 0} \cap G_{\bf 0'}$ consists of
those matrices $A \in G$ such that $a_{1, 3} = a_{2, 1} = a_{2, 3} = a_{3, 1} = 0$. 

The isomorphism $G_{\bf 0} \rightarrow G_{\bf 0'},\,A \mapsto T A T$ induces the involutive automorphism $\omega_T$ 
of $G_{\bf 0, 0'}$, defined by $\omega_T(A) = A'$ with $a'_{i, j} = a_{\tau(i), \tau(j)}$, where $\tau$ is the 
transposition $(1, 3)$. Consequently, we obtain the relations of type (I)
\begin{equation}
\label{Eq:G/G0RelI}
T A = \omega_T(A) T,
\end{equation}
where $A$ ranges over a set of generators of $G_{\bf 0, 0'}$.

To obtain the relations of type (II) we have to choose representatives in $G_{\bf 0}$ for the cosets of 
$G_{\bf 0}/G_{\bf 0, 0'}$ and the double cosets of $G_{\bf 0, 0'}\backslash G_{\bf 0}/G_{\bf 0, 0'}$. Finally, to get an explicit presentation of $G = SL_3(K)$, it remains to describe effectively the stabilizer $G_{\bf 0}$ and the double stabilizer $G_{\bf 0, 0'}$, as well as all required additional structures; this task will be achieved step by step in 
the following, using Theorem~\ref{Thm:PresGpAct} and its Corollary~\ref{Cor:Main2Trans} as basic tools . 

\subsection{The structure of the stabilizers $G_{\bf 0}$ and $G_{\bf 0, 0'}$}
\label{SubSubSec:G0}
The underlying abelian group of the $2$-dimensional vector space $\mathfrak{V}:= K e_1 \oplus K e_2$ is identified with 
a normal subgroup $\hat{\mathfrak{V}}$ of $G_{\bf 0}$ via the embedding $\mathfrak{V} \rightarrow G_{\bf 0}$ defined by 
$v = a_1 e_ 1 + a_2 e_2 \mapsto \hat{v} = {\begin{pmatrix}1&0&0\\0&1&0\\a_1&a_2&1\end{pmatrix}}$, while the 
group $GL_2(K)$ is identified with a complement $\widehat{GL_2(K)}$ of $\hat{\mathfrak{V}}$ via the embedding 
$GL_2(K)\rightarrow G_{\bf 0}$ given by 
$A \mapsto \hat{A} = {\begin{pmatrix}{}&{}&0\\{}&A&0\\0&0&\vert A\vert^{-1}\end{pmatrix}}$; in particular,
the matrices $Z(a), R(a)\,(a \in K^\times), S$, and $J$ in $GL_2(K)$, as defined in Section~\ref{Sec:GL2}, are lifted to
\begin{multline*}
\hat{Z}(a) = \begin{pmatrix}a&0&0\\0&a&0\\0&0&a^{-2}\end{pmatrix},\,\, \hat{R}(a) = 
\begin{pmatrix}a&0&0\\0&1&0\\0&0&a^{-1}\end{pmatrix}\,\,\,(a \in K^\times),\\ 
\hat{S} = \begin{pmatrix}-1&1&0\\0&1&0\\0&0&-1\end{pmatrix},\,\mbox{ and }\hat{J} = 
\begin{pmatrix}0&-1&0\\-1&0&0\\0&0&-1\end{pmatrix}.
\end{multline*} 
Thus, the group $G_{\bf 0}$ is identified with the semi-direct product $\tilde{G}_{\bf 0} = \mathfrak{V} \ltimes GL_2(K)$ 
induced by the action from the left of $GL_2(K)$ on the vector space $\mathfrak{V} = Ke_1 \oplus K e_2$ given by 
the $GL_2(K)$-endomorphism
\[
GL_2(K) \ni A \mapsto \vert A\vert^{-2}(J R(-1) A R(-1) J) = \begin{pmatrix} \frac{a_{2,2}}{\vert A\vert^2}& \frac{-a_{2,1}}{\vert A\vert^2}\\[2mm] \frac{-a_{1,2}}{\vert A\vert^2}& \frac{a_{1,1}}{\vert A\vert^2}\end{pmatrix}.
\]
The kernel of the action is $\{\zeta I_2\,:\,\zeta^3 = 1\}$, isomorphic with the cyclic group of order $3$ provided
the field $K$ contains a primitive root of order $3$ of the unity; otherwise, the action is faithfull; this happens, in particular, 
if ${\rm char}\,K = 3$, and also for $K$ formally real, or for $K = GF(q)$ finite with 
$q \equiv\,- 1\,{\rm mod}\,3$. Note also that there are only two $GL_2(K)$-orbits of $\mathfrak{V}$: the singleton 
$\{0\}$ and its complementary set.

The action above induces by restriction an action of the Borel subgroup $B = B_2(K)$ of upper-triangular matrices of $GL_2(K)$ on
the $1$-dimensional subspace $K e_2$ of $\mathfrak{V}$ given by $(A, a e_2) \mapsto \frac{a}{a_{1, 1} a_{2, 2}^2} e_2$. The kernel of
the action, as well as the stabilizer of $e_2$ under the action of $GL_2(K)$, is the normal subgroup $B'$ of $B$ 
consisting of all matrices $A \in B$ such that $a_{1, 1} a_{2, 2}^2 = 1$. Consequently, the double stabilizer 
$G_{\bf 0, 0'} \subseteq G_{\bf 0}$ is identified with the semi-direct product $\tilde{G}_{\bf 0, 0'} = (K e_2)
\ltimes B$ induced by the action above.

\subsection{A presentation for the double stabiliser $G_{\bf 0, 0'}$}
\label{SubSubSec:G00/B}
For $K \cong GF(2)$, we have $G_{\bf 0, 0'} \cong C_2 \times C_2$, the Klein $4$-group; thus, we may
assume that $K \not \cong GF(2)$. Although a presentation of $G_{\bf 0, 0'} \cong \tilde{G}_{\bf 0, 0'}$ 
could easily be obtained using its semi-direct product structure given at the end of Section~\ref{SubSubSec:G0}, we prefer to apply Corollary~\ref{Cor:Main2Trans} 
using the fact that $\tilde{G}_{\bf 0, 0'}$ is a group extension of the commutative polygroup $\fP_2$ by the 
Borel group $B = B_2(K)$, whose presentation is given in Corollary~\ref{Cor:PresB(K)}. Indeed, $B$ is not normal in 
$\tilde{G}_{\bf 0, 0'}$, and the space $B\backslash \tilde{G}_{\bf 0, 0'}/B$ contains only one non-trivial 
double coset $B g B$ for some (for all) $g \in \tilde{G}_{\bf 0, 0'} \setminus B$. As 
$\tilde{G}_{\bf 0, 0'} \setminus B$ contains involutions, we may choose one of them, say 
$W:= (e_2) \cdot R(- 1) = R(- 1) \cdot (- e_2)$; note that $W = e_2$ provided ${\rm char}\,K = 2$. Thus $G_{\bf 0, 0'}$
is generated by its subgroup $\hat{B} \cong B$ plus one extra generator 
\[
\hat{W}:= \begin{pmatrix}-1&0&0\\0&1&0\\0&1&-1\end{pmatrix},
\]
the image of $W$ via the isomorphism $\tilde{G}_{\bf 0, 0'} \rightarrow G_{\bf 0, 0'}$. It remains to deduce the
relations of type (I) and (II).

It follows that $B \cap W B W = B'$, while the  automorphism (of order $2$) $\omega_W$ of $B'$, $A \mapsto W A W$, sends
a matrix $A \in B'$ to the matrix $A' \in B'$ with $a'_{i, i} = a_{i, i}$ for $i = 1, 2$, and $a'_{1, 2} = - a_{1, 2}$. 
To obtain the type (I) relations, we have to choose a set of generators of $B'$. Note that the 
monomorphisms $K^+ \rightarrow B'$, $a\mapsto E_{1,2}(a)=${\tiny $\begin{pmatrix}1&a\\0&1\end{pmatrix}$}, and 
$K^\times \rightarrow B'$, $a\mapsto R(a)^3Z(a)^{-1}=${\tiny $\begin{pmatrix}a^2&0\\0&a^{-1}\end{pmatrix}$}, induce 
an isomorphism from the semi-direct product $K^+\ltimes K^\times$ onto $B'$, where the action of $K^\times$ on 
$K^+$ is given by $(a,b)\mapsto a^3b$; in particular, the matrices $\hat{R}(a)^3 \hat{Z}(a)^{-1}$ for 
$a\in K^\times\setminus\{1\}$ together with the elementary matrices 
$\hat{E}_{1,2}(a) = \hat{R}(a) \hat{S} \hat{R}(-a)^{-1}$ for $a\in K^\times/(K^\times)^3$ form a set of 
generators for $\hat{B'}\subseteq \hat{B}$, now thought of as embedded into $SL_3(K)$. Consequently, the relations 
of type (I) take the form
\begin{align}
[\hat{W}, \hat{R}(a)^3 \hat{Z}(a)^{-1}] &= 1\quad(a\in K^\times\backslash\{1\}),\label{Eq:G00/BTI1}\\
(\hat{W} \hat{E}_{1, 2}(a))^2 &= 1\quad(a\in \cR),\label{Eq:G00/BTI2}
\end{align}
where $\cR \subseteq K^\times$ is a fixed set of representatives for the quotient group $K^\times/(K^\times)^3$ with $1\in\cR$.

Next we note that the epimorphism $B\rightarrow K^\times$, $A\mapsto a_{1,1}a_{2,2}^2$, with kernel $B'$, admits 
the homomorphic section $K^\times \rightarrow B$, $a\mapsto R(a)$; hence, the subgroup 
$\{R(a): a\in K^\times\} \cong K^\times$ is a complement to $B'$ in $B$, whence a system of representatives for
the cosets (double cosets) in $B/B' = B'\backslash B/B'$. Moreover, it follows that $B$ is isomorphic to the 
semi-direct product $B' \ltimes K^\times \cong (K^+\ltimes K^\times)\ltimes K^\times$ induced by the action
\[
K^\times\times(K^+\ltimes K^\times) \rightarrow K^+\ltimes K^\times,\, (a,(b,c))\mapsto (ab,c).
\]
Consequently, any element $g \in \tilde{G}_{\bf 0, 0'} \setminus B$ can be uniquely written in the form $g = R(a) W A$ with
$a \in K^\times$ and  $A \in B$; in particular, $a e_2 = R(a)^{- 1} W R(- a)$ for all $a \in K^\times$.

We thus obtain the type (II) relations 
in the following form, similar to that of the relations of type (II) (\ref{Eq:SII1}) and (\ref{Eq:SII3}) 
in the presentation of $B_2(K)$ in Corollary ~\ref{Cor:PresB(K)}.

\begin{align}
\hat{W}^2 &= 1,\label{Eq:G00/BTII1}\\[1mm]
\hat{W} \hat{R}(a) \hat{W} &= \hat{R}\big(\frac{a}{a-1}\big) \hat{W} \hat{R}(1-a),\quad (a\in K^\times\backslash\{1\}).\label{Eq:G00/BTII2}
\end{align}
Summarising the preceding discussion, and implementing the presentation for $B_2(K)$ given in Corollary~\ref{Cor:PresB(K)}, 
we thus obtain the following.

\begin{pr}
\label{Prop:G00Pres}
The group $G_{\bf 0,0'}$ is generated by the matrices $\hat{Z}(a)$, $\hat{R}(a)$ for $a\in K^\times\backslash\{1\},$ 
together with the matrices $\hat{S},$ and $\hat{W}$, subject to the relations
\begin{align}
\hat{Z}(a) \hat{Z}(b) &= \hat{Z}(ab),\label{Eq:G00Rel1}\\
\hat{R}(a) \hat{R}(b) &= \hat{R}(ab),\label{Eq:G00Rel2}\\
[\hat{Z}(a), \hat{R}(b)] &= 1,\label{Eq:G00Rel3}\\ 
[\hat{S}, \hat{Z}(a)] &= 1,\label{Eq:G00Rel4}\\
\hat{S}^2 &= 1,\label{Eq:G00Rel5}\\
[\hat{S}, \hat{R}(a)] &= \hat{E}_{1, 2}(1-a),\label{Eq:G00Rel6}
\end{align} 
where $a,b\in K^\times\backslash\{1\}, \hat{E}_{1, 2}(a):= \hat{R}(a) \hat{S} \hat{R}(- a)^{- 1}$, 
together with the relations {\em (\ref{Eq:G00/BTI1})--(\ref{Eq:G00/BTII2})}.
\end{pr}
Note that the trivial case where $K \cong GF(2)$ is also included.

{\footnotesize 
\begin{rem} \em
Assuming $K \not \cong GF(2)$, the group $G_{\bf 0, 0'} \cong \tilde{G}_{\bf 0, 0'}$  is a {\em double} extension
of the polygroup $\fP_2$ by the torus $D \cong K^\times \times K^\times$. Indeed, $B$ is an extension of $\fP_2$ by $D$,
and $G_{\bf 0, 0'}$ is an extension of $\fP_2$ by $B$. The presentation of $G_{\bf 0, 0'}$ given by 
Proposition~\ref{Prop:G00Pres} could be seen as an application in two successive steps of Corollary~\ref{Cor:Main2Trans}
involving the polygroup $\fP_2$. As an alternative approach, we could see $G_{\bf 0, 0'}$ as a group extension of
the polygroup with support $D\backslash \tilde{G}_{\bf 0, 0'}/D$ by $D$, applying only once Theorem~\ref{Thm:PresGpAct}
to obtain a presentation of $G_{\bf 0, 0'}$; this approach is left as exercies to the interested reader.
\end{rem}}

\subsection{A presentation for the stabilizer $G_{\bf 0}$}
\label{SubSubSec:G0/GL2K}
As shown in \ref{SubSubSec:G0}, $G_{\bf 0} \cong \tilde{G}_{\bf 0} = \mathfrak{V} \ltimes GL_2(K)$, where
$\mathfrak{V} = K e_1 \oplus K e_2$; hence, we could use the semi-direct product structure to obtain a presentation 
for $G_{\bf 0}$. We prefer however the alternative of applying Theorem~\ref{Thm:PresGpAct} or its 
Corollary~\ref{Cor:Main2Trans} to a suitable pair $(\tilde{G}_{\bf 0}, U)$, where $U$ is one of the subgroups of 
$\tilde{G}_{\bf 0}$ for which we already have a presentation; for instance, we may take $U$ as  
$\tilde{G}_{\bf 0, 0'} = (K e_2) \ltimes B$, or as $GL_2(K)$. In view of the complexity of the polygroup with support 
$U\backslash \tilde{G}_{\bf 0}/U$, we choose the latter case because of its simplicity, leaving the first case to the interested reader as an exercise.

Indeed, since $GL_2(K)$ acts transitively on the set $\mathfrak{V} \setminus \{0\}$, it follows that there exists 
exactly one non-trivial double coset $GL_2(K) g GL_2(K)$ for some (for all) $g \in \tilde{G}_{\bf 0} \setminus GL_2(K)$. 
As $GL_2(K)$ is not normal in $\tilde{G}_{\bf 0}$, $\tilde{G}_{\bf 0}$ is a group extension of the polygroup 
$\mathfrak{P}_2$ by $GL_2(K)$. Since the involution $W \in \tilde{G}_{\bf 0, 0'} \subseteq \tilde{G}_{\bf 0}$, as defined in 
\ref{SubSubSec:G00/B}, does not belong to $GL_2(K)$, we may use Corollary~\ref{Cor:Main2Trans} to obtain a presentation 
of $G_{\bf 0} \cong \tilde{G}_{\bf 0}$ in terms of $\widehat{GL_2(K)}$, the lifting of $GL_2(K)$ to $SL_3(K)$,
plus one involutive extra generator $\hat{W}$, the lifting of $W$ to $SL_3(K)$. 

Note that 
\[
GL_2(K) \cap W GL_2(K) W^{- 1} = B \cap W B W = B' = \big\{A \in B\,|\,a_{1, 1} a_{2, 2}^2 = 1\big\}, 
\]
the stabilizer of the point $e_2 \in \mathfrak{V}$ under the action of $GL_2(K)$, while the induced automorphism 
$\omega_W \in {\rm Aut}\,(B')$ is already defined in \ref{SubSubSec:G00/B}.
Consequently, the type (I) relations for $G_{\bf 0} \cong \tilde{G}_{\bf 0}$, the same as for $G_{\bf 0, 0'}$, are 
given by (\ref{Eq:G00/BTI1}) and (\ref{Eq:G00/BTI2}).

In order to establish the relations of type (II), it remains to provide systems of representatives for the cosets of
the space $GL_2(K)/B'$ and the double cosets of the space $B'\backslash GL_2(K)/B'$. Thus we have to suitably extend 
the set $\{R(a)\,:\,a \in K^\times\}$  of representatives for the elements of the common subspace
$B/B' = B'\backslash B/B'$. Using the Bruhat decomposition $GL_2(K) = B \sqcup B J B$ and the normal
form of the elements of $GL_2(K)$ according to Lemma~\ref{Lem:BGB} with $\alpha = 0$, we obtain the following.
\begin{lem}
\label{Lem:GL_2/B'DC's}
$(1)$ The matrices 
\[
R(a)\,\,(a \in K^\times)\,\,\mbox{and}\,\,E_{1, 2}(a) J R(b) = \begin{pmatrix}-ab&-1\\
-b&0\end{pmatrix}\,\,(a \in K, b \in K^\times)
\] 
form a set of (pairwise inequivalent) representatives for the cosets of the space $GL_2(K)/B'$.

$(2)$ The matrices $R(a)$ for $a \in K^\times$ together with the matrices 
$J R(b) = {\tiny \begin{pmatrix}0&-1\\-b&0\end{pmatrix}},$ with $b$ ranging over a set $\cR \subseteq K^\times$
of representatives for the quotient group $K^\times/(K^\times)^3,$ form a set of (pairwise inequivalent) 
representatives for the double cosets of the space $B'\backslash GL_2(K)/B'$.   
\end{lem}

For any representative $g$ of a double coset in $B'\backslash GL_2(K)/B'$, we have to compute 
the uniquely determined matrices $\rho_g, \lam_g \in GL_2(K)$ such that $\rho_g$ belongs to the set of
representatives for the cosets of the space $GL_2(K)/B'$, and the identity $W g W = \rho_g W \lam_g$ holds in 
$\tilde{G}_{\bf{0}}$. Doing the computations, we obtain as relations of type (II) for $G_{\bf 0} \cong \tilde{G}_{\bf 0}$
the relations (\ref{Eq:G00/BTII1}), (\ref{Eq:G00/BTII2}), the same as for $G_{\bf 0, 0'}$, together with the more 
intricate relations
\begin{equation}
\hat{W} \hat{J} \hat{R}(a) \hat{W} = \hat{\rho_a} \hat{W} \hat{\lambda_a},\quad a\in \cR,\label{Eq:G0TII3}
\end{equation}
where 
\begin{align*}
\hat{\rho_a}:=\,& \hat{E}_{1,2}(-a) \hat{J}\hat{R}(a) = {\tiny \begin{pmatrix}a^2&-1&0\\-a&0&0\\0&0&-a^{-1}\end{pmatrix}},\\
\hat{\lambda_a}:=\,& \hat{E}_{1,2}(-a^{-2}) \hat{J} \hat{E}_{1,2}(a^2) \hat{R}(a)^4 \hat{Z}(a)^{-2} = 
{\tiny \begin{pmatrix} 1&0&0\\-a^2&-1&0\\0&0&-1\end{pmatrix}}.
\end{align*}
In particular, setting $a = 1$ in (\ref{Eq:G0TII3}), we obtain a relation involving the involutionary matrices 
$\hat{S}, \hat{W}, \hat{J}$, and $\hat{R}(- 1)$, namely 
\begin{equation}
(\hat{S} \hat{R}(- 1) \hat{W} \hat{J})^3 = 1.\quad\label{Eq:SRWJOrder}
\end{equation}
 
{\footnotesize 
\begin{rem} \em
\label{Rem:G0JR(a)}
If $K^\times = (K^\times)^3$, which is the case, in particular, if $K$ is algebraically closed, or real closed, or a perfect field 
of characteristic $3$, or a finite field $GF(q)$ with $q \equiv - 1\,{\rm mod}\,3$, then (\ref{Eq:G0TII3}) reduces 
to the single relation (\ref{Eq:SRWJOrder}).
\end{rem}}

The identity (\ref{Eq:G0TII3}) holds for all $a \in K^\times$. Indeed, for $a, b \in K^\times$, we deduce
the identity (\ref{Eq:G0TII3}) for $a b^3$ from that for $a$, using the relations (\ref{Eq:G00Rel1})--(\ref{Eq:G00Rel3}), 
(\ref{Eq:GRelI}) (lifted to $SL_3(K)$), (\ref{Eq:G00/BTI1}) and (\ref{Eq:G00/BTI2}).

Moreover, we observe that (\ref{Eq:G0TII3}) is strong enough to imply (\ref{Eq:G00/BTII2}). Indeed, let 
$a \in K^\times \setminus \{1\}$. Then, using (\ref{Eq:G0TII3}), we obtain
$W R(a) W = (W J W) (W J R(a) W) = \rho_1 W \lam_1 \rho_a W \lam_a$. Since
\[
\lam_1 \rho_a = \begin{matrix}
\left ( \matrix
a^2 & - 1\\
a(1 - a) & 1
\endmatrix \right )
\end{matrix} = E_{1, 2}(\frac{a}{1 - a}) (J R(\frac{a}{a - 1})) 
E_{1, 2}(\frac{(a - 1)^2}{a}) (R(1 - a)^3 Z(1 - a)^{- 1})
\]
belongs to $B' (J R(\frac{a}{a - 1})) B'$, it follows by (\ref{Eq:G00/BTI1}), (\ref{Eq:G00/BTI2}), and 
(\ref{Eq:G0TII3}) that
\[
W \lam_1 \rho_a W = E_{1, 2}(\frac{a}{a - 1}) \rho_{\frac{a}{a - 1}} W 
\lam_{\frac{a}{a - 1}} E_{1, 2}(- \frac{(a - 1)^2}{a}) R(1 - a)^3 Z(1 - a)^{- 1},
\]
hence $W R(a) W = R(\frac{a}{a - 1}) W R(1 - a)$ as desired, since, again by (\ref{Eq:G00/BTI2}) and (\ref{Eq:G0TII3}) 
\[
\rho_1 E_{1, 2}(\frac{a}{a - 1}) \rho_{\frac{a}{a - 1}} W = R(\frac{a}{a - 1})
E_{1, 2}(\frac{1 - a}{a}) W = R(\frac{a}{a - 1}) W E_{1, 2}(\frac{a - 1}{a}).
\]

Summarising the preceding discussion, and implementing the presentation for $GL_2(K)$ given in 
Corollary~\ref{Cor:GL_2(K)}, we have thus obtained the following result.

\begin{pr}
\label{Prop:G_0Pres}
The group $G_{\bf 0}$ has a presentation with generators $\hat{R}(a)$ with $a\in K^\times\setminus\{1\},$ together 
with the matrices $\hat{S}, \hat{J},$ and $\hat{W},$ subject to the relations {\em (\ref{Eq:G00Rel1})--(\ref{Eq:G00Rel6}),}
and {\em (\ref{Eq:J,JSOrder})} lifted to $SL_3(K),$ that is, 
\begin{align}
\hat{J}^2 &= (\hat{J} \hat{S})^3 = 1,\label{Eq:LiftedJ,JSOrder}
\end{align}
plus the relations {\em (\ref{Eq:G00/BTI1}), (\ref{Eq:G00/BTI2}), 
(\ref{Eq:G00/BTII1}),} and {\em (\ref{Eq:G0TII3}),} where $\hat{Z}(a):= (\hat{J} \hat{R}(a))^2$ and 
$\hat{E}_{1, 2}(a):= \hat{R}(a) \hat{S} \hat{R}(- a)^{- 1}$ for $a\in K^\times$.
\end{pr}

{\footnotesize 
\begin{rem} \em 
The natural semidirect product structure of $\tilde{G}_{\bf 0} = \mathfrak{V} \ltimes GL_2(K)$
provides a presentation with generators $R(a), Z(a) (a \in K^\times), S, J$ (for $GL_2(K)$),
and $a e_i, i = 1, 2 (a \in K^\times)$ (for $(\mathfrak{V} = K e_1 + K e_2, +)$), subject to the relations 
for $GL_2(K)$ (cf.\ Corollary~\ref{Cor:GL_2(K)}), together with the relations
\[
(\mathfrak{V}, +)\, (a e_i) \cdot (b e_i) = (a + b)e_i, i= 1, 2;\,
(a e_1) \cdot (b e_2) = (b e_2) \cdot (a e_1) (a, b \in K^\ast),
\]
with $(0 e_i) = 1$, and the relations describing the action as defined in \ref{SubSubSec:G0}
\begin{align*}
({\rm Act}\,)\,R(a) (b e_1) R(a)^{- 1} &= a^{- 2} b e_1,\, R(a) (b e_2) R(a)^{- 1} = 
a^{- 1} b e_2,\\
Z(a) (b e_i) Z(a)^{- 1} &= a^{- 3} b e_i, i = 1, 2\quad (a, b \in K^\times),\\
S (b e_1) S &= (b e_1) \cdot (- b e_2),\\
S (b e_2) S &= - b e_2,\\ 
J (b e_1) J &= b e_2\quad (b \in K^\times).
\end{align*}
The substitutions 
\begin{align*}
W &\mapsto e_2 \cdot R(- 1),\,{\rm and}\\
a e_1 &\mapsto J R(a)^{- 1} W R(- a) J, \\
a e_2 &\mapsto R(a)^{- 1} W R(- a) (a \in K^\times)
\end{align*}
are compatible with the corresponding relations, inducing  inverse isomorphisms between the presentations.
\end{rem}}

\subsection{The space of double cosets $G_{\bf 0, 0'}\backslash G_{\bf 0}/G_{\bf 0, 0'}$}

As explained at the beginning of Section~\ref{Sec:SL3}, we need a description of the space of double cosets 
$G_{\bf 0, 0'}\backslash G_{\bf 0}/G_{\bf 0, 0'}$ in order to obtain the desired presentation of the group $G = SL_3(K)$.
Note that the $2$-transitive action of $G$ on  $\P^2(K)$ induces by restriction a transitive action of 
$G_{\bf 0}$ on $\P^2(K) \setminus \{\bf 0\}$, while $G_{\bf 0, 0'}$ is the stabilizer of the point $\bf 0'$ with 
respect to the induced action. Although the polygroup with support $G_{\bf 0, 0'}\backslash G_{\bf 0}/G_{\bf 0, 0'}$
is not isomorphic to the polygroup $\fP_2$ (as the induced transitive action is not $2$-transitive), it is quite
small as we shall see next.

Let us denote by ${\fP_3} = (\{0, 1, 2\}, \oplus)$ the commutative polygroup of cardinality $3$, with $0$ as scalar identity, and $1 \oplus 1 = \{0, 1\}, 2 \oplus 2 = \{0, 1, 2\}, 1 \oplus 2 = 2 \oplus 1 = 2$. Note that 
the polygroup $\fP_2$ is identified with the subpolygroup $(\{0, 1\}, \oplus)$ of $\fP_3$.

\begin{lem}
\label{Lem:G00G0G00}
$(1)$ The space $X:= G_{\bf 0, 0'}\backslash G_{\bf 0}/G_{\bf 0, 0'}$ contains exactly two distinct non-trivial 
double cosets, which may be represented by the involutions 
$\hat{J} = {\tiny \begin{pmatrix}0&-1&0\\-1&0&0\\0&0&-1\end{pmatrix}}$ and 
$\hat{V}:= \hat{J} \hat{W} \hat{J} = {\tiny \begin{pmatrix}1&0&0\\0&-1&0\\1&0&-1\end{pmatrix}}$.

$(2)$ The map 
\[
\{0, 1, 2\} \rightarrow X \mbox{ given by }0 \mapsto G_{\bf 0, 0'},\, 1 \mapsto G_{\bf 0, 0'} \hat{V} G_{\bf 0, 0'},\, 
2 \mapsto G_{\bf 0, 0'} \hat{J} G_{\bf 0, 0'}
\] 
maps the polygroup $\fP_3$ isomorphically onto the polygroup with support $X$.
\end{lem}

\bp
(1) Recall from \ref{SubSubSec:G0} that $G_{\bf 0} \cong \tilde{G}_{\bf 0} = \mathfrak{V} \ltimes GL_2(K)$, where 
$\mathfrak{V} = K e_1 \oplus K e_2$, and $G_{\bf 0, 0'} \cong \tilde{G}_{\bf 0, 0'} = (K e_2) \ltimes B$; hence, we
have to describe the space $\tilde{X}:= \tilde{G}_{\bf 0, 0'}\backslash \tilde{G}_{\bf 0}/ \tilde{G}_{\bf 0, 0'}$.
Set $C(g):= \tilde{G}_{\bf 0, 0'} g \tilde{G}_{\bf 0, 0'}$ for $g \in \tilde{G}_{\bf 0}$.

First let $H:= \mathfrak{V} \ltimes B$, a subgroup of $\tilde{G}_{\bf 0}$ lying over $\tilde{G}_{\bf 0, 0'}$. Note that
the involution
\[
V := e_1 \cdot R(- 1) Z(- 1) = R(- 1) Z(- 1) \cdot (- e_1)
\]
belongs to $H \setminus \tilde{G}_{\bf 0, 0'}$. Note also that $V = J W J$ since $e_1 = J e_2 J = J W R(- 1) J$
and $J R(- 1) J = R(- 1) Z(- 1)$.

Next let $g \in \tilde{G}_{\bf 0} \setminus \tilde{G}_{\bf 0, 0'}$.
We distinguish the following two cases.

(i) $g \in H \setminus \tilde{G}_{\bf 0, 0'}$. Then $g = (a_1 e_1 + a_2 e_2) \cdot A$ with 
$a_1 \in K \setminus \{0\}, a_2 \in K, A \in B$, so $g = (a_2 e_2) \cdot (a_1 e_1) \cdot A \in C(a_1 e_1)$; hence, 
we may assume without loss that $g = a e_1$ with $a \in K \setminus \{0\}$. We obtain 
\[
g = a e_1 = (R(a) Z(a)^{- 1}) \cdot V \cdot (R(- a)^{- 1} Z(- a)) \in C(V),
\] 
so $H \setminus \tilde{G}_{\bf 0, 0'} = C(V)$.

(ii) $g \in \tilde{G}_{\bf 0} \setminus H$. Then $g = (a_1 e_1 + a_2 e_2) \cdot A$ with 
$a_i \in K, A \in GL_2(K) \setminus B$. As $g \in \tilde{G}_{\bf 0, 0'} ((a_1 e_1) \cdot A)$, we may assume without 
loss that $g = (a e_1) \cdot A$ with $a \in K, A \in GL_2(K) \setminus B$. Using the normal form for matrices in 
$GL_2(K)$ (cf.\ Lemma~\ref{Lem:BGB} for $\alpha = 0$), we obtain $A = E_{1, 2}(b) J A'$ for some $b \in K, A' \in B$, 
and hence we may assume that $g = (a e_1) \cdot E_{1, 2}(b) \cdot J$. It follows that
\begin{align*}
g &= E_{1, 2}(b) \cdot (E_{1, 2}(- b) \cdot (a e_1) \cdot E_{1, 2}(b)) \cdot J \\
&= E_{1, 2}(b) \cdot (a e_1 + a b e_2) \cdot J \\
&= (E_{1, 2}(b) \cdot (a b e_2)) \cdot J \cdot (a e_2),
\end{align*}
therefore $g \in C(J)$. Thus $\tilde{G}_{\bf 0} \setminus H = C(J)$, and the statement (1) is proved.

(2) We obtain 
\begin{align*}
C(V) \cdot C(V) &= C(1) \sqcup C(V),\\
C(J) \cdot C(J) &= C(1) \sqcup C(V) \sqcup C(J),\\ 
C(V) \cdot C(J) &= C(J) \cdot C(V) = C(J),
\end{align*} 
as desired.
\ep

\subsection{A presentation of the group $SL_3$ over a field} 
\label{SubSubSec:SL3Pres}
We have now collected together all the necessary ingredients to establish an explicit presentation for the group $G = SL_3(K)$ in terms
of the group $G_{\bf 0}$, with presentation given by Proposition~\ref{Prop:G_0Pres}, and the involutionary extra
generator $T$, defined by (\ref{Eq:T}) at the beginning of Section~\ref{Sec:SL3}.

Replacing the matrix $A$ in (\ref{Eq:G/G0RelI}) with the generators of the group $G_{\bf 0, 0'}$, whose presentation 
is given by Proposition~\ref{Prop:G00Pres}, the type (I) relations take the form
\begin{align}
[\hat{Z}(a), T] &= \hat{R}(a)^3,\label{Eq:SL3/G0RelI1}\\[1mm]
(T \hat{R}(a))^2 &= 1,\label{Eq:SL3/G0RelI2}\\[1mm]
T \hat{S} T &= \hat{W},\label{Eq:SL3/G0RelI3}
\end{align}
where $a \in K^\times \setminus \{1\}$.

Next, using the description of the space of double cosets $G_{\bf 0, 0'}\backslash G_{\bf 0}/G_{\bf 0, 0'}$ given in 
Lemma~\ref{Lem:G00G0G00}, we obtain the type (II) relations 
\begin{align}
T^2 &= 1,\label{Eq:SL3/G0RelII1}\\[1mm]
(T \hat{J})^3 &= 1,\label{Eq:SL3/G0RelII2}\\[1mm]
(T \hat{V})^3 &= 1,\label{Eq:SL3/G0RelII3}
\end{align}
where $\hat{V} = \hat{J} \hat{W} \hat{J}$.

Using the relations (\ref{Eq:SL3/G0RelI1})--(\ref{Eq:SL3/G0RelI3}), (\ref{Eq:SL3/G0RelII1})--(\ref{Eq:SL3/G0RelII3}) 
above, and implementing the presentation of $G_{\bf 0}$ given by Proposition~\ref{Prop:G_0Pres}, we finally obtain
the following presentation for the group $SL_3(K)$ over an arbitrary field $K$.

\begin{te}
\label{Thm:SL3Pres}
$SL_3(K)$ is generated by the matrices
\begin{multline*}
\hat{R}(a) = \begin{pmatrix}a&0&0\\0&1&0\\0&0&a^{-1}\end{pmatrix}\,\,\,(a \in K^\times \sm \{1\}),\, 
\hat{S} = \begin{pmatrix}-1&1&0\\0&1&0\\0&0&-1\end{pmatrix},\\
\hat{J} = \begin{pmatrix}0&-1&0\\-1&0&0\\0&0&-1\end{pmatrix}.\,\mbox{ and }\,
T = \begin{pmatrix}0&0&-1\\0&-1&0\\-1&0&0\end{pmatrix},
\end{multline*}
subject to the following relations: 
\begin{enumerate}
\item[(i)] $\hat{S}^2 = \hat{J}^2 = T^2 = 1,$
\vspace{1,5mm}
\item[(ii)] $\hat{R}(a) \hat{R}(b) = \hat{R}(ab),\quad(a,b\in K^\times\backslash\{1\}),$ where 
$\hat{R}(1):=1,$
\vspace{1.5mm}
\item[(iii)] $\hat{Z}(a) \hat{Z}(b) = \hat{Z}(ab),\quad (a, b\in K^\times\backslash\{1\}),$ where 
$\hat{Z}(a):= (\hat{J}\hat{R}(a))^2$\\ for $a\in K^\times\backslash\{1\},$
\vspace{1.5mm}
\item[(iv)] $[\hat{R}(a), \hat{Z}(b)] = 1,\quad(a, b\in K^\times\backslash\{1\}),$
\vspace{1.5mm}
\item[(v)] $[\hat{S}, \hat{Z}(a)] = 1,\quad(a\in K^\times\backslash\{1\}),$
\vspace{1.5mm}
\item[(vi)] $[\hat{S}, \hat{R}(a)] = \hat{E}_{1, 2}(1-a),\quad(a\in K^\times\backslash\{1\}),$ where 
$\hat{E}_{1, 2}(a):=\hat{R}(a) \hat{S} \hat{R}(-a)^{-1}$ for $a\in K^\times\backslash\{1\},$
\vspace{1.5mm}
\item[(vii)] $(\hat{J}\hat{S})^3 = (T\hat{J})^3 = 1,$
\vspace{1.5mm}
\item[(viii)] $(T \hat{R}(a))^2 = 1,\quad (a\in K^\times\backslash\{1\}),$
\vspace{1.5mm}
\item[(ix)] $[\hat{Z}(a), T] = \hat{R}(a)^3,\quad (a\in K^\times\backslash\{1\}),$
\vspace{1.5mm}
\item[(x)] $(\hat{W}\hat{E}_{1, 2}(a))^2 = 1,\quad(a\in\mathcal{R}),$ where 
$\hat{W}:= T \hat{S} T,$ and $1 \in \mathcal{R}\subseteq K^\times$ is a set of representatives for the quotient group 
$K^\times/(K^\times)^3,$
\vspace{1.5mm}
\item[(xi)] $\hat{W} \hat{J}\hat{R}(a) \hat{W} = \rho_a \hat{W} \lambda_a,\quad 
(a\in\mathcal{R}),$ where $\rho_a:= \hat{E}_{1, 2}(-a) \hat{J} \hat{R}(a)$ and 
$\lambda_a:= \hat{E}_{1, 2}(-a^{-2})\hat{J}\hat{E}_{1, 2}(a^2) \hat{R}(a)^4 \hat{Z}(a)^{-2}$.
\end{enumerate}
\end{te} 

\begin{proof}
Taking (\ref{Eq:SL3/G0RelI3}) as definition for $\hat{W}$, it follows by Corollary~\ref{Cor:Main2Trans}
that $SL_3(K)$ is generated by $\hat{R}(a)\,(a \in K^\times \setminus \{1\})$, $\hat{S}, \hat{J}$, and $T$
subject to the relations (\ref{Eq:G00Rel1})--(\ref{Eq:G00Rel6}), (\ref{Eq:LiftedJ,JSOrder}), 
(\ref{Eq:G00/BTI1})--(\ref{Eq:G00/BTII1}), and (\ref{Eq:G0TII3}) from the presentation of $G_{\bf 0}$  
plus the relations (\ref{Eq:SL3/G0RelI1}), (\ref{Eq:SL3/G0RelI2}), (\ref{Eq:SL3/G0RelII1})--(\ref{Eq:SL3/G0RelII3})
involving the extra generator $T$. Note that the set of relations just specified consists of the relations (i)--(xi) plus the relations (\ref{Eq:G00/BTI1}), (\ref{Eq:G00/BTII1}), and (\ref{Eq:SL3/G0RelII3}); hence, 
in order to complete the proof, it remains to check that the last three relations are consequences of relations (i)--(xi). The first two relations follow easily by conjugating the relations (\ref{Eq:G00Rel4}) and (\ref{Eq:G00Rel5})
respectively by the involution $T$. To deduce (\ref{Eq:SL3/G0RelII3}), we use the identity 
$\hat{J} T \hat{J} = T \hat{J} T$ relating the involutions $\hat{J}$ and $T$, cf.\ (\ref{Eq:SL3/G0RelII2}), 
and the definitions of the involutions $\hat{W}$ and $\hat{V}$ to obtain
\[
T \hat{V} = T \hat{J} \hat{W} \hat{J} = T \hat{J} T \hat{S} T \hat{J} = (\hat{J} T) (\hat{J} \hat{S}) (\hat{J} T)^{- 1}.
\]
Thus, $T \hat{V}$ is a conjugate of $\hat{J} \hat{S}$, and hence $(T \hat{V})^3 = 1$ by (\ref{Eq:LiftedJ,JSOrder})
as desired.  
\end{proof}

\begin{co}
\label{Cor:SL3(q)Pres}
Let $\zeta$ be a generator of the cyclic multiplicative group $K^\times$ of order $q-1$. Then $SL_3(q)$ is generated by the matrices $\hat{R}:=\hat{R}(\zeta),$ $\hat{S},$ $\hat{J},$ and $T$ (the last three as in Theorem~{\em {\ref{Thm:SL3Pres}}),} subject to the relations
\begin{enumerate}
\item[(I)] $\hat{R}^{q-1} = \hat{Z}^{q-1} = \hat{S}^2 = \hat{J}^2 = T^2 = (\hat{J}\hat{S})^3 = (T\hat{J})^3 = (T\hat{R})^2 = 1,$\\ where  $\hat{Z}:= (\hat{J}\hat{R})^2;$

\vspace{2mm}

\item[(II)] $[\hat{R}, \hat{Z}] = [\hat{S}, \hat{Z}] = 1,\,\, [\hat{Z}, T] = \hat{R}^3,\, \mbox{ and }\, [\hat{S}, \hat{R}^i] = \hat{E}_{1,2}(\zeta^{\gamma(i)}),$ where
\[
\hat{E}_{1,2}(\zeta^j) = \begin{cases}
\hat{R}^j \hat{S} \hat{R}^{-j},& q\equiv 0\mod{2},\\
\hat{R}^j \hat{S} \hat{R}^{-j-\frac{q-1}{2}},& q\equiv 1\mod{2}\end{cases}
\] 
and
\[
1 - \zeta^i = \zeta^{\gamma(i)},\quad 1\leq i\leq q-2;
\]
\vspace{2mm}
\item[(III)] $(\hat{W} \hat{E}_{1,2}(a))^2 = 1$ for $a\in\mathcal{R},$ where $\hat{W}:= T\hat{S}T,$ and where $\mathcal{R}\subseteq K^\times$ is a set of representatives for the elements of the quotient group $K^\times/(K^\times)^3$ with $1\in\mathcal{R};$ 

\vspace{2mm}

\item[(IV)] $\hat{W} \hat{J} \hat{R}^i \hat{W} = \rho_i \hat{W} \lambda_i$ for $a=\zeta^i\in\mathcal{R},$ where $\rho_i:= \hat{E}_{1,2}(-a) \hat{J} \hat{R}^i$ and $\lambda_i:= \hat{E}_{1,2}(-a^{-2}) \hat{J} \hat{E}_{1,2}(\zeta^{2i}) \hat{R}^{4i} \hat{Z}^{-2i}$.
\end{enumerate}
\end{co}

\begin{co}
\label{Cor:PSL3Pres}
Let $K$ be a field. Then the projective linear group $PSL_3(K)$ has a presentation with generators 
$\hat{R}(a)$ for $a\in K^\times\backslash\{1\}$, $\hat{S}$, $\hat{J}$, and $T$, subject to the relations 
{\em (i)--(xi)} from Theorem~{\em \ref{Thm:SL3Pres},} plus one extra relation 
$(\hat{J}\,\hat{R}(\omega))^2=1$ in case that $K$ contains a primitive third root of unity $\omega$. 
\end{co}

In preparation for our work on the large Mathieu groups in Section~\ref{Subsec:LargeMathieu}, we derive from Corollary~\ref{Cor:PSL3Pres} an explicit presentation for the group $PSL_3(4)$.
\begin{co}
\label{Cor:PSL34Pres}
The group $PSL_3(4)$ is generated by symbols $\s,$ $\mathfrak{j},$ $\mathfrak{t},$ $\mathfrak{r}$ induced, respectively, by the $SL_3(4)$-matrices $\hat{S}={\tiny \begin{pmatrix}1&1&0\\0&1&0\\0&0&1\end{pmatrix}}$,\, $\hat{J} = {\tiny \begin{pmatrix}0&1&0\\1&0&0\\0&0&1\end{pmatrix}}$, \, $T = {\tiny \begin{pmatrix}0&0&1\\0&1&0\\1&0&0\end{pmatrix}},$ and $\hat{R}(\zeta) = {\tiny \begin{pmatrix}\zeta&0&0\\0&1&0\\0&0&\zeta^{-1}\end{pmatrix}},$ where $\zeta$ is a primitive element of the field $K=GF(4),$ with defining relations
\begin{multline*}
\s^2 = \mathfrak{j}^2 = \mathfrak{t}^2 = \mathfrak{r}^3 = (\mathfrak{j}\mathfrak{r})^2 = (\mathfrak{t}\mathfrak{r})^2 = (\s\mathfrak{r})^3 = (\mathfrak{j}\s)^3 = (\mathfrak{t}\mathfrak{j})^3 = (\mathfrak{t}\s\mathfrak{t}\mathfrak{j}\s)^3 = 1,\\[1mm] 
\mathfrak{t}\s\mathfrak{t}\cdot\mathfrak{j}\mathfrak{r}\cdot \mathfrak{t}\s\mathfrak{t} = \mathfrak{r}\s\mathfrak{r}^{-1}\mathfrak{j}\mathfrak{r} \cdot \mathfrak{t}\s\mathfrak{t}\cdot \mathfrak{r}\s\mathfrak{r}^{-1}\mathfrak{j}\mathfrak{r}^{-1}\s\mathfrak{r}^{-1}.
\end{multline*} 
\end{co}

\begin{proof}
The group $PSL_3(4)$ is generated by symbols $\s, \mathfrak{j}, \mathfrak{t},$ $\mathfrak{r}$ (the images of the matrices $\hat{S}$, $\hat{J}$, $T$, and $\hat{R}(\zeta)$ under the canonical projection $SL_3(4)\rightarrow PSL_3(4)$) modulo the $SL_3$-relations (I)--(IV) from Corollary~\ref{Cor:SL3(q)Pres} plus the $PSL$-relation $(\mathfrak{j}\mathfrak{r})^2=1$ from Corollary~\ref{Cor:PSL3Pres}. 

Relations (I) from Corollary~\ref{Cor:SL3(q)Pres} together with the $PSL$-relation $(\mathfrak{j}\mathfrak{r})^2=1$ lead to the relations
\begin{equation}
\label{Eq:PSL3(4)Pres1}
\mathfrak{r}^3 = \s^2 = \mathfrak{j}^2 = \mathfrak{t}^2 = (\mathfrak{j}\s)^3 = (\mathfrak{t}\mathfrak{j})^3 = (\mathfrak{t}\mathfrak{r})^2 = (\mathfrak{j}\mathfrak{r})^2 = 1.
\end{equation}
Second, since $\hat{Z}=1$, $1-\zeta=\zeta^2$, and $1-\zeta^2=\zeta$ in our situation, (II) boils down to the two relations $(\s\mathfrak{r})^3=1$ and $(\s\mathfrak{r}^{-1})^3=1$, with the second one following from the first plus the relations $\mathfrak{r}^3 = \s^2=1$:
\begin{equation*}
(\s\mathfrak{r}^{-1})^3 = \s\mathfrak{r}\s \cdot \s\mathfrak{r}\s \cdot \mathfrak{r}^{-1}\s \mathfrak{r}^{-1} = \mathfrak{r}^{-1} \s \mathfrak{r}^{-1} \cdot \mathfrak{r}^{-1} \s \mathfrak{r}^{-1}\cdot \mathfrak{r}^{-1} \s \mathfrak{r}^{-1} = \mathfrak{r}^{-1} \s \mathfrak{r} \s \mathfrak{r} \cdot \s\mathfrak{r}^{-1} = \mathfrak{r} \cdot \s \cdot \s\mathfrak{r}^{-1} = 1.
\end{equation*}
Hence, we obtain the one new relation 
\begin{equation}
\label{Eq:PSL3(4)Pres2}
(\s\mathfrak{r})^3=1.
\end{equation}
Third, since $\mathcal{R} = \{1, \zeta, \zeta^2\}$ in our present situation, Relations (III) take the form 
\begin{align}
(\mathfrak{t} \s)^4 & = 1 \mbox{ (for $a=1$), }\label{Eq:PSL3(4)Pres3}\\ 
[\mathfrak{t}\s\mathfrak{t}, \mathfrak{r}\s\mathfrak{r}^{-1}] &= 1 \mbox{ (for $a=\zeta$), }\label{Eq:PSL3(4)Pres4} \\
[\mathfrak{t}\s\mathfrak{t}, \mathfrak{r}^{-1} \s \mathfrak{r}] & = 1 \mbox{ (for $a=\zeta^2$).}\label{Eq:PSL3(4)Pres5}
\end{align} 
From (IV), we obtain three relations of the form
\[
\mathfrak{t}\s\mathfrak{t}\cdot \mathfrak{j} \cdot \mathfrak{r}^i\cdot \mathfrak{t}\s\mathfrak{t} = \rho_i \cdot \mathfrak{t}\s\mathfrak{t}\cdot \lambda_i,\quad(a=\zeta^i\in\mathcal{R}).
\]
Computing the quantities $\rho_i, \lambda_i$ occurring in these equations, we obtain the explicit relations
\begin{align}
\mathfrak{t}\s\mathfrak{t} \cdot \mathfrak{j} \cdot \mathfrak{t}\s\mathfrak{t} & = \s\mathfrak{j} \cdot \mathfrak{t}\s\mathfrak{t} \cdot \s\mathfrak{j}\s \mbox{ (for $a = 1$)},\label{Eq:PSL3(4)Pres6}\\
\mathfrak{t} \s \mathfrak{t} \cdot \mathfrak{j} \mathfrak{r} \cdot \mathfrak{t}\s\mathfrak{t} &= \mathfrak{r} \s \mathfrak{r}^{-1} \mathfrak{j} \mathfrak{r} \cdot \mathfrak{t}\s\mathfrak{t} \cdot \mathfrak{r}\s\mathfrak{r}^{-1} \mathfrak{j} \mathfrak{r}^{-1} \s \mathfrak{r}^{-1} \mbox{ (for $a = \zeta$)},\label{Eq:PSL3(4)Pres7}\\
\mathfrak{t} \s \mathfrak{t} \cdot \mathfrak{j}\mathfrak{r}^{-1} \cdot \mathfrak{t}\s\mathfrak{t} & = \mathfrak{r}^{-1} \s \mathfrak{r} \mathfrak{j} \mathfrak{r}^{-1} \cdot \mathfrak{t}\s\mathfrak{t} \cdot \mathfrak{r}^{-1}\s\mathfrak{r}\mathfrak{j}\mathfrak{r}\s\mathfrak{r} \mbox{ (for $a = \zeta^2$)}. \label{Eq:PSL3(4)Pres8}  
\end{align} 
Next, we note that, on the basis of the earlier relations 
\[
\s^2 = \mathfrak{j}^2 = \mathfrak{t}^2 = (\mathfrak{t}\s)^4 = 1, 
\]
Relation~(\ref{Eq:PSL3(4)Pres6}) is equivalent to 
\begin{equation}
\label{Eq:PSL3(4)Pres9}
(\mathfrak{t}\s\mathfrak{t}\mathfrak{j}\s)^3 = 1.
\end{equation}
Indeed, we have
\[
1 = \mathfrak{t}\s\mathfrak{t}\mathfrak{j}\mathfrak{t}\s\mathfrak{t}\s\mathfrak{j}\s\mathfrak{t}\s\mathfrak{t}\mathfrak{j}\s = \mathfrak{t}\s\mathfrak{t}\mathfrak{j}\s \cdot \s\mathfrak{t}\s\mathfrak{t}\s \cdot \mathfrak{j}\s\mathfrak{t}\s\mathfrak{t}\mathfrak{j}\s = \mathfrak{t}\s\mathfrak{t}\mathfrak{j}\s \cdot \mathfrak{t}\s\mathfrak{t}\mathfrak{j}\s \cdot \mathfrak{t}\s\mathfrak{t}\mathfrak{j}\s = (\mathfrak{t}\s\mathfrak{t}\mathfrak{j}\s)^3. 
\]
We have thus obtained a presentation for the group $PSL_3(4)$ with generators $\s, \mathfrak{j}, \mathfrak{t}, \mathfrak{r}$ and defining relations (\ref{Eq:PSL3(4)Pres1})--(\ref{Eq:PSL3(4)Pres5}) and (\ref{Eq:PSL3(4)Pres7})--(\ref{Eq:PSL3(4)Pres9}). Finally, one checks, for instance using the computer algebra system GAP \cite{GAP}, that the relations (\ref{Eq:PSL3(4)Pres3})--(\ref{Eq:PSL3(4)Pres5}), as well as (\ref{Eq:PSL3(4)Pres8}), are consequences of the remaining relations. We thus arrive at the presentation for $PSL_3(4)$ given in Corollary~\ref{Cor:PSL34Pres}.
\end{proof}

\section{Deformations of groups and presentations}
\label{Sec:Deform}

The machinery of group deformations was originally conceived by the first author as a useful tool for studying some specific profinite group actions arising in the framework of co-Galois theory, and was subsequently developed further in its own right by the authors; cf.\ \cite{CGA}, \cite{GFrame}, and \cite{BM1}.\footnote{Co-Galois theory, which has its roots in classical work by Mordell \cite{Mo}, Siegel \cite{CLS}, 
Kneser \cite{K}, and Schinzel \cite{Sch}, investigates  radical field extensions, being to some extent dual to classical Galois theory. For further developments of the theory see the monograph \cite{Albu} and the references therein, as well as the more recent papers \cite{AB}, \cite{KHK}, \cite{CGA}, \cite{GFrame} on abstract co-Galois theory.} 

The true power of deformation theory is best revealed in a topological setting; our aim in this section however is somewhat more modest: to describe some of the more elementary (discrete) aspects of deformation theory, and to demonstrate by means of an 
example how the technique of group deformation interacts with (and thereby enhances) the 
presentation method of Theorem~\ref{Thm:PresGpAct}. 

\subsection{Deformations of a group induced by actions on itself}
Let $G$ be a group, and let $\varphi: G\rightarrow \mathrm{Aut}(G)$ be an action by automorphisms of $G$ on itself. With the pair $(G,\varphi)$ we associate a binary operation $\circ_\varphi$ on the underlying set of $G$, given by
\[
g_1\circ_\varphi g_2:= g_1\cdot \varphi(g_1)^{-1}(g_2),\quad(g_1, g_2\in G).
\]
One checks that the identity element $1$ of $G$ is also a two-sided identity element with respect to the operation $\circ_\varphi$, and that the element $I(g):= \varphi(g)(g)^{-1}$ is a right inverse of $g$ with respect to $\circ_\varphi$. However, the binary operation $\circ_\varphi$ on $G$ need not, in general, be associative. The map
\[
\Phi: \mathcal{A}(G):= \mathrm{Hom}(G, \mathrm{Aut}(G)) \longrightarrow G^{G\times G},\quad \varphi\mapsto \circ_\varphi
\]
is clearly injective. If $\varphi$ is trivial, that is, $\varphi(g)=1_G$ for all $g\in G$, then $\circ_\varphi$ is the original group operation $\cdot$ on $G$ while, setting $\varphi(g) = \iota_g$, where $\iota_g: h\mapsto g\cdot h\cdot g^{-1}$ is the inner automorphism of $G$ induced by $g$, the operation $\circ_\varphi$ is the dual of the original group operation of $G$. In what follows, we shall denote the trivial action and the conjugation action of $G$ on itself by $\varphi_0$ and $\varphi_1$, respectively. The automorphism group $\mathrm{Aut}(G)$ of $G$ acts naturally from the left on the set $\mathcal{A}(G)$ via
\[
({}^\theta\varphi)(g):= \theta\circ \varphi(\theta^{-1}(g))\circ \theta^{-1},\quad (\theta\in\mathrm{Aut}(G),\, \varphi\in\mathcal{A}(G),\, g\in G),
\]
and this action is compatible with the map $\Phi$ introduced above in the sense that
\begin{equation}
\label{Eq:AdStable}
\theta(g_1\circ_\varphi g_2) = \theta(g_1) \circ_{{}^\theta\varphi} \theta(g_2),\quad (\theta\in\mathrm{Aut}(G),\, \varphi\in\mathcal{A}(G),\, g_1, g_2\in G).
\end{equation}
For $\varphi\in\mathcal{A}(G)$, we denote by $\mathrm{Stab}_\varphi$ the stabilizer of $\varphi$ in $\mathrm{Aut}(G)$; that is,
\[
\mathrm{Stab}_\varphi = \big\{\theta\in\mathrm{Aut}(G):\, {}^\theta\varphi = \varphi\big\}.  
\]
By Formula (\ref{Eq:AdStable}) plus injectivity of the map $\Phi$, we find that
\begin{equation}
\label{Eq:StabForm}
\mathrm{Stab}_\varphi = \big\{\theta\in\mathrm{Aut}(G):\,\forall\,g_1, g_2\in G,\, \theta(g_1\circ_\varphi g_2) = \theta(g_1) \circ_\varphi \theta(g_2)\big\}.
\end{equation}

Our first result states criteria for the algebraic system $G_\varphi:= (G, \circ_\varphi)$ associated with an action $\varphi$ by automorphisms of $G$ on itself to be a group.
\begin{lem}
\label{Lem:AssCrit}
Let $\varphi\in\mathcal{A}(G)$. Then the following assertions are equivalent:
\vspace{-2mm}

\begin{enumerate}
\item[(i)] $G_\varphi$ is a group, with neutral element $1,$ and with $I(g)$ as inverse of $g\in G_\varphi;$
\vspace{2mm}

\item[(ii)] the binary operation $\circ_\varphi$ is associative;
\vspace{2mm}

\item[(iii)]  the identity  
$\varphi(\varphi(g^{-1})(h)) = \varphi(g^{-1}hg)$ 
holds for all $g, h\in G;$
\vspace{2mm}

\item[(iv)] we have $\varphi(G) \subseteq \mathrm{Stab}_\varphi;$
\vspace{2mm}

\item[(v)] $\varphi(G)$ is a normal subgroup of $\mathrm{Stab}_\varphi$.
\end{enumerate}
\end{lem}
\textit{Proof} (sketch). The equivalences (i) $\Leftrightarrow$ (ii) and (iv) $\Leftrightarrow$ (v) are clear, while the equivalence (iii) $\Leftrightarrow$ (iv) follows from the fact that, by definition of the action of $\mathrm{Aut}(G)$ on the complex $\mathcal{A}(G)$, ${}^{\varphi(g)}\varphi = \varphi$ for fixed $g\in G$ is equivalent to $\varphi(\varphi(g^{-1})(h)) = \varphi(g^{-1}hg)$ for all $h\in G$. Thus, it remains to establish the equivalence of Conditions (ii) and (iii), which follows by direct computation.
\hfill $\Box$\par

We call a homomorphism $\varphi: G\rightarrow \mathrm{Aut}(G)$ \textit{admissible}, if it satisfies the equivalent conditions of Lemma~\ref{Lem:AssCrit}. The (non-empty) subset of $\mathcal{A}(G)$ consisting of all admissible  actions of $G$ on itself is denoted by $\mathcal{A}_{\mathrm{ad}}(G)$. By (\ref{Eq:AdStable}), $\mathcal{A}_{\mathrm{ad}}(G)$ is stable under the action of $\mathrm{Aut}(G)$. The set of global fixed points of the action of $\mathrm{Aut}(G)$ on $\mathcal{A}(G)$,
\[
\mathcal{A}(G)^0 = \big\{\varphi\in \mathcal{A}(G):\, \mathrm{Stab}_\varphi = \mathrm{Aut}(G)\big\} \subseteq \mathcal{A}_{\mathrm{ad}}(G)
\]
is non-empty, since it contains $\varphi_0$ and $\varphi_1$. Moreover, for every fixed point $\varphi\in\mathcal{A}(G)^0$, $\varphi(G)$ is a normal subgroup of $\mathrm{Aut}(G)$, and $\mathrm{Aut}(G)$ is a subgroup of $\mathrm{Aut}(G_\varphi)$.

\subsection{Admissible group actions}
Let $\varphi: G\rightarrow \mathrm{Aut}(G)$ be an admissible action of $G$ on itself. By Lemma~\ref{Lem:AssCrit} and Formula~(\ref{Eq:StabForm}), the image $\varphi(G)$ acts naturally by automorphisms on both $G$ and $G_\varphi$, so we can form the semi-direct products $G\ltimes \varphi(G)$ and $G_\varphi\ltimes \varphi(G)$ with respect to these actions. In both cases, the underlying set is the cartesian product $G\times \varphi(G)$, with multiplication in $G\ltimes \varphi(G)$  given by the rule
\[
(g_1, \theta_1) (g_2, \theta_2) = (g_1\cdot \theta_1(g_2), \theta_1\circ \theta_2),
\]
and by
\[
(g_1,\theta_1) (g_2, \theta_2) = (g_1\circ_\varphi \theta_1(g_2), \theta_1\circ \theta_2) = (g_1\cdot \varphi(g_1^{-1})(\theta_1(g_2)), \theta_1 \circ  \theta_2)
\]
in $G_\varphi \ltimes \varphi(G)$. Our next result, whose proof is  straightforward, states some further properties of admissible actions.

\begin{lem}
\label{Lem:AdeqProps}
Let $\varphi: G\rightarrow \mathrm{Aut}(G)$ be  admissible. Then the following assertions hold.

\vspace{-2mm}
\begin{enumerate}
\item[(i)] The map $\alpha: G\ltimes \varphi(G) \rightarrow G_\varphi\ltimes \varphi(G)$ given by $(g,\theta)\mapsto (g, \varphi(g)\circ \theta)$ is an isomorphism.

\vspace{2mm}
\item[(ii)] The map $\beta: G_\varphi \rightarrow \varphi(G)$ given by $g\mapsto \varphi(g)^{-1}$ is a morphism of $\varphi(G)$-groups, where $\varphi(G)$ acts on itself by conjugation.

\vspace{2mm}
\item[(iii)] The map $G\rightarrow G_\varphi$ sending $g$ to $g^{-1}$ induces an isomorphism of $\varphi(G)$-groups
\[
\gamma: G/\mathrm{ker}(\varphi) \cong G_\varphi/\mathrm{ker}(\varphi).
\]
Consequently, $G_\varphi$ is an extension of $\mathrm{ker}(\varphi)$ by $\varphi(G);$ its associated $\varphi(G)$-kernel $\kappa: \varphi(G)\rightarrow \mathrm{Out}(\mathrm{ker}(\varphi))$ being induced by the map $G \rightarrow \mathrm{Aut}(\mathrm{ker}(\varphi)),$ $g\mapsto F_g,$ where $F_g(h):= g^{-1}\cdot \varphi(g)(h)\cdot g$ for $h\in\mathrm{ker}(\varphi)$.

\vspace{2mm}
\item[(iv)] The map $G\times G_\varphi\rightarrow G_\varphi$ given by $(g,x)\mapsto {}^gx:= \varphi(g)(x)$ is a left action of $G$ on the group $G_\varphi,$ and the identity map $1: G\rightarrow G_\varphi,$ $g\mapsto g,$ is a $1$-cocycle with respect to this action; that is, we have
\begin{equation}
\label{Eq:Id1Cocycle}
g\cdot h = g\circ_\varphi {}^gh\quad(g,h\in G).
\end{equation}
\end{enumerate}
\end{lem}

\begin{co}
\label{Cor:EqImpIso}
Let $\varphi$ and $\psi$ be admissible $G$-actions, equivalent under the action of the group $\mathrm{Aut}(G)$. Then $G_\varphi \cong G_\psi$ and $\varphi(G) \cong \psi(G)$.
\end{co}
\begin{proof}
Let $\theta\in\mathrm{Aut}(G)$ be such that ${}^\theta\varphi = \psi$. By (\ref{Eq:AdStable}), $\theta$ is an isomorphism of $G_\varphi$ onto $G_\psi$. Also, for $g\in G$, we have, by definition of the action of $\mathrm{Aut}(G)$, that $g\in \mathrm{ker}(\varphi)$ if, and only if, $\theta(g)\in \mathrm{ker}(\psi)$, so that $\theta$ restricts to an isomorphism $\mathrm{ker}(\varphi) \cong \mathrm{ker}(\psi)$. Using Part~(iii) of Lemma~\ref{Lem:AdeqProps}, we obtain an isomorphism
\[
\varphi(G) \cong G_\varphi/\mathrm{ker}(\varphi) \overset{\theta}{\longrightarrow} G_\psi/\mathrm{ker}(\psi) \cong \psi(G),
\]
as required.
\end{proof}

{\footnotesize 
\begin{rem} \em
Part~(ii) of Lemma~\ref{Lem:AdeqProps} provides (weak) obstructions for a deformation to possess certain properties. More precisely, we have the following. 
\textit{Let $G$ be a group, let $\varphi: G\rightarrow \mathrm{Aut}(G)$ be an admissible action, and let $\mathfrak{P}$ be a group property inherited by quotients. Then, if $G_\varphi$ has property $\mathfrak{P},$ so does $\varphi(G)$.}
\end{rem}}
 
{\footnotesize 
\begin{prob} \em 
\label{Prob:IsoDeforms}
Obtain a characterisation of the equivalence relation $\sim$ on the set $\mathcal{A}_\mathrm{ad}(G)$ given by $\varphi \sim \psi :\Longleftrightarrow G_\varphi \cong G_\psi;$ in other words, find a necessary and sufficient condition for two admissible actions by automorphisms of $G$ on itself to lead to isomorphic deformations.
\end{prob}}

\subsection{Deformation pairs}

Given groups $G$ and $H$, we say that $(G,H)$ is a \textit{deformation pair}, if there exists $\varphi\in\mathcal{A}_{\mathrm{ad}}(G)$ such that $G_\varphi \cong H$. Trivially, if the groups $G$ and $H$ are isomorphic, then $(G,H)$ and $(H,G)$ are deformation pairs. Note also that the groups $G$ and $H$ have the same cardinality, provided $(G,H)$ is a deformation pair.

We record some important characterisations of deformation pairs.

\begin{pr}
\label{Prop:Sym} 
Let $G$ and $H$ be groups. Then the following assertions are equivalent.

\vspace{-3mm}

\begin{enumerate}
\item[(i)] $(G, H)$ is a deformation pair.

\vspace{2mm}

\item[(ii)] $(H, G)$ is a deformation pair.

\vspace{2mm}

\item[(iii)] There exist actions by automorphisms
\[
G \times H \rightarrow H, (g, h) \mapsto {}^g h,\,\,\,H \times G \rightarrow G, (h, g) \mapsto {}^h g,
\]
and a bijective $1$-cocycle $\eta \in Z^1(G, H),$ such that its inverse $\eta^{- 1} \in Z^1(H, G)$.

\vspace{2mm}

\item[(iv)] There exist an action by automorphisms $G \times H \rightarrow H, (g, h) \mapsto {}^g h$
and a bijective cocycle $\eta \in Z^1(G, H),$ such that 
\[
\eta^{- 1}({}^{g_1} \eta(g_2)) \cdot g_1 \cdot g_2^{- 1} \cdot g_1^{- 1} \in \Delta,\quad(g_1, g_2\in G),
\]
where $\Delta$ is the kernel of the action of $G$ on $H$.

\vspace{2mm}

\item[(v)] There exist an action by automorphisms $H \times G \rightarrow G, (h, g) \mapsto {}^h g$ and
a bijective cocycle $\zeta \in Z^1(H, G),$ such that 
\[
\zeta^{- 1}({}^{h_1} \zeta(h_2)) \cdot h_1 \cdot h_2^{- 1} \cdot h_1^{- 1} \in \Delta',\quad(h_1, h_2\in H),
\]
where $\Delta'$ is the kernel of the action of $H$ on $G$.
\end{enumerate}
\end{pr}

In order to establish Proposition~\ref{Prop:Sym}, as well as Corollary~\ref{Co:Lattiso} below, we need a technical result, which we state without proof.

\begin{lem}
\label{Lem:SymExtend}
Let $G$ and $H$ be groups, let $\omega_G: G\rightarrow \mathrm{Aut}(H)$ be an action by automorphisms of $G$ on $H,$ and let $\eta: G\rightarrow H$ be a bijective $1$-cocycle with respect to this action. For $g\in G$ and $h\in H,$ set ${}^g h:= \omega_G(g)(h)$. Then the following assertions are equivalent.

\vspace{-3mm}

\begin{enumerate}
\item[(i)] There exists an action by automorphisms $H\times G\rightarrow G,\,\,(h,g)\mapsto {}^h g$ of $H$ on $G,$ such that $\eta^{-1}\in Z^1(H,G)$ with respect to the  latter action.

\vspace{2mm}

\item[(ii)] The pair $(\omega_G, \eta)$ satisfies the identity
\begin{equation}
\label{Eq:SymExtendCond}
\omega_G(\eta^{-1}({}^{g_1}\eta(g_2))) = \omega_G(g_1g_2g_1^{-1}),\quad (g_1, g_2\in G).
\end{equation}
\end{enumerate} 
\end{lem}

\begin{proofof}{Proposition~{\em \ref{Prop:Sym}.}}
(i) $\Rightarrow$ (iv) and (ii) $\Rightarrow$ (v). By symmetry, it suffices to prove only one implication, say (i) $\Rightarrow$ (iv). By assumption there exist $\varphi \in \cA_{\mathrm{ad}}(G)$ and an isomorphism $\theta: G_\varphi \rightarrow H$.
Since $\varphi(G) \subseteq \mathrm{Stab}_\varphi \subseteq  \mathrm{Aut}(G_\varphi)$, it follows that the map
\[
\omega: G \longrightarrow \mathrm{Aut}(H),\, \omega(g) := \theta \circ \varphi(g) \circ \theta^{- 1}
\]
is well defined, and thus a homomorphism, defining an action by automorphisms
\[
G \times H \rightarrow H,\,\, (g, h) \mapsto {}^g h:= \omega(g)(h).
\]
With respect to this action, the map $\theta: G \rightarrow H$ becomes a bijective $1$-cocycle: for $g_1, g_2 \in G$, we have 
\[
\theta(g_1 \cdot g_2) = \theta(g_1 \circ_\varphi \varphi(g_1)(g_2)) = \theta(g_1) \cdot \theta(\varphi(g_1)(g_2)) = \theta(g_1) \cdot {}^{g_1} \theta(g_2),
\]
as required. It remains to show that $\omega(\theta^{- 1}({}^{g_1} \theta(g_2))) = \omega(g_1 g_2 g_1^{- 1})$ for $g_1, g_2 \in G$. For $g_1, g_2 \in G$, we have 
\[
\theta^{- 1}({}^{g_1} \theta(g_2)) = \theta^{- 1}(\omega(g_1)(\theta(g_2))) = \theta^{-1}((\theta\circ\varphi(g_1)\circ\theta^{-1})(\theta(g_2))) = \varphi(g_1) (g_2),
\]
and thus 
\begin{multline*}
\omega(\theta^{- 1}({}^{g_1} \theta(g_2))) = \omega(\varphi(g_1)(g_2))
= \theta \circ \varphi(\varphi(g_1)(g_2)) \circ \theta^{- 1}
= \theta \circ \varphi(g_1 g_2 g_1^{- 1}) \circ \theta^{- 1}\\ 
= \omega(g_1 g_2 g_1^{- 1}),
\end{multline*}
as desired, where we have used Part~(iii) of Lemma~\ref{Lem:AssCrit} in the third step.

(iv) $\Rightarrow$ (iii) and (v) $\Rightarrow$ (iii). Again, by symmetry, it suffices to prove one of these implications, say (iv) $\Rightarrow$ (iii); this assertion however is immediate from the implication (ii) $\Rightarrow$ (i) of Lemma~\ref{Lem:SymExtend}.

(iii) $\Rightarrow$ (i) and (iii) $\Rightarrow$ (ii). In view of  symmetry, it again suffices to prove only one implication, say
(iii) $\Rightarrow$ (i).

By hypothesis there exist actions by automorphisms
\[
G \times H \rightarrow H,\,\, (g, h) \mapsto {}^g h \,\,\mbox{ and }\,\, H \times G \rightarrow G,\,\, (h, g) \mapsto {}^h g,
\]
and a bijective $1$-cocycle $\eta : G \rightarrow H$, such that $\eta^{- 1} \in Z^1(H, G)$.
For all $h_1, h_2 \in H$, we have 
\[
h_1 \cdot h_2 = \eta(\eta^{- 1}(h_1 \cdot h_2)) 
= \eta(\eta^{- 1}(h_1) \cdot {}^{h_1} \eta^{- 1}(h_2))
= h_1 \cdot {}^{\eta^{- 1}(h_1)} \eta({}^{h_1} \eta^{- 1}(h_2)),
\]
thus ${}^{\eta^{- 1}(h_1)^{- 1}} h_2 = \eta({}^{h_1} \eta^{- 1}(h_2))$, and hence 
\begin{equation}
\label{Eq:ActCond}
\eta^{- 1}({}^{\eta^{- 1}(h_1)^{- 1}} h_2) = {}^{h_1} \eta^{- 1}(h_2).
\end{equation}
Setting $h_1 = \eta(g_1^{- 1})$ and $h_2 = \eta(g_2)$ with  elements $g_1, g_2 \in G$, it follows
that the map
\[
\varphi: G \rightarrow \mathrm{Aut}(G),\,\, \varphi(g_1)(g_2):= {}^{\eta(g_1^{- 1})} g_2
\]
is a homomorphism, defining an action by automorphisms of $G$ on itself. Indeed, under these substitutions, (\ref{Eq:ActCond}) becomes
\[
\eta^{-1}({}^{g_1}\eta(g_2)) = {}^{\eta(g_1^{-1})} g_2,\quad g_1, g_2\in G,
\]
and repeated use of this last identity yields that, for $g_1, g_1', g_2\in G$, 
\begin{multline*}
\varphi(g_1 g_1')(g_2) = \eta^{-1}({}^{(g_1 g_1')}\eta(g_2))
= \eta^{-1}({}^{g_1}({}^{g_1'}\eta(g_2)))
= \eta^{-1}({}^{g_1}\eta({}^{\eta(g_1'^{-1})} g_2)) =\\
 \eta^{-1}({}^{g_1}\eta(\varphi(g_1')(g_2)))
= \eta^{-1}(\eta({}^{\eta(g_1^{-1})}\varphi(g_1')(g_2)))
= \varphi(g_1)(\varphi(g_1')(g_2))
= (\varphi(g_1)\circ\varphi(g_1'))(g_2),
\end{multline*}
as required. On the other hand, we deduce that, for $g_1, g_2\in G$,
\[
\eta(g_1 \circ_\varphi g_2) = \eta(g_1 \cdot \varphi(g_1^{- 1})(g_2))
= \eta(g_1\cdot {}^{\eta(g_1)} g_2)
= \eta(g_1 \cdot \eta^{- 1}({}^{g_1^{- 1}} \eta(g_2))) 
= \eta(g_1) \cdot \eta(g_2),
\]
hence $\varphi \in \cA_{\mathrm{ad}}(G)$ and
$\eta: G_\varphi \rightarrow H$ is an isomorphism, so that $(G,H)$ is a deformation pair, as claimed.
\hfill $\Box$\par
\end{proofof}

\begin{co}
\label{Co:Lattiso}
Let $(G, H)$ be a deformation pair, let $G \times H \rightarrow H$ and  $H \times G \rightarrow G$ be actions by automorphisms, and let $\eta \in Z^1(G, H)$ be a bijective cocycle, such that $\eta^{- 1} \in Z^1(H, G)$. Denote by $\mathcal{L}(G)$\,(respectively $\mathcal{L}(H)$) the lattice of all subgroups of $G$\,(respectively $H$), which are stable under the action of $H$\,(respectively $G$). 
\begin{enumerate}
\item[(a)]
The maps
\[
\mathcal{L}(G) \rightarrow \mathcal{L}(H),\, U \mapsto \eta(U)\, \mbox{ and }\, \mathcal{L}(H) \rightarrow \mathcal{L}(G),\, K \mapsto \eta^{- 1}(K)
\]
are well-defined lattice isomorphisms, inverse to each other.
\vspace{2mm}

\item[(b)]
Let $\omega_G: G\rightarrow \mathrm{Aut}(H)$ be the homomorphism defining the action of $G$ on $H$. Then $\mathrm{ker}(\omega_G)\in\mathcal{L}(G)$.
\end{enumerate}
\end{co}

\begin{proof}
(a) It suffices to show that $\eta(U) \in \cL(H)$, provided $U \in \cL(G)$. Suppose that $U \in \cL(G)$, and let $u_1, u_2 \in U$ be arbitrary elements. Using the fact that $U$ is stable under the action of $H$, plus the cocycle condition for $\eta^{-1}$, we see that
\[
\eta^{- 1}(\eta(u_1) \cdot \eta(u_2)^{- 1}) = u_1\cdot {}^{\eta(u_1)}\eta^{-1}(\eta(u_2)^{-1}) = u_1 \cdot {}^{\eta(u_1) \eta(u_2)^{- 1}} (u_2^{- 1}) \in U,
\]
so that $\eta(U)$ is a subgroup of $H$. Moreover, let $g \in G$ and $u \in U$. Using the cocycle conditions for $\eta$ and $\eta^{-1}$, we obtain
\begin{multline*}
\eta^{- 1}({}^g \eta(u)) = \eta^{- 1}(\eta(g)^{- 1} \cdot \eta(g u)) 
= \eta^{-1}(\eta(g)^{-1})\cdot {}^{\eta(g)^{-1}}(gu) 
=\\ 
{}^{\eta(g)^{-1}}(g^{-1})\cdot {}^{\eta(g)^{-1}}(gu) 
= {}^{\eta(g)^{- 1}} u \in U,
\end{multline*}
thus $\eta(U)$ is stable under the action of $G$, so $\eta(U)\in \mathcal{L}(H)$, as desired.

(b) Let $\Delta:= \mathrm{ker}(\omega_G)$. First note that $\eta(\Delta)$ is a subgroup of $H$, since the restriction map $\eta\vert_\Delta: \Delta\rightarrow H$ is a homomorphism. Thus, it remains to check that $\eta(\Delta)$ is stable under the action of $G$. However, for $\delta\in\Delta$ and $g\in G$, we have
\[
\omega_G(\eta^{-1}({}^g\eta(\delta))) = \omega_G(g \delta g^{-1}) = 1_H,
\] 
since the pair $(\omega_G, \eta)$ satisfies Identity (\ref{Eq:SymExtendCond}) by our hypothesis plus the implication (i) $\Rightarrow$ (ii) of Lemma~\ref{Lem:SymExtend}. It follows that $\eta(\Delta)\in \mathcal{L}(H)$, so $\Delta\in\mathcal{L}(G)$ by Part~(a), as claimed. 
\end{proof}

For any cardinal number $\kappa \geq 1,$ let $\mathcal{G}_\kappa$ be the set of isomorphism classes $\widehat{G}$ of groups $G$ of cardinality $\kappa$. By Proposition~\ref{Prop:Sym}, the binary relation on the class of all groups 
\[
G\,\mathrm{def}\, H :\Longleftrightarrow (G,H) \mbox{ is a deformation pair}
\] 
induces a reflexive and symmetric relation on the set $\mathcal{G}_\kappa$. The associated graph $\Gamma_\kappa$ has $\mathcal{G}_\kappa$ as set of vertices, while the geometric edges are pairs $(\widehat{G}, \widehat{H})$
with $\widehat{G} \neq \widehat{H}$ and $\widehat{G}\, \mathrm{def} \,\widehat{H}$. For instance, $\Gamma_4$ and $\Gamma_6$ each consist of a segment, that is, two vertices connected by an edge. A more interesting example is afforded by the graph $\Gamma_8$: we claim that $\Gamma_8$ consists of a triangle having as vertices the isomorphism classes of the cyclic group $C_8$, the dihedral group
$\mathbb{D}_8 \cong C_4 \ltimes C_2 \cong (C_2 \times C_2) \ltimes C_2$, and the group $\mathfrak{Q}$ of quaternions, 
and a segment having as vertices the
isomorphism classes of the abelian groups $C_4 \times C_2$ and $C_2 \times C_2 \times C_2$.

To see that $(\mathbb{D}_8, \mathfrak{Q})$ is a deformation pair, we let 
\[
\mathbb{D}_8 = \big\langle x, y\,\big\vert\, x^2 = y^4 = (x y)^2 = 1 \big\rangle,
\]
consider the admissible homomorphism
$\varphi: \mathbb{D}_8 \rightarrow \mathrm{Aut}(\mathbb{D}_8)$ given by $\varphi(x)(x) = y^2 x$, $\varphi(x)(y) = y$, and $\varphi(y)=1_{\mathbb{D}_8}$,
and check that $(\mathbb{D}_8)_\varphi \cong \mathfrak{Q}$.

In order to see that $(C_4 \times C_2, C_2 \times C_2 \times C_2)$ is a deformation pair, we let 
\[
G:= \big\langle \sigma, \tau\,\big\vert\,\sigma^4 = \tau^2 = [\sigma, \tau] = 1 \big\rangle\, \mbox{ and }\, H:= \bigoplus_{1 \leq i \leq 3} \mathrm{GF(2)} e_i,
\]
consider the action by automorphisms of $G$ on $H$ given by the homomorphism
\[
\omega: G \rightarrow \mathrm{Aut}(H),\,\, \sigma \mapsto (e_1, e_2),\, \tau \mapsto 1_H,
\]
with $\mathrm{Ker}(\omega) = \langle \sigma^2, \tau \rangle \cong C_2 \times C_2$, and check that the bijective $1$-cocycle
\[
\eta: G \rightarrow H,\,\, \sigma \mapsto e_1,\, \tau \mapsto e_3
\]
satisfies Condition (iv) of Proposition~\ref{Prop:Sym}.

{\footnotesize
\begin{prob} \em
\label{Pr:Equiv}
Is the deformation relation $\mathrm{def}$ transitive, that is, an equivalence relation? If not, describe the set of those natural numbers $n$ for which the deformation relation is an equivalence relation on $\mathcal{G}_n$.
\end{prob}}

A group $G$ is termed {\em rigid}, if $G \cong H$ provided $(G, H)$ is a deformation pair;
equivalently, $G_\varphi \cong G$ for all $\varphi \in \mathcal{A}_{\mathrm{ad}}(G)$. It can be shown that the cyclic groups $C_{p^n}$, where $n\in \mathbb{N}$ and $p$ is an odd prime number, are rigid; by contrast, the cyclic groups $C_{2^n}$ are not rigid if $n\geq2$. It is not hard to see that all simple groups (finite or infinite) are rigid.

{\footnotesize 
\begin{prob} \em
\label{Pr:Rigid}
Provide a characterisation of rigid groups.
\end{prob}}

{\footnotesize
\begin{rem} \em
\label{Re:TopFrame}
The concepts and results discussed above naturally extend to topological groups by adding suitable continuity 
conditions for actions and maps. In fact, they were first considered in the framework of abstract co-Galois 
theory. In particular, certain remarkable deformation pairs of profinite groups are naturally induced by strongly 
co-Galois actions; see \cite[Sec.~5]{CGA} and \cite[Sec~4.4]{BM1}, while other applications of the deformation machinery are given in \cite{GFrame}.
\end{rem}}

\subsection{The Zassenhaus group $\mathrm{M}(q^2)$ as deformation of $PGL_2(q^2)$}

For an odd prime $p$ and an integer $m\geq1$, let $q=p^m$, and let $K=GF(q^2)$ be the finite field of order $q^2$. 
Furthermore, let $\alpha$ be the automorphism of order $2$ of the field $K$, that is, the automorphism of $K$ given by 
$\alpha(\omega) = \omega^{q}$; $\alpha$ is the $m$-th power of the Frobenius automorphism of $K$, and the generator 
of the Galois group of the quadratic field extension $K\, \vert\, GF(q)$. By slight abuse of notation, we also denote 
by $\alpha$ the automorphism of the multiplicative groups $K^\times \cong C_{q^2-1}$ and $GL_2(K)$ induced by the 
automorphism $\alpha$ of the field $K$, and by $\bar{\alpha}$ the corresponding automorphism induced on $PGL_2(K)$. 
Moreover, the field automorphism $\alpha$ is extended to an involution of the projective line $\mathbb{P}^1(K)$ over $K$
by setting $\alpha(\infty) = \infty$. By definition, the Zassenhaus group  $\mathrm{M}(q^2)$ is the sharply $3$-transitive 
permutation group on $\mathbb{P}^1(K)$ consisting of the transformations $h$ given by
\footnote{See \cite[Chap.~XI, \S~1]{HuBl} for more information on the groups $\mathrm{M}(q^2)$.}
\[
h(\omega) = \frac{a\beta(\omega) + b}{c\beta(\omega) + d},\quad \omega\in\mathbb{P}^1(K),
\]
where $a,b,c,d\in K$, $\Delta := ad-bc\neq0$, and 
\[
\beta = \begin{cases}
1_{\mathbb{P}^1(K)},& \Delta\in (K^\times)^2,\\
\alpha,& \Delta\not\in (K^\times)^2.
\end{cases}
\] 
Let $\varphi: GL_2(K) \rightarrow \mathrm{Aut}(GL_2(K))$ 
be the homomorphism obtained by composing the surjective homomorphism $GL_2(K) \rightarrow K^\times/(K^\times)^2$ given by $A\mapsto \mathrm{det}(A) (K^\times)^2$ with the embedding of the cyclic group $K^\times/(K^\times)^2 \cong C_2$ into $\mathrm{Aut}(GL_2(K))$, with image $\langle \alpha\rangle$. Explicitly, $\varphi$ is given by
\[
\varphi(A) = \left.\begin{cases}
             1_{GL_2(K)},&\mathrm{det}(A)\in (K^\times)^2\\[2mm]
             \alpha,& \mathrm{det}(A)\not\in (K^\times)^2
             \end{cases}\right\}\quad(A\in GL_2(K)).
\]
Similarly, composing the well-defined map $PGL_2(K) \cong GL_2(K)/K^\times  \rightarrow K^\times/(K^\times)^2$ given by $AK^\times\mapsto \mathrm{det}(A) (K^\times)^2$ with the embedding of $K^\times/(K^\times)^2$ into $\mathrm{Aut}(PGL_2(K))$ with image $\langle \bar{\alpha}\rangle$, we obtain a homomorphism 
$\bar{\varphi}: PGL_2(K) \rightarrow \mathrm{Aut}(PGL_2(K))$ 
given explicitly by
\[
\bar{\varphi}\Big(\omega\mapsto\frac{a\omega+b}{c\omega+d}\Big) = \left.\begin{cases}
             1_{PGL_2(K)},&\mathrm{det}({\begin{pmatrix} a&b\\c&d\end{pmatrix}})\in (K^\times)^2\\[5mm]
             \bar{\alpha},& \mathrm{det}({\begin{pmatrix} a&b\\c&d\end{pmatrix}})\not\in (K^\times)^2
             \end{cases}\right\}\quad({\begin{pmatrix} a&b\\c&d\end{pmatrix}}\in GL_2(K)).
\]
It is easily checked that, for $A\in GL_2(K)$,  
$({}^\alpha\varphi)(A) = \varphi(A)$. 
Hence, by Lemma~\ref{Lem:AssCrit}, we have  $\varphi\in\mathcal{A}_{\mathrm{ad}}(GL_2(K))$, and a similar computation shows that $\bar{\varphi}\in \mathcal{A}_{\mathrm{ad}}(PGL_2(K))$. Consequently, the actions $\varphi: GL_2(K) \rightarrow \mathrm{Aut}(GL_2(K))$ and $\bar{\varphi}: PGL_2(K) \rightarrow$\linebreak  $\mathrm{Aut}(PGL_2(K))$ give rise to deformations $G_\varphi:=(GL_2(K))_\varphi$, $\bar{G}_{\bar{\varphi}}:=(PGL_2(K))_{\bar{\varphi}}$ of $GL_2(K)$ and $PGL_2(K)$, respectively. The facts summarized in the following lemma are established by straightforward computation. 
\begin{lem}
\label{Lem:GvarphiActetc}
\begin{enumerate}
\item[(a)] The canonical projection
\[
\pi: \mathrm{GL}_2(K) \rightarrow \mathrm{PGL}_2(K),\quad {\begin{pmatrix}a&b\\c&d\end{pmatrix}}\, \mapsto\, \Big(\omega \mapsto \frac{a\omega + b}{c\omega+d}\Big),
\]
viewed as a map from the deformation $G_\varphi$ onto $\bar{G}_{\bar{\varphi}}$ is a homomorphism with kernel $\zeta_1(\mathrm{GL}_2(K)) \cong K^\times$.
\vspace{2mm}

\item[(b)] The map   
$\bar{\Phi}: (PGL_2(q^2))_{\bar{\varphi}} \longrightarrow \mathrm{M}(q^2)$ given by
\[
\Big(\omega \mapsto \frac{a\omega+b}{c\omega+d}\Big)\,\, \longmapsto\,\, \Big(\omega \mapsto \frac{a \beta(\omega) + b}{c\beta(\omega)+d}\Big)
\]
is an isomorphism.
\vspace{2mm}

\item[(c)] Setting
\[
A\cdot \omega:= \frac{a\beta(\omega)+b}{c\beta(\omega)+d},\quad (A={\begin{pmatrix}a&b\\c&d\end{pmatrix}}\in G_\varphi,\, \omega\in\mathbb{P}^1(K))
\]
defines a $3$-transitive action from the left of the group $G_\varphi$ on the projective line $\mathbb{P}^1(K)$.
\end{enumerate}
\end{lem}

\subsection{Presentations for the groups $G_\varphi$ and $\mathrm{M}(q^2)$}
There are various ways of obtaining a presentation for the group $\mathrm{M}(q^2)$; for instance, one may start with the, rather obvious, presentation
\begin{equation}
\label{Eq:RPres}
\mathrm{M}(q^2)_{0,\infty} = \Big\langle a, b\,\big\vert\,b^{\frac{q^2-1}{2}} = 1,\, a^2 = b^{-\frac{q+1}{2}},\, a^{b} = b^{q-1} a\Big\rangle
\end{equation}
for the double stabiliser $\mathrm{M}(q^2)_{0,\infty}$, where $a(\omega) = \zeta^{-1} \omega^q$ and $b(\omega) = \zeta^2\omega$ for some primitive root $\zeta$ of $K$, and then use 
Corollary~\ref{Cor:Main2Trans} twice to obtain presentations for the stabiliser $\mathrm{M}(q^2)_\infty$ and for $\mathrm{M}(q^2)$ itself.
Alternatively, one may observe that $\mathrm{M}(q^2) \cong (\mathrm{PGL}_2(K))_{\bar{\varphi}}$ is an extension of $\mathrm{ker}(\bar{\varphi}) = \mathrm{PSL}_2(K)$ by $\bar{\varphi}(\mathrm{PGL}_2(K)) \cong C_2$, and lift a presentation for $\mathrm{PSL}_2(K)$ to a presentation for $\mathrm{M}(q^2)$ by means of a $2$-cocycle describing the extension class. In what follows, we shall apply Corollary~\ref{Cor:Main2Trans} directly to the action of $G_\varphi$ on $\mathbb{P}^1(K)$, described in Part~(c) of Lemma~\ref{Lem:GvarphiActetc}, to obtain a suitable presentation of the group $G_\varphi = (\mathrm{GL}_2(K))_\varphi$, and then deduce a presentation for $\mathrm{M}(q^2)$ itself by means of Parts (a) and (b) of that lemma. The special case where $q=3$ will play a role in the next section, when deriving presentations for the small Mathieu groups $M_{11}$ and $M_{12}$. 

\subsubsection{A presentation for $D_\varphi$}

Let $D_\varphi = (G_\varphi)_{0,\infty}$ be the subgroup of $G_\varphi$ consisting of the diagonal matrices. Setting
$Z(\omega):= \omega I_2 \mbox{ and } R(\omega):= \begin{pmatrix}\omega &0\\0&1\end{pmatrix}\, \mbox{ for }\,\omega\in K^\times$,
we have
\begin{multline*}
Z(a_{2,2}) \circ_\varphi R(a_{1,1} a_{2,2}^{-1}) = Z(a_{2,2})\cdot\varphi(Z(a_{2,2}))^{-1}(R(a_{1,1}a_{2,2}^{-1}))\\ 
= Z(a_{2,2})\cdot R(a_{1,1}a_{2,2}^{-1}) = \begin{pmatrix}a_{1,1}&0\\[1mm]0&a_{2,2}\end{pmatrix},
\end{multline*}
so that $D_\varphi$ is generated by the matrices $Z(\omega)$ and $R(\omega)$ for $\omega\in K^\times\backslash\{1\}$. The action of $G_\varphi$ on $\mathbb{P}^1(K)$ induces by restriction a transitive action of the group $D_\varphi$ on the set $\mathbb{P}^1(K)\backslash\{0,\infty\} = K^\times$ given by 
\[
\begin{pmatrix}a_{1,1}&0\\[1mm]0&a_{2,2}\end{pmatrix}\cdot \omega = a_{2,2}^{-1} a_{1,1} \beta(\omega),\quad\omega\in K^\times,
\]
where
\[
\beta = \begin{cases} 1_{K^\times},&a_{1,1} a_{2,2}\in (K^\times)^2\\[1mm]
\alpha\vert_{K^\times},&a_{1,1}a_{2,2}\not\in (K^\times)^2. \end{cases}
\]
In the notation of Theorem~\ref{Thm:PresGpAct}, let $\omega_0=1$, so that 
\[
H = (D_\varphi)_{\omega_0} = Z = \big\{Z(\omega):\,\omega\in K^\times\big\} \cong K^\times.
\]
We have
\[
E = Z\backslash D_\varphi/Z \cong Z\backslash K^\times = \{e\} \cup \big\{\{\omega\}:\,\omega\in K^\times\backslash\{1\}\big\},
\]
where $e:= \{\omega_0\} = \{1\}$. Moreover, identifying a singleton set $\{\omega\}$ with $\omega$ for $\omega\in K^\times$, we have
\[
\overline{\omega} = \left.\begin{cases}
\omega^{-1},& \omega\in (K^\times)^2\\[1mm]
\omega^{-q},& \omega\not\in (K^\times)^2\end{cases}\right\}\quad(\omega\in K^\times).
\]
Thus, since $R_\omega\cdot \omega_0 = \omega$, the set of matrices $\{R(\omega): \omega\in K^\times\}$ is a system of pairwise inequivalent representatives for the double cosets in $E = Z\backslash D_\varphi/Z$, and we may set
\[
\sigma(C_\omega) = R(\omega),\quad \omega\in K^\times
\]
in accordance with the requirements of Section~\ref{Subsec:PresGpAct}. We note that $H_\omega = Z\,(\omega\in K^\times)$ and that 
\[
\iota_{\omega_1}(Z(\omega_2)) = \left.\begin{cases}
Z(\omega_2),&\omega_1\in (K^\times)^2,\\
Z(\omega_2^q),& \omega_1\not\in (K^\times)^2\end{cases}\right\}\quad(\omega_1, \omega_2\in K^\times).
\]
As orientation of $E$, we may, for instance, choose any set of the form
\[
E_+ = \Big\{\zeta^{2i}:\, 0\leq i\leq \frac{q^2-1}{4}\Big\} \cup N,
\]
where $\zeta$ is some fixed primitive element of the field $K$, and $N\subseteq K^\times$ is a set of non-squares of size $\vert N\vert = \frac{q^2-1}{4}$ such that $N \cap \alpha(N^{-1}) = \emptyset$. Applying Theorem~\ref{Thm:PresGpAct}, we find that $D_\varphi$ is generated by the subgroup
\[
Z = \big\langle Z_0\,\vert\,Z_0^{q^2-1} = 1\big\rangle \cong C_{q^2-1},\,\, Z_0:= \zeta I_2
\]
together with extra generators $R(\omega)$ for $\omega \in E_+ - \{1\}$, subject to the relations of type (I)
\begin{equation}
\label{Eq:DphiRel1}
R(\omega_1)^{-1} Z(\omega) R(\omega_1) = \left.\begin{cases}
Z(\omega),&\omega_1\in(K^\times)^2\\[1mm]
Z(\omega^{q}),& \omega_1\not\in (K^\times)^2\end{cases}\right\}\quad(\omega_1\in E_+ - \{1\},\,\omega\in K^\times\backslash\{1\}) 
\end{equation}
plus the type (II) relations
\begin{equation}
\label{Eq:DphiRel2}
R(\omega_1) R(\omega_2) = \left.\begin{cases} R(\omega_1\omega_2),&\omega_1\in (K^\times)^2\\[1mm]
R(\omega_1\omega_2^{q}),& \omega_1\not\in (K^\times)^2\end{cases}\right\}\quad(\omega_1, \omega_2\in K^\times\backslash\{1\}),
\end{equation}
where $R(1)=1$ and $R(\omega)$ is to be interpreted as $R(\overline{\omega})^{-1}$ for $\omega\not\in E_+$.

Set $a:=R(\zeta^{-1})$ and $b:= R(\zeta^2)$, and let
\[
\mathcal{R} = \big\{R(\omega):\,\omega\in K^\times\big\}\leq D_\varphi.
\]
Then, using (\ref{Eq:DphiRel2}), we see that 
\[
R(\zeta^\mu) = \left.\begin{cases}
b^{\frac{\mu}{2}},& \mbox{$\mu$ even}\\[1mm]
b^{\frac{\mu+1}{2}} a,& \mbox{$\mu$ odd}\end{cases}\right\}\quad(0<\mu<q^2), 
\]
so that $\mathcal{R}$ is generated by the elements $a$ and $b$. Moreover, from the relations (\ref{Eq:DphiRel2}) we also  deduce that
\begin{equation}
\label{Eq:RRels}
b^{\frac{q^2-1}{2}} = 1,\, a^2 = R(\zeta^{-(q+1)}) = b^{-\frac{q+1}{2}},\mbox{ and }a^b = b^{q-1}a.
\end{equation}
Conversely, the abstract group generated by $a, b$ modulo the relations (\ref{Eq:RRels}) clearly has order at most $q^2-1$, so that $\mathcal{R} \cong \mathrm{M}(q^2)_{0,\infty}$ with presentation given by (\ref{Eq:RPres}). 
Furthermore, by (\ref{Eq:DphiRel1}), the subgroup $Z$ is normal in $D_\varphi$, with complement $\mathcal{R}$, and the action of $\mathcal{R}$ on $Z$ (that is, the values of $Z_0^a$ and $Z_0^b$) can be read off from Relations~(\ref{Eq:DphiRel1}). We thus obtain the following.
\begin{lem}
\label{Lem:DphiPres}
The deformation $D_\varphi$ of the group $D\cong C_{q^2-1} \times C_{q^2-1}$ is metabelian. It may be generated by the elements $a=R(\zeta^{-1}),$ $b=R(\zeta^2),$  and $Z_0=\zeta I_2,$ where   $\zeta$ is a primitive root of the field $K=GF(q^2),$ subject to the relations
\begin{equation*}
(D_\varphi) \quad Z_0^{q^2-1} = b^{\frac{q^2-1}{2}} = 1,\,a^2 = b^{-\frac{q+1}{2}},\, a^b = b^{q-1}a,\, Z_0^a=Z_0^{q},\, Z_0^b=Z_0.
\end{equation*}
\end{lem}
\begin{proof}
In view of the above observations, the presentation for $D_\varphi$ given in the lemma is immediate. Moreover, it is easy to see from this presentation that
\[
[D_\varphi, D_\varphi] = \big\langle Z_0^{q-1}, b^{q-1}\big\rangle \cong C_{q+1} \times C_{\frac{q+1}{2}},
\]
whence the fact that $D_\varphi$ is metabelian.
\end{proof}

{\footnotesize 
\begin{rem}
\label{Rem:DphiAltPres}
{\em Setting $x:=a,$ $y:=b^{-1},$ and $V:= Z_0^{-1}b,$ we obtain the alternative presentation}
\[
D_\varphi \cong \Big\langle x, y, V\,\Big\vert\, x^2 = y^{\frac{q+1}{2}},\, y^{\frac{q^2-1}{2}} = [y, V] = 1,\,xyx^{-1} = y^q,\, xVx^{-1} = V^q\Big\rangle,
\]
{\em where, as before, $q=p^m$.}
\end{rem}}

\subsubsection{A presentation for the Borel subgroup $B_\varphi$}

The group $B_\varphi = (G_\varphi)_\infty$ of upper triangular matrices acts on the affine line $\mathbb{A}^1(K) = K$ via
\[
\begin{pmatrix} a&b\\0&d\end{pmatrix}\cdot \omega = d^{-1}(a\beta(\omega)+b),\quad \omega\in K,
\]
where 
\[
\beta = \begin{cases} 1_K,& ad\in (K^\times)^2\\[1mm]
\alpha,& ad\not\in (K^\times)^2\end{cases}. 
\]
Let $\omega_0=0$, so that $(B_\varphi)_{\omega_0}=D_\varphi$. Since the action of $B_\varphi$ on $K$ is $2$-transitive, we may apply Corollary~\ref{Cor:Main2Trans}. We let $\omega_1=1$, choose $\tau\in B_\varphi$ as the involution 
$\tau = {\begin{pmatrix} 1&-1\\0&-1\end{pmatrix}}$, and set 
\[
P = Q = \big\{R(\omega): \omega\in K^\times\big\}
\]
all in accordance with the requirements of the corollary. Since the matrix
\[
\tau\circ_\varphi R(\omega) = \tau \cdot R(\omega) = \begin{pmatrix}\omega&-1\\0&-1\end{pmatrix}
\]
sends the point $\omega_1$ to the point $1-\omega = R(1-\omega)\cdot \omega_1$, we have
\[
\rho_\omega' = R(1-\omega),\quad \omega\in K^\times\setminus\{1\},
\]
and a straightforward computation gives that
\begin{equation}
\label{Eq:lambdaprime}
\lambda_\omega' = (R(1-\omega)^{-1} \tau R(\omega))^\tau = \begin{cases} 
Z(-1) R(\frac{\omega}{\omega-1}),&1-\omega\in (K^\times)^2,\\[1mm]
Z(-1) R((\frac{\omega}{\omega-1})^q), & 1-\omega\not\in (K^\times)^2.\end{cases}
\end{equation}
Applying Corollary~\ref{Cor:Main2Trans}, we find that the Borel subgroup $B_\varphi$ of $G_\varphi$ is generated by its subgroup $(B_\varphi)_{\omega_0} = D_\varphi$ plus one extra generator $T$ (which may be identified with the matrix ${\tiny \begin{pmatrix}1&-1\\0&-1\end{pmatrix}}$), subject to the relations $T^2=1$, $Z_0^T=Z_0$, as well as the relations
\begin{equation}
\label{Eq:BphiLongRel}
T R(\omega) T = R(1-\omega) T \lambda_\omega',\quad \omega\in K^\times\setminus\{1\},
\end{equation}
with $\lambda_\omega'$ given by (\ref{Eq:lambdaprime}).

For an integer $j$ with $1\leq j<q^2-1$, let $k_j$ be the integer satisfying
$1- \zeta^j = \zeta^{k_j}$ and $1\leq k_j<q^2-1$. 
Combining our findings with Lemma~\ref{Lem:DphiPres}, we obtain the following.

\begin{lem}
\label{Lem:BphiPres}
The group $B_\varphi$ has a presentation with generators $a=R(\zeta^{-1}),$ $b=R(\zeta^2),$ $Z_0=\zeta I_2,$ and 
$T={\tiny \begin{pmatrix}1&-1\\0&-1\end{pmatrix}},$ subject to the relations $(D_\varphi)$ plus the relations
\begin{align*}
&(B_\varphi)_1\qquad T^2 = [Z_0, T] = 1,\\[1mm]
&(B_\varphi)_{2a}\qquad T R(\zeta^j) T = R(\zeta^{k_j}) T Z_0^{\frac{q^2-1}{2}} R(\zeta^{\frac{q^2-1}{2} - k_j+j}),\quad (1\leq j < q^2-1,\, 2\mid k_j),\\[1mm]
&(B_\varphi)_{2b}\qquad T R(\zeta^j) t = R(\zeta^{k_j}) T Z_0^{\frac{q^2-1}{2}} R(\zeta^{\frac{q^2-1}{2} - qk_j + qj}),\quad (1\leq j < q^2-1, \,2\nmid k_j),
\end{align*}
where $R(\zeta^k)$ has to be rewritten as
\[
R(\zeta^k) = b^{\lceil k/2\rceil}\,a^{k-2\lfloor k/2\rfloor},\quad 1\leq k<q^2-1.
\]
\end{lem}

\subsubsection{A presentation for $G_\varphi$}
Since the action of $G_\varphi$ on $\mathbb{P}^1(K)$ described in Part (c) of Lemma~\ref{Lem:GvarphiActetc} is $3$-transitive, we may apply Corollary~\ref{Cor:Main2Trans} in its most convenient form. We let $\omega_0=\infty$ and $\omega_1=0$, so that $H = (G_\varphi)_{\omega_0} = B_\varphi$ and $H_1=(G_\varphi)_{\omega_0} \cap (G_\varphi)_{\omega_1} = D_\varphi$. Moreover, in accordance with the requirements of Corollary~\ref{Cor:Main2Trans}, we let $\tau = {\tiny \begin{pmatrix}0&-1\\-1&0\end{pmatrix}}$, noting that $\tau^2=1$, and we set
\[
P = \big\{1\big\}  \cup \big\{\kappa_\omega: \omega\in K^\times\big\},\quad \kappa_\omega:=\begin{pmatrix}1&-1\\0&-\omega\end{pmatrix};
\] 
in particular, $\kappa_1 = T$. As representative for the non-trivial double coset in $D_\varphi\backslash B_\varphi/D_\varphi$ we may take $q=T$. The element $\rho_T'\in P$ has to map the point $0$ to $1$, thus $\rho_T'=T$, and one finds that $\lambda_T' = Z(-1) T$. Finally, one computes that the automorphism $\iota: D_\varphi\rightarrow D_\varphi$ given by $h\mapsto \tau^{-1} h \tau$ satisfies $\iota(Z_0) = Z_0$, $\iota(a) = Z_0^{-1} b a$, and $\iota(b) = Z_0^2 b^{-1}$. Applying Corollary~\ref{Cor:Main2Trans}, and implementing  Lemma~\ref{Lem:BphiPres}, we obtain the following.

\begin{pr}
\label{Prop:GphiPres}
The group $G_\varphi$ is generated by the elements $a=R(\zeta^{-1}),$ $b = R(\zeta^2),$ $Z_0=\zeta I_2,$ $T={\begin{pmatrix}1&-1\\0&-1\end{pmatrix}},$ and 
$J={\begin{pmatrix} 0&-1\\-1&0\end{pmatrix}},$ subject to the relations $(D_\varphi),$ $(B_\varphi)_1,$ $(B_\varphi)_{2a},$ $(B_\varphi)_{2b},$ plus the relations
\begin{align*}
&(G_\varphi)_1\qquad J^2 = [J, Z_0] = 1,\\[1mm]
&(G_\varphi)_2\qquad [J, a] = Z_0^{-1} b,\,\, (J b)^2 = Z_0^2,\,\, (J T)^3 = Z_0^{\frac{q^2-1}{2}}.
\end{align*} 
\end{pr}

\subsubsection{A presentation for the Zassenhaus group $\mathrm{M}(q^2)$}
Combining Proposition~\ref{Prop:GphiPres} with Parts (a) and (b) of Lemma~\ref{Lem:GvarphiActetc}, and removing the generator $b$, we find the following presentation for the Zassenhaus group $\mathrm{M}(q^2)$.

\begin{te}
\label{Thm:MPres}
Let $p$ be an odd prime, $m\geq1$ an integer, and let $q=p^m$. Then the group $(PGL_2(q^2))_{\bar{\varphi}} \cong \mathrm{M}(q^2)$ is generated by symbols $a, T,$ and $J,$ subject to the relations
\begin{align*}
&(\bar{D}_\varphi) \qquad [J,a]^{\frac{q^2-1}{2}} = 1,\,\,a^2 = [J,a]^{-\frac{q+1}{2}},\,\,a^{[J,a]} = [J,a]^{q-1} a,\\[1mm]
&(\bar{B}_\varphi)_1\qquad T^2=1,\\[1mm]
&(\bar{B}_\varphi)_{2a}\qquad T R(\zeta^j) T = R(\zeta^{k_j}) T R(\zeta^{\frac{q^2-1}{2} - k_j + j})\quad (1\leq j<q^2-1,\, 2\mid k_j),\\[1mm]
&(\bar{B}_\varphi)_{2b}\qquad T R(\zeta^j) T = R(\zeta^{k_j}) T R(\zeta^{\frac{q^2-1}{2} - qk_j + qj})\quad (1\leq j<q^2-1,\, 2\nmid k_j),\\[2mm] 
&(\bar{G}_\varphi) \qquad J^2 = (J [J,a])^2 = (J T)^3 = 1,
\end{align*}
where $k_j$ is the integer in the range $1\leq k_j<q^2-1$ satisfying $1-\zeta^j = \zeta^{k_j},$ and $R(\zeta^k)$ has to be rewritten as
\[
R(\zeta^k) = [J,a]^{\lceil k/2\rceil}\,a^{k-2\lfloor k/2\rfloor},\quad 1\leq k<q^2-1.
\] 
\end{te}

In preparation for our work in the next section on the small Mathieu groups, we derive from Theorem~\ref{Thm:MPres} an explicit presentation for the group $\mathrm{M}(3^2)$.

\begin{co}
\label{Cor:M(9)}
We have
\begin{multline}
\label{Eq:M(9)Pres}
\mathrm{M}(3^2) \cong  (\mathrm{PGL}_2(3^2))_{\bar{\varphi}} \cong\\[1mm] \big\langle a, T, J \,\big\vert\, T^2 = J^2 = (JT)^3 = (Ta^2)^3 = [J,a]^4 =(J[J,a])^2 = 1,\\[1mm] 
a^2 = [J,a]^2, (aT)^2=T^{[J,a]}\big\rangle,
\end{multline}
where the generators $a, T, J$ are, respectively, the M\"obius transformations $a: \omega\mapsto \zeta^{-1}\omega,$ $T: \omega\mapsto 1-\omega,$ and $J: \omega\mapsto \omega^{-1},$ and $\zeta$ is a primitive element of the field $K= \mathrm{GF}(3^2)$.
\end{co}

\begin{proof}
Using the relation $\zeta^2=1-\zeta$, one finds that $k_1=2$, $k_2= 1$, $k_3= 6$, $k_4=4$, $k_5= 7$, $k_6=3$, and $k_7=5$. Applying Theorem~\ref{Thm:MPres} we obtain, after straightforward simplification, a presentation for $\mathrm{M}(3^2)$ with generators $a$, $T$, and $J$, and defining relations as specified in Corollary~\ref{Cor:M(9)}, plus an extra relation $a^{[J,a]} = a^{-1}$. A dose of GAP then shows that this extra relation is a consequence of the ones given in the corollary, whence the result.
\end{proof}

\section{The Mathieu groups}
\label{Sec:Mathieu}
\subsection{On Witt's construction of multiply transitive groups}
We shall need a version of Witt's construction \cite[Satz~1]{Witt}, where only one new point and one new generator are adjoined at a time. This is the following.

\begin{pr}
\label{Prop:Witt} 
Let $G$ be a $(t-1)$-transitive group on the set 
\[
\Omega = \big\{p_1,\ldots, p_s, q_1,\ldots, q_{t-1}\big\},
\]
where $t\geq3$. For $2\leq j\leq t-1,$ let $S_j\in G$ be a permutation interchanging $q_{j-1}$ with $q_j,$ while fixing $q_\nu$ for all $\nu$ with $1\leq \nu\leq t-1$ and $\nu\neq j-1,j$. Moreover, let $q_t\not\in\Omega$ be a new point, set $\Omega':= \Omega\cup\{q_t\},$ and let $S_t\in\Sigma:=\mathrm{Sym}(\Omega')$ be a permutation interchanging $q_{t-1}$ with $q_t,$ while fixing $q_\nu$ for $1\leq \nu\leq t-2$. Set $H:= G_{q_1,\ldots,q_{t-1}}$ and $G_1:= G_{q_2,\ldots,q_{t-1}},$ and suppose that 
\begin{align}
S_t^2 &\equiv 1\,\,\mathrm{mod}\,\,H,\label{Eq:WittRel1}\\
(S_jS_t)^2 &\equiv 1\,\,\mathrm{mod}\,\,H,\quad (2\leq j\leq t-2),\label{Eq:WittRel2}\\
(S_{t-1} S_t)^3 &\equiv 1\,\,\mathrm{mod}\,\,H,\label{Eq:WittRel3}\\
S_t G_1 S_t &= G_1.\label{Eq:WittRel4}
\end{align}
Then the set of permutations
\[
\widetilde{G} = G \cup G S_t G \subseteq \Sigma
\]
is a $t$-transitive permutation group on the set $\Omega',$ and we have $\widetilde{G}_{q_t} = G$.
\end{pr}

The proof is a minor variation of Witt's original argument in \cite{Witt}. Our next result combines Witt's construction with Corollary~\ref{Cor:Main2Trans} to obtain a presentation for the transitive extension  $\widetilde{G}$ in terms of the original group $G$. 
\begin{pr}
\label{Prop:WittTransfer}
In the notation of Proposition~{\em \ref{Prop:Witt},} let $\sigma_1,\ldots, \sigma_r$ be generators for $G_1,$ and set 
\begin{align*}
\sigma_\rho' &= S_t^{-1} \sigma_\rho S_t\in G_1,\quad 1\leq \rho\leq r,\\
h_j &= [S_t^{-1}, S_j^{-1}]\in H,\quad 2\leq j\leq t-2,\\  
h &=S_t^2\in H.
\end{align*}
Then $\widetilde{G}$ is generated by its subgroup $G$ plus one extra generator $U$ (which may be identified with the element $S_t$), subject to the relations
\begin{align}
\sigma_\rho^U &= \sigma_\rho',\quad 1\leq \rho\leq r,\label{Eq:WittTransRel1}\\
[U^{-1}, S_j^{-1}] &= h_j,\quad 2\leq j\leq t-2,\label{Eq:WittTransRel2}\\
U^2 &= h,\label{Eq:WittTransRel3}\\
U S_{t-1} U &= S_{t-1} U S_t^{S_{t-1} S_t}.\label{Eq:WittTransRel4}
\end{align}
\end{pr}
\begin{proof}
Since $\widetilde{G}$ is $3$-transitive on $\Omega'$, we may apply Corollary~\ref{Cor:Main2Trans} in its most convenient form. In the notation of that corollary, let $\omega_0=q_t$, $\omega_1=q_{t-1}$, set $\tau=S_t$, and fix a system of generators $\{\sigma_1, \ldots, \sigma_r\}$ for $G_1$, noting that $\widetilde{G}_{\omega_0,\omega_1} = G_{q_{t-1}}$ is generated by the elements $\sigma_1,\ldots,\sigma_r, S_2^{-1},\ldots, S_{t-2}^{-1}$. Moreover, we have $\theta = S_t^2=h$, and the automorphism $\iota$ of $G_{q_{t-1}}$ is given by $x \mapsto S_t^{-1} x S_t$, so that $\iota(\sigma_\rho) = \sigma_\rho'$ for $1\leq \rho\leq r$ and $\iota(S_j^{-1}) = h_j S_j^{-1}$ for $2\leq j\leq t-2$. Next, since $\vert\Omega'\vert = s + t \geq3$ and the action of $\tilde{G}$ on $\Omega'$ is $3$-transitive, the space $G_{q_{t-1}}\backslash G/G_{q_{t-1}}$ consists only of two double cosets, with the non-trivial one being represented, for instance, by the element $S_{t-1}$. Also, we may set 
\[
P = \big\{1, \kappa_{p_1}, \ldots, \kappa_{p_\tau}, \kappa_{q_1}, \ldots, \kappa_{q_{t-2}}\big\},
\]
where
\begin{align*}
\kappa_{q_j} &= \prod_{\mu=j+1}^{t-1} S_\mu,\quad 1\leq j\leq t-2,\\
\kappa_{p_k} &= \xi_k \prod_{\mu=2}^{t-1} S_\mu,\quad 1\leq k\leq \tau,
\end{align*}
with $\xi_k\in G_1$ chosen such that $\xi_k(q_1) = p_k$ (this is possible since, by hypothesis, $G_1$ is transitive on the points $p_1,\ldots, p_\tau, q_1$). Finally, the element $\rho' = \rho'_{S_{t-1}}\in P$ has to map the point $q_{t-1}$ to $q_{t-2}$, so that we may take $\rho' = S_{t-1} = \kappa_{q_{t-2}}$, and we compute that $\lambda' = \lambda'_{S_{t-1}} = S_t^{S_{t-1} S_t}\in G$. Applying Corollary~\ref{Cor:Main2Trans}, the assertions of Proposition~\ref{Prop:WittTransfer} follow.
\end{proof}
\subsection{The Mathieu groups $M_{11}$ and $M_{12}$}
\label{Subsec:SmallMathieu}
Let $K=GF(3^2)$, and let $\zeta\in K$ be such that $\zeta^2=1-\zeta$; thus, $\zeta$ is a primitive element of $K$, and $\{1, \zeta\}$ is a basis of $K$ over the prime field $\Pi\cong GF(3)$. The group $M_{11}$ arises from Witt's construction, in the form of Proposition~\ref{Prop:Witt}, with $t=4$ and $G=\mathrm{M}(3^2)$, a sharply $3$-transitive permutation group on the set $\Omega=\mathbb{P}^1(K)=K\cup\{\infty\}$. More precisely, we set 
\[
\Omega = (K^\times\backslash\{1\}) \cup \{q_1, q_2, q_3\},
\]
where $q_1=1$, $q_2=0$, and $q_3=\infty$, we let 
$S_2(\omega) = 1-\omega$ and $S_3(\omega) = \omega^{-1}$ for $\omega\in\Omega$, and we let 
\[
S_4(\infty) = v,\,\,S_4(v) = \infty,\,\,
S_4(\alpha+\zeta \beta) = \alpha-\zeta \beta,\quad(\alpha, \beta\in \Pi),
\]
where $v=q_4\not\in\Omega$; cf.\ \cite[Chap.~XII, Theorem~1.3]{HuBl}. Then
$\tilde{G} = \langle G, S_4\rangle = M_{11}$, 
a sharply $4$-transitive permutation group  of degree $11$ on the set $\Omega' = \Omega \cup\{v\}$. We note that, in the notation of Corollary~\ref{Cor:M(9)}, we have $S_2=T$ and $S_3=J$; in particular, $S_2^2=S_3^2=1$. Since $\mathrm{M}(3^2)$ is sharply $3$-transitive, we have $H=\mathrm{M}(3^2)_{0,1,\infty} = 1$, thus $h_2=1$ and  $h=S_4^2=1$. The group $G_1=\mathrm{M}(3^2)_{0,\infty}$ is generated by the transformations $a: \omega\mapsto \zeta^{-1}\omega^{3}$ and  $b=[J,a]: \omega\mapsto \zeta^2\omega$, and we compute that
$(S_4 a S_4)(1) = S_4(\zeta^{-1}) = 1-\zeta = \zeta^2$ and 
$(S_4 b S_4)(1) = S_4(\zeta^2) = 1+\zeta =\zeta^{-1}$, so that $\rho_1'=S_4aS_4=b$ and $\rho_2'=S_4bS_4=a$, since $G_1$ is sharply transitive on $K^\times$.
Moreover, one checks that the permutation $X=(S_4S_3)^2S_4\in \mathrm{M}(3^2)$ fixes the point $1$ and interchanges $0$ and $\infty$; hence, $X = S_3 = J$. Applying Proposition~\ref{Prop:WittTransfer}, we find that $M_{11}$ is generated by its subgroup $\mathrm{M}(3^2)$ plus one extra generator $U$ (which may be identified with $S_4$), subject to the relations
\begin{equation}
\label{Eq:M11Rel}
U^2 = (UT)^2 = (UJ)^3 = 1,\, a^U = [J,a].
\end{equation}
We thus obtain the following result.
\begin{pr}
\label{Prop:M11Pres}
The group $M_{11}$ is generated by symbols $a, T, J$ (as in Corollary~{\em \ref{Cor:M(9)}}) and $U=S_4,$ subject to the relations specified in Corollary~{\em \ref{Cor:M(9)}} plus the relations {\em (\ref{Eq:M11Rel})}. 
\end{pr}

Similarly, the group $M_{12}$ arises from Proposition~\ref{Prop:Witt} with $t=5$ and $G=M_{11}$ acting on the set
\[
\Omega' = (K^\times\backslash\{1\}) \cup \{q_1,q_2,q_3,q_4\},
\]
where $q_1,q_2,q_3, q_4$ are as before, as are the permutations $S_2=S$, $S_3=T$, and $S_4=U$. Let $q_5=w\not\in\Omega'$ be a new point, and let $S_5$ be the permutation on $\Omega''=\Omega'\cup\{w\}$ given by
\[
S_5(\infty) = \infty, \,S_5(v)=w,\, S_5(w)=v,\,S_5(\omega) = \omega^3\,(\omega\in K).
\]
Then $\tilde{G} = \langle G, S_5\rangle = M_{12}$, 
a sharply $5$-transitive permutation group of degree $12$ on the set $\Omega''$; cf.\ again \cite[Chap.~XII, Theorem~1.3]{HuBl}. In the notation of Proposition~\ref{Prop:WittTransfer}, we have $H=(M_{11})_{0,1,\infty, v} = 1$, so that $h_2 = h_3 = 1$ and  $h =S_5^2 = 1$. Also, 
$\sigma_1' = S_5aS_5$ and $\sigma_2' = S_5bS_5$ with $a, b$ generating the group  
$G_1 = (M_{11})_{0,\infty,v} = \mathrm{M}(3^2)_{0,\infty}$, as above.
We compute $(S_5aS_5)(1)=\zeta^5$ and $(S_5bS_5)(1) = \zeta^6$, so that $\sigma_1' = b^{-1}a$ and $\sigma_2' = b^{-1}$. Moreover, one checks that the permutation $X=(S_5S_4)^2S_5\in M_{11}$ fixes $0, 1, w$ and interchanges $v$ and $\infty$; thus, $X=S_4=U$. Applying Proposition~\ref{Prop:WittTransfer}, we thus obtain the following.
\begin{pr}
\label{Prop:M12Pres}
The group $M_{12}$ is generated by symbols $a, T, J, U$ (as in Proposition~{\em \ref{Prop:M11Pres}}) and $V=S_5,$ subject to the relations specified in Proposition~{\em \ref{Prop:M11Pres},} plus the relations
\[
V^2 = (VT)^2 = (VJ)^2 = (VU)^3 = 1,\,\, 
a^V = [J,a]^{-1} a,\,\, [J,a]^V = [J,a]^{-1}.
\]
\end{pr}

\subsection{The Mathieu groups $M_{22}$, $M_{23}$, and $M_{24}$}
\label{Subsec:LargeMathieu}
The group $M_{22}$ arises from Witt's construction (in the form of Proposition~\ref{Prop:Witt}) with $t=3$ and $G=\mathrm{PSL}_3(4)$, the latter interpreted as a $2$-transitive permutation group on the $21$ points of the projective plane $\Omega=\mathbb{P}^2(4)$. 
More precisely, we set $q_1=[0,1,0]$,  $q_2= [1,0,0]$, $q_3=u\not\in\Omega$, we let $S_2$ be the permutation on $\Omega$ given by
\[
S_2[x,y,z] = [y,x,z],\quad [x,y,z]\in\Omega,
\] 
so that $S_2$ is induced by the permutation matrix   ${\tiny \begin{pmatrix}0&1&0\\1&0&0\\0&0&1\end{pmatrix}}\in SL_3(4)$; and we let $S_3$ be the permutation on $\Omega'=\Omega\cup\{u\}$ given by $S_3 u = [1,0,0]$, $S_3 [1,0,0] = u$, and
\[
S_3[x,y,z] = [x^2+yz, y^2, z^2],\quad([x,y,z]\in\Omega\backslash\{[1,0,0]\}).
\]
Then $\tilde{G} = \langle G, S_3\rangle = M_{22}$.\footnote{For this as well as corresponding information concerning 
the groups $M_{23}$ and $M_{24}$, see Theorem~1.4 in \cite[Chap.~XII]{HuBl}.} We note that $S_2^2=1$ and $S_3^2=h=1$. Also, 
$S_2 = \mathfrak{j} = (S_3S_2)^2S_3$.
Let $\zeta$ be a primitive element of $GF(4)$. The group $G_1 = \big(\mathrm{PSL}_3(4)\big)_{q_2}$ is generated by the images $\mathfrak{r}$, $\s$, $\s^\mathfrak{t}$, $\s^{\mathfrak{j}\mathfrak{t}}$ of  the matrices $\hat{R
}(\zeta)$,  $\hat{S}$, $\hat{S}^T$, and $\hat{S}^{\hat{J}T}$.   We compute that
$S_3 \hat{R}(\zeta) S_3 = \hat{J} \hat{R}(\zeta)\hat{J}$,\,
$S_3 \hat{S} S_3 = \hat{S}$,\,
$S_3 \hat{S}^T S_3 = \hat{S} \hat{S}^T$, and that 
$S_3 \hat{S}^{\hat{J}T} S_3 = \hat{S}^{\hat{J}T} \hat{S}^{T\hat{J}T}$. 
It follows from Proposition~\ref{Prop:WittTransfer} that $M_{22}$ is generated by the group $PSL_3(4)$ plus one extra generator $\mathfrak{u}$ (which may be identified with $S_3$), subject to the relations
\begin{equation}
\mathfrak{u}^2 = (\mathfrak{u}\s)^2 = (\mathfrak{u}\mathfrak{j})^3 = 1,\,
\mathfrak{r}^\mathfrak{u} = \mathfrak{r}^\mathfrak{j},\, 
\s^{\mathfrak{t}\mathfrak{u}} = \s \s^\mathfrak{t},\,
\s^{\mathfrak{j}\mathfrak{t}\mathfrak{u}} = \s^{\mathfrak{j}\mathfrak{t}} \s^{\mathfrak{t}\mathfrak{j}\mathfrak{t}}.\label{Eq:M22ExtraRel}
\end{equation}
Combining this with Corollary~\ref{Cor:PSL34Pres}, we readily obtain a presentation for $M_{22}$.
\begin{pr}
\label{Prop:M22Pres}
The group $M_{22}$ is generated by symbols $\s$, $\mathfrak{j}$, $\mathfrak{t}$, $\mathfrak{r}$ (as in Corollary~{\em \ref{Cor:PSL34Pres})} and $\mathfrak{u}=S_3,$ subject to the relations specified in Corollary~{\em \ref{Cor:PSL34Pres}} plus the relations {\em (\ref{Eq:M22ExtraRel})}.
\end{pr}
Next, the group $M_{23}$ arises from Proposition~\ref{Prop:Witt} with $t=4$ and $G=M_{22}$, a $3$-transitive permutation group of degree $22$ on the set $\Omega'=\Omega\cup\{u\}$. Let $q_1,q_2,q_3$ and $S_2=\mathfrak{j}$, $S_3=\mathfrak{u}$ be as above, let $q_4=v\not\in\Omega'$ be a new point, and let $S_4$ be the permutation on $\Omega'':=\Omega'\cup\{v\}$ given by
\[
S_4(u)=v,\, S_4(v)=u, \mbox{ and } S_4[x,y,z] = [x^2, y^2, \zeta z^2]\mbox{ for }[x,y,z]\in\Omega,
\] 
where $\zeta$ is as before. Then 
$\tilde{G} = \langle G, S_4\rangle = M_{23}$,
a $4$-transitive permutation group of degree $23$ on the set $\Omega''$. We note that $S_4^2 = h = 1$, $h_2 = (S_4S_2)^2 = 1$, and that $(S_4 S_3)^2S_4 = S_3 = \mathfrak{u}$, and we compute that 
$S_4 \hat{R}(\zeta) S_4 = \hat{R}(\zeta)^{-1}$,\,
$S_4 \hat{S} S_4 = \hat{S}$,\,
$S_4 \hat{S}^T S_4 = \hat{S}^{T\hat{R}(\zeta)}$, and that 
$S_4 \hat{S}^{\hat{J}T} S_4 = \hat{S}^{\hat{J}T\hat{R}(\zeta)}$.
Applying Proposition~\ref{Prop:WittTransfer}, we find the following.
\begin{pr}
\label{Prop:M23Pres}
The group $M_{23}$ is generated by the symbols $\s,$ $\mathfrak{j},$ $\mathfrak{t},$ $\mathfrak{r},$ $\mathfrak{u},$ and  $\mathfrak{v}=S_4,$ (the first five as in Proposition~{\em \ref{Prop:M22Pres}),} subject to the relations specified in Proposition~{\em \ref{Prop:M22Pres},} plus the relations
\begin{equation}
\mathfrak{v}^2 = (\mathfrak{v}\s)^2 = (\mathfrak{v}\mathfrak{u})^3 = (\mathfrak{v}\mathfrak{j})^2 = (\mathfrak{v}\mathfrak{r})^2 = 1,\, 
\s^{\mathfrak{t}\mathfrak{v}} = \s^{\mathfrak{t}\mathfrak{r}},\,
\s^{\mathfrak{j}\mathfrak{t}\mathfrak{v}} = \s^{\mathfrak{j}\mathfrak{t}\mathfrak{r}}.\label{Eq:M23ExtraRels}
\end{equation}
\end{pr}

Finally, the group $M_{24}$ arises from Proposition~\ref{Prop:Witt} with $t=5$ and $G=M_{23}$. Let $q_1$, $q_2$, $q_3$, $q_4$, $S_2=\mathfrak{j}$, $S_3=\mathfrak{u}$, $S_4=\mathfrak{v}$, and $\Omega''$ be as above, let $q_5=w\not\in\Omega''$ be a new point, and let $S_5$ be the permutation on $\Omega'''=\Omega''\cup\{w\}$ given by
\[
S_5(u)=u,\, S_5(v)=w,\,S_5(w)=v,\mbox{ and }S_5([x,y,z]) = [x^2, y^2, z^2] \mbox{ for }[x,y,z]\in\Omega.
\]
Then $\tilde{G} = \langle G, S_5\rangle = M_{24}$, a $5$-transitive permutation group of degree $24$ on $\Omega'''$. We have $h_2=h_3=h=1$, $(S_5S_4)^2S_5 = S_4=\mathfrak{v}$, and we compute that 
$S_5 \hat{R}(\zeta) S_5 = \hat{R}(\zeta)^{-1}$,\,
$S_5 \hat{S} S_5 = \hat{S}$,\,
$S_5 \hat{S}^T S_5 = \hat{S}^T$, and 
$S_5 \hat{S}^{\hat{J}T} S_5 = \hat{S}^{\hat{J}T}$. 
Applying Proposition~\ref{Prop:WittTransfer}, we conclude that $M_{24}$ is generated by the group $M_{23}$ plus one extra generator $\mathfrak{w}=S_5$, subject to the relations
\begin{equation}
\mathfrak{w}^2 = (\mathfrak{w}\mathfrak{r})^2 =  (\mathfrak{w}\mathfrak{s})^2 = (\mathfrak{w}\mathfrak{j})^2 = (\mathfrak{w}\mathfrak{u})^2 = (\mathfrak{w}\mathfrak{s}^\mathfrak{t})^2 =  (\mathfrak{w}\mathfrak{s}^{\mathfrak{j}\mathfrak{t}})^2 = (\mathfrak{w}\mathfrak{v})^3 = 1.\label{Eq:M24ExtraRel}
\end{equation}
Combining this information with Proposition~\ref{Prop:M23Pres}, we obtain the following.
\begin{pr}
\label{Prop:M24Pres}
The group $M_{24}$ is generated by the symbols $\mathfrak{s},$ $\mathfrak{j},$ $\mathfrak{t},$ $\mathfrak{r},$ $\mathfrak{u},$ $\mathfrak{v},$ and $\mathfrak{w}=S_5$ (the first six as in Proposition~{\em \ref{Prop:M23Pres}),} subject to the relations specified in Proposition~{\em \ref{Prop:M23Pres},} plus the relations {\em (\ref{Eq:M24ExtraRel})}.
\end{pr}

\section{Group extensions of polygroups by groups}
\label{Sec:ExtPr}

\subsection{The category $\cG\cE\cX\cT(E, H)$ and the extension space ${\rm GEXT}(E, H)$}
\label{Subsec:GroupExt}
Let $H$ be a group, and let $E$ be a polygroup. A triple $(G, j, p)$, consisting of a group $G$, a monomorphism 
$j: H \rightarrow G$, and a surjective map $p : G \rightarrow E$, is called a {\em group extension of the polygroup $E$ by the group} 
$H$ if the map $p$ induces an isomorphism of polygroups $j(H) \sm G /j(H) \rightarrow E$. We denote by $\cG\cE\cX\cT(E, H)$ the
{\em class} of all group extensions of the polygroup $E$ by the group $H$. Note that $\cG\cE\cX\cT(E, H)$ may be empty; this happens, for instance, if $E$ is not a double coset polygroup, or if $H$ is trivial and $E$ is not a group. 

$\cG\cE\cX\cT(E, H)$ is a groupoid (a category in which all morphisms are invertible), where a morphism 
$(G, j, p) \rightarrow (G', j', p')$ is, by definition, a group isomorphism $\varphi: G \rightarrow G'$ such that $\varphi \circ j = j'$ and $p' \circ \varphi = p$. We denote by ${\rm GEXT}(E, H)$ the {\em set} of isomorphism classes (connected components) of the groupoid $\cG\cE\cX\cT(E, H)$. 

Given a pair $(E, H)$ as above, the corresponding {\em extension problem} asks for a parametrization of the (possibly empty) set
${\rm GEXT}(E, H)$ in terms of a suitable generalization of Schreier's {\em factor systems} ({\em $2$-cocycles}). It 
turns out that the {\em group-like graphs of groups}, introduced in \cite[Sec.~4]{Ba1} in the context 
of group actions on connected groupoids, after suitable modification and refinement, provide the desired generalization of Schreier's factor systems (see Theorem~
\ref{Thm:SolExtP} below). 
\subsection{Group-like graphs of groups} 
\label{Subsec:GrouplikeGraphsGroups}
Let $(E, H)$ be a pair consisting of a group $H$ and a set $E$ endowed with an involution \,$\bar{}: E \rightarrow E, f \mapsto
\bar{f},$ together with a distinguished element $e = \bar{e}\in E$, viewed as a graph with a single vertex, with $E$ as the set of edges. We introduce a suitable version of Serre's concept of graph of groups (\cite[Sec.~4.4, Definition 8]{Trees}), based on 
the given pair $(E, H)$.
\begin{de} 
\label{Def:GraphGroups}
{\em By a {\em graph of groups based on the pair} $(E, H)$ we mean a triple $(\H, \iota, \theta)$ consisting of a family 
$\H = (H_f)_{f \in E}$ of subgroups of $H$, a family of isomorphisms $\iota = (\iota_f)_{f\in E}$, where $\iota_f: H_f \rightarrow H_{\bar{f}}$,
and a family of distinguished group elements $\theta = (\theta_f)_{f \in E}$ with $\theta_f \in H_f$, subject to the 
following conditions:

\vspace{-2mm}

\ben
\item[\rm (i)] $H_e = H$, $\theta_e = 1$, and $\iota_e = 1_H$, the identity automorphism of $H$;

\vspace{1mm}

\item[\rm (ii)] for $f\in E$ with $f \neq \bar{f}$, $\iota_{\bar{f}} = \iota_f^{- 1}$, and $\theta_f = 1$;

\vspace{1mm}

\item[\rm (iii)] for $f\in E$ with $f = \bar{f}$, $\iota_f^2=\iota_f \circ \iota_f$ is the inner automorphism of $H_f$ given by 
$h \mapsto \theta_f^{- 1} h \theta_f$, and we have $\iota_f(\theta_f) = \theta_f$.
\een }
\end{de}

Thus, roughly speaking, a graph of groups based on $(E, H)$ is a suitable family of {\em partial automorphisms} of $H$ indexed by the elements of $E$.

Given a graph of groups $(\H, \iota, \theta)$ as above, let $\fG$ be the set of triples $\mathfrak{g}=(f, \rho, \lam)$, where
$f \in E$, $\rho$ is a coset in $H/H_f$, and $\lam: \rho \rightarrow H$ is a map such that 
$\lam(p h) = \iota_f(h)^{- 1} \lam(p)$ for $p \in \rho, h \in H_f$; in particular, $\lam$ is injective, and is completely determined by its value on some element $p \in \rho$. For any element $\mathfrak{g} \in \fG$, we denote its components
by $\eps(\mathfrak{g}) \in E$,\, $\rho(\mathfrak{g}) \in H/H_{\eps(\mathfrak{g})}$, and $\lam(\mathfrak{g}): \rho(\mathfrak{g}) \rightarrow H$, so that $\mathfrak{g} = (\eps(\mathfrak{g}), \rho(\mathfrak{g}), \lam(\mathfrak{g}))$.
We identify $E$ and $H$ with subsets of $\fG$ via the injective maps 
\[
E \lra \fG,\,\, f \mapsto \wt{f}:= (f, H_f, 1 \mapsto 1) = (f, H_f, h \mapsto \iota_f(h)^{- 1})
\]
and 
\[
H \lra \fG,\,\, h \mapsto \wt{h}:= (e, H, 1 \mapsto h) = (e, H, h' \mapsto (h')^{- 1} h);
\] 
in particular, $\wt{e} = (e, H, h\mapsto h^{-1}) = \wt{1}$ is the unique common element of the images $\wt{E}$ and $\wt{H}$ in $\fG$.
 
The map $H \times E \times H \lra \fG, (h_1, f, h_2) \mapsto (f, h_1 H_f, h_1 \mapsto h_2)$ is surjective, and 
identifies $\fG$ with the quotient of $H \times E \times H$ modulo the equivalence relation $\sim$ given by
\begin{equation}
\label{Eq:Equiv}
(h_1, f, h_2) \sim (h_1', f', h_2') \Llra f = f',\, h_1^{- 1} h_1' \in H_f,\,\mbox{and}\,\,\iota_f(h_1^{- 1} h_1') =
h_2 (h_2')^{- 1}.
\end{equation}
 
Next, we consider group-like graphs of groups; these are graphs of groups in the sense of Definition~\ref{Def:GraphGroups}, together with a map $\alpha: E \times H \times E \lra \fG$ satisfying certain axioms as follows.

\begin{de} 
\label{Def:GroupGraphGroups}
{\em By a {\em group-like graph of groups} based on the pair $(E, H)$ we mean a graph of groups $(\H, \iota, \theta)$ based on $(E, H)$ together with a map $\alpha: E \times H \times E \rightarrow \fG$, such that

\vspace{-2mm}

\ben
\item[\rm (i)] $\alpha(e, h, f) = (f, h H_f, h \mapsto 1)$;

\vspace{1.5mm}

\item[\rm (ii)] $\alpha(f, h, e) = (f, H_f, 1 \mapsto h)$;

\vspace{1.5mm}

\item[\rm (iii)] $\alpha(f, 1, \bar{f}) = \wt{\theta_f}$;

\vspace{1.5mm}

\item[\rm (iv)] $\alpha(D) = (\eps(\alpha(C)), u \rho(\alpha(C)), u p \mapsto \lam(\alpha(C))(p) \cdot \lam(\alpha(B))(v)^{- 1})$,
where 
\[
A = (f, h, f'), B = (f', h', f''), u \in \rho(\alpha(A)), v \in \rho(\alpha(B)),
\] 
\[
C = (\eps(\alpha(A)), \lam(\alpha(A))(u) \cdot h', f''),\,\mbox{and}\,D = (f, h v, \eps(\alpha(B))).
\]
\een}
\end{de}
Note that (i) and (ii) imply $\alpha(e, 1, f) = \alpha(f, 1, e) = \wt{f}$ for $f \in E$, and $\alpha(e, h, e) = \wt{h}$
for $h \in H$; in particular, $\alpha(e, 1, e) = \wt{e} = \wt{1}$.

Our next result exhibits a crucial property of the map $\alpha$.

\begin{lem} 
\label{Lem:alpha}
Let $(\H, \iota, \theta, \alpha)$ be a group-like graph of groups based on the pair $(E, H)$. Let $A = (f, h, f') \in
E \times H \times E$, $a \in H_f, b \in H_{f'}$, and let $B = (f, \iota_f(a) h b, f')$. Then we have 
\begin{equation}
\label{Eq:alpha(B)}
\alpha(B) = (\eps(\alpha(A)), a \rho(\alpha(A)), a p \mapsto \lam(\alpha(A))(p) \cdot \iota_{f'}(b)).
\end{equation}
\end{lem}
\bp
With $f, f' \in E$ and $h \in H, a \in H_f$, set $A' := (f, \iota_f(a), e)$ and $B' := (e, h, f')$. Then, by Parts~(i) and (ii) of  
Definition~\ref{Def:GroupGraphGroups}, we obtain 
\[
\alpha(A') = (f, H_f, 1 \mapsto \iota_f(a)) = (f, H_f, a \mapsto 1)
\] 
and 
\[
\alpha(B') = (f', h H_{f'}, h \mapsto 1) = (f', h H_{f'}, h b \mapsto \iota_{f'}(b)^{- 1}).
\]
If we now replace $A, B$, $u \in \rho(\alpha(A))$, $v \in \rho(\alpha(B))$ in Part~(iv) of Definition~\ref{Def:GroupGraphGroups} 
with $A', B'$, $a \in \rho(\alpha(A'))$, and  $h b \in \rho(\alpha(B'))$, respectively, then we find that the triples $C, D$ there become,
respectively, 
$C' = (f, h, f') = A$ and $D' = (f, \iota_f(a) h b, f') = B$. Applying Part~(iv) of Definition~\ref{Def:GroupGraphGroups} now yields 
\begin{multline*}
\alpha(B) = \alpha(D') = (\eps(\alpha(C')), a \rho(\alpha(C')), a p \mapsto \lam(\alpha(C'))(p) \cdot \lam(\alpha(B'))(h b)^{- 1})\\[1mm]
= (\eps(\alpha(A)), a \rho(\alpha(A)), a p \mapsto \lam(\alpha(A))(p) \cdot \iota_{f'}(b)),
\end{multline*}
as desired.
\ep

{\footnotesize 
\begin{rem} \em 
\label{Rem:epsilon/rho}
Let $f, f' \in E$. According to Lemma~\ref{Lem:alpha}, the map $H \longrightarrow E$, $h \mapsto \eps(\alpha(f, h, f'))$, factors 
through the space of double cosets $H_{\bar{f}} \sm H/ H_{f'}$, while the map $H \longrightarrow \coprod_{f^{'\!'} \in E} H/H_{f^{'\!'}}$,
$h \mapsto \rho(\alpha(f, h, f'))$, factors through the space of left cosets $H/H_{f'}$.
\end{rem}}

\begin{co}
\label{Cor:Operation}
In the context of Lemma~{\em \ref{Lem:alpha},} let $\mathfrak{g}, \mathfrak{g}' \in \fG$ and, for $i=1,2,$ let $a_i \in \rho(\mathfrak{g}),$ $a_i' \in \rho(\mathfrak{g}'),$ and $A_i = (\eps(\mathfrak{g}), \lam(\mathfrak{g})(a_i) \cdot a_i', \eps(\mathfrak{g}')) \in 
E \times H \times E$. Then the identity 
\begin{multline}
\label{Eq:GId}
 (\eps(\alpha(A_1)), a_1 \rho(\alpha(A_1)), a_1 p_1 \mapsto \lam(\alpha(A_1))(p_1) \cdot \lam(\mathfrak{g}')(a_1')) \\[1mm]
= (\eps(\alpha(A_2)), a_2 \rho(\alpha(A_2)), a_2 p_2 \mapsto \lam(\alpha(A_2))(p_2) \cdot \lam(\mathfrak{g}')(a_2'))
\end{multline}
holds in $\fG$.
\end{co}
\bp
Clearly, the right-hand, as well as the left-hand side of (\ref{Eq:GId}), are elements of $\fG$, so it suffices to prove equality of these two triples. By assumption, $a_2 = a_1 b, a_2' = a_1' b'$ with some $b \in H_{\eps(\mathfrak{g})}, b' \in H_{\eps(\mathfrak{g}')}$, and hence
\[ 
\lam(\mathfrak{g})(a_2) \cdot a_2' = \iota_{\eps(\mathfrak{g})}(b)^{- 1} \cdot (\lam(\mathfrak{g})(a_1) \cdot a_1') \cdot b'.
\] 
Applying Lemma~\ref{Lem:alpha} with $A_1$ in place of $A$, and with $b^{-1}, b'$ in place of the elements $a$ and $b$, respectively  (so that $B$ becomes $A_2$), we get 
\[
\alpha(A_2) = (\eps(\alpha(A_1)), b^{- 1} \rho(\alpha(A_1)), b^{- 1} p_1 \mapsto \lam(\alpha(A_1))(p_1) \cdot 
\iota_{\eps(\mathfrak{g}')}(b')).
\]
Consequently, we have
$\varepsilon(\alpha(A_2)) = \varepsilon(\alpha(A_1))$ as well as  
\[
a_2 \rho(\alpha(A_2)) = (a_1 b) (b^{- 1} \rho(\alpha(A_1))) = a_1 \rho(\alpha(A_1)), 
\]
and the map
\[
a_2 \rho(\alpha(A_2)) \rightarrow H,\,\, a_2 p_2 \mapsto \lam(\alpha(A_2))(p_2) \cdot \lam(\mathfrak{g}')(a_2')
\] 
sends $a_1 p_1 = a_2 \cdot (b^{- 1} p_1)$, for $p_1 \in \rho(\alpha(A_1))$, to 
\begin{align*}
\lam(\alpha(A_2))(b^{- 1} p_1) \cdot \lam(\mathfrak{g}')(a_1' b') 
&= (\lam(\alpha(A_1))(p_1) \cdot \iota_{\eps(\mathfrak{g}')}(b')) \cdot (\iota_{\eps(\mathfrak{g}')}(b')^{- 1} \cdot \lam(\mathfrak{g}')(a_1'))\\[2mm] 
&= \lam(\alpha(A_1))(p_1) \cdot \lam(\mathfrak{g}')(a_1'), 
\end{align*}
as claimed.
\ep
We now introduce a binary operation $\bullet$ on $\fG$ via
\begin{equation}
\label{Eq:Operation}
\mathfrak{g} \bullet \mathfrak{g}':= (\eps(\alpha(A)), a \rho(\alpha(A)), a p \mapsto \lam(\alpha(A))(p) \cdot \lam(\mathfrak{g}')(a')),
\end{equation}
where $a \in \rho(\mathfrak{g}), a' \in \rho(\mathfrak{g}')$, and $A = (\eps(\mathfrak{g}), \lam(\mathfrak{g})(a) \cdot a', \eps(\mathfrak{g}'))$ (this is well defined in view of Corollary~\ref{Cor:Operation}). Our next result shows that $\bullet$ is in fact a group operation on $\fG$, thereby in particular justifying the terminology
``group-like''.

\begin{pr} 
\label{Prop:AssGroup}

\ben
\item[\rm (1)] Given a group-like graph of groups $(\H, \iota, \theta, \alpha)$ based on the pair $(E, H),$ the set 
$\fG$ carries a canonical group structure, with multiplication $\bullet$ given by {\em (\ref{Eq:Operation})}. The 
identity element of $\fG$ is $\wt{e} = \wt{1},$ and for any element $\mathfrak{g} \in \fG$, its inverse is 
\[
\mathfrak{g}^{-1} = (\overline{\eps(\mathfrak{g})}, \lam(\mathfrak{g})(\rho(\mathfrak{g}))^{- 1}, \lam(\mathfrak{g})(p)^{- 1} \mapsto (p\, \theta_{\eps(\mathfrak{g})})^{- 1}).
\]
 
\item[\rm (2)] The map $H \lra \fG,\, h \mapsto \wt{h} = (e, H, 1 \mapsto h)$ is a monomorphism of groups, identifying
$H$ with the subgroup $\wt{H} = \big\{\wt{h}\,:\,h \in H\big\}$ of $\fG$.

\vspace{1mm}

\item[\rm (3)] For any $\mathfrak{g} \in \fG,$ and for any $a \in \rho(\mathfrak{g}),$ the identity $\mathfrak{g} = \wt{a} \bullet \wt{\eps(\mathfrak{g})} \bullet
\wt{\lam(\mathfrak{g})(a)}$ holds in $\fG;$ in particular, the group $\fG$ is generated by the set $\wt{H} \cup \wt{E}$.

\vspace{1mm}

\item[\rm (4)] For any triple $A = (f, h, f') \in E \times H \times E,$ we have $\alpha(A) = \wt{f} \bullet \wt{h} \bullet \wt{f'};$ in particular, for any pair $(f, f') \in E \times E,$ the map $H \rightarrow \fG$ sending $h$ to $\alpha(f, h, f')$ is injective. 

\vspace{1mm}

\item[\rm (5)] The map $\eps: \fG \lra E$ is a retract of the embedding $E \lra \fG, f \mapsto \wt{f} = (f, H_f, 
1 \mapsto 1)$, inducing an isomorphism $\wh{\eps}$ from the polygroup with support $\wt{H} \sm \fG/ \wt{H}$ onto the 
polygroup $(E;\,\circ,\,\bar{}\,,\,e)$ with the associative hyperoperation 
\[
f \circ f':= \big\{\eps(\alpha(f, h, f')): \,h\in H\big\}.
\]
Thus, the triple $(\fG,\,\wt{}: H \rightarrow \fG, \eps: \fG \rightarrow E)$ is a group extension of the polygroup 
$(E;\,\circ,\,\bar{}\,,e)$ by the group $H$.
\een
\end{pr}
\bp
(1) Let us show first that $\mathfrak{g} \bullet \wt{e} = \mathfrak{g}$ for all $\mathfrak{g} \in \fG$. Setting $\mathfrak{g}' = \wt{e}$ in (\ref{Eq:Operation}), choosing some $a \in \rho(\mathfrak{g})$, and letting 
$a' = \lam(\mathfrak{g})(a)^{- 1} \in H=\rho(\mathfrak{g}')$, we get $A = (\eps(\mathfrak{g}), 1, e)$, so that, according to Part~(ii) of 
Definition~\ref{Def:GroupGraphGroups}, $\alpha(A) = (\eps(\mathfrak{g}), H_{\eps(\mathfrak{g})}, 1 \mapsto 1) = \wt{\eps(\mathfrak{g})}$. Hence, 
\[
\mathfrak{g} \bullet \wt{e} = (\varepsilon(\mathfrak{g}), a H_{\varepsilon(\mathfrak{g})}, ap \mapsto \iota_{\varepsilon(\mathfrak{g})}(p)^{-1}\cdot(a')^{-1}) = \mathfrak{g},
\]
as desired. Similarly, applying Part~(i) of Definition~\ref{Def:GroupGraphGroups}, we find that $\wt{e} \bullet \mathfrak{g} = \mathfrak{g}$ for all $\mathfrak{g} \in \fG$. Thus, $\wt{e}$ is a two-sided identity element for $(\fG, \bullet)$.

Next, let $\mathfrak{g} \in \fG$ be arbitrary, and set 
\[
\mathfrak{g}':= (\overline{\eps(\mathfrak{g})}, \lam(\mathfrak{g})(\rho(\mathfrak{g}))^{- 1}, \lam(\mathfrak{g})(p)^{- 1} \mapsto (p \theta_{\eps(\mathfrak{g})})^{- 1})
\] 
in (\ref{Eq:Operation}). Choosing $a \in \rho(\mathfrak{g})$, and letting $a':= \lam(\mathfrak{g})(a)^{- 1} \in \rho(\mathfrak{g}')$, we obtain $A = (\eps(\mathfrak{g}), 1, \overline{\eps(\mathfrak{g})})$, so that, according to Part~(iii) of Definition~\ref{Def:GroupGraphGroups},
\[
\alpha(A) = \wt{\theta_{\eps(\mathfrak{g})}} = (e,\, H,\, h \mapsto h^{-1}\theta_{\eps(\mathfrak{g})}). 
\]
Hence, 
\[
\mathfrak{g} \bullet \mathfrak{g}' = (e, aH, ap \mapsto p^{-1}\theta_{\varepsilon(\mathfrak{g})} (a \theta_{\varepsilon(\mathfrak{g})})^{-1}) = (e, H, h\mapsto h^{-1}) = \widetilde{e},
\]
as required. If instead we want to evaluate the product $\mathfrak{g}' \bullet \mathfrak{g}$, we choose $a \in \rho(\mathfrak{g}) \in H/H_{\eps(\mathfrak{g})}$, so that 
$a \theta_{\eps(\mathfrak{g})}^{- 1} \in \rho(\mathfrak{g})$, since $\theta_{\eps(\mathfrak{g})} \in H_{\eps(\mathfrak{g})}$ by Definition~\ref{Def:GraphGroups}. We may thus take $a':= \lam(\mathfrak{g})(a \theta_{\eps(\mathfrak{g})}^{- 1})^{- 1} \in \rho(\mathfrak{g}')$ in (\ref{Eq:Operation}) with $\mathfrak{g}, \mathfrak{g}'$ 
interchanged, and the corresponding triple $A$ becomes $A = (\overline{\eps(\mathfrak{g})}, 1, \eps(\mathfrak{g}))$. Proceeding as
above, we obtain $\mathfrak{g}' \bullet \mathfrak{g} = \wt{e}$, so that $\mathfrak{g}' = \mathfrak{g}^{- 1}$ is the two-sided inverse of $\mathfrak{g}$.

Finally, we check associativity of the multiplication $\bullet$ on $\fG$. For $i=1,2,3$, let $\mathfrak{g}_i \in \fG$, and choose arbitrary elements $a_i \in \rho(\mathfrak{g}_i)$. Then, according to (\ref{Eq:Operation}), we have, for $j=1,2$, 
\[
\mathfrak{g}_j \bullet \mathfrak{g}_{j + 1} = (\eps(\alpha(A_j)), a_j \rho(\alpha(A_j)), a_j p_j \mapsto \lam(\alpha(A_j))(p_j) \cdot 
\lam(\mathfrak{g}_{j + 1})(a_{j + 1})),
\]
where 
\[
A_j = (\eps(\mathfrak{g}_j),\, \lam(\mathfrak{g}_j)(a_j) \cdot a_{j + 1},\, \eps(\mathfrak{g}_{j + 1})),\quad j = 1, 2. 
\]
Choosing, for $j=1,2$, an element $p_j \in \rho(\alpha(A_j))$, we obtain, again by (\ref{Eq:Operation}),
\[
(\mathfrak{g}_1 \bullet \mathfrak{g}_2) \bullet \mathfrak{g}_3 = (\eps(\alpha(A_3)), a_1 p_1 \rho(\alpha(A_3)), a_1 p_1 p_3 \mapsto \lam(\alpha(A_3))(p_3)
\cdot \lam(\mathfrak{g}_3)(a_3)),
\]
where $A_3 = (\eps(\alpha(A_1)), \lam(\alpha(A_1))(p_1) \cdot \lam(\mathfrak{g}_2)(a_2) \cdot a_3, \eps(\mathfrak{g}_3))$, and
\[
\mathfrak{g}_1 \bullet (\mathfrak{g}_2 \bullet \mathfrak{g}_3) = (\eps(\alpha(A_4)), a_1 \rho(\alpha(A_4)), a_1 p_4 \mapsto \lam(\alpha(A_4))(p_4) \cdot
\lam(\alpha(A_2))(p_2) \cdot \lam(\mathfrak{g}_3)(a_3)),
\]
where $A_4 = (\eps(\mathfrak{g}_1), \lam(\mathfrak{g}_1)(a_1) \cdot a_2 p_2, \eps(\alpha(A_2)))$. 

Applying Part~(iv) of Definition~\ref{Def:GroupGraphGroups} to the triples $A_1, A_2$ in place of $A$ and $B$, respectively (so that  $C, D$ become, respectively, $A_3$ and $A_4$), we get that  
\[
\alpha(A_4) = (\eps(\alpha(A_3)), p_1 \rho(\alpha(A_3)), p_1 p_3 \mapsto \lam(\alpha(A_3))(p_3) \cdot 
\lam(\alpha(A_2))(p_2)^{- 1}),
\]
implying, in particular, $\varepsilon(\alpha(A_4)) = \varepsilon(\alpha(A_3))$ and $\rho(\alpha(A_4)) = p_1\rho(\alpha(A_3))$. Moreover, we find that the map $\lambda(\mathfrak{g}_1\bullet (\mathfrak{g}_2\bullet \mathfrak{g}_3)): a_1 p_1 \rho(\alpha(A_3)) \rightarrow H$ sends $a_1p_1p_3$, for $p_3\in \rho(\alpha(A_3))$, to
\begin{multline*}
\lambda(\alpha(A_4))(p_1p_3) \cdot \lambda(\alpha(A_2))(p_2)\cdot \lambda(\mathfrak{g}_3)(a_3)\\[1mm] 
= (\lambda(\alpha(A_3))(p_3)\cdot \lambda(\alpha(A_2))(p_2)^{-1})\cdot(\lambda(\alpha(A_2))(p_2)\cdot \lambda(\mathfrak{g}_3)(a_3))\\[1mm] 
= \lambda(\alpha(A_3))(p_3) \cdot \lambda(\mathfrak{g}_3)(a_3) = \lambda((\mathfrak{g}_1\bullet \mathfrak{g}_2)\bullet \mathfrak{g}_3)(a_1p_1p_3),
\end{multline*}
whence the desired identity $(\mathfrak{g}_1 \bullet \mathfrak{g}_2) \bullet \mathfrak{g}_3 = \mathfrak{g}_1 \bullet (\mathfrak{g}_2 \bullet \mathfrak{g}_3)$. The proof
of Part (1) is thus complete.

(2) We have already observed that the map in question is injective, so it suffices to show that it is a group homomorphism. Let $h, h' \in H$. Setting $\mathfrak{g} = \wt{h} = (e, H, 1 \mapsto h)$,\, $\mathfrak{g}' = \wt{h'} = (e, H, 1 \mapsto h')$, and  $a = a' = 1$ in (\ref{Eq:Operation}), and using, say, Part~(ii) of Definition~\ref{Def:GroupGraphGroups}, we find that $A = (e, h, e)$, that 
$\alpha(A) = \wt{h}$, and that 
\[
\wt{h} \bullet \wt{h'} = (e, H, 1 \mapsto h \cdot h') = \wt{h \cdot h'}, 
\]
as desired. 

(3) Let $\mathfrak{g} \in \fG$ and $a \in \rho(\mathfrak{g})$. From (\ref{Eq:Operation}) and Part~(i) of Definition~\ref{Def:GroupGraphGroups}, we get
\[
\wt{a} \bullet \wt{\eps(\mathfrak{g})} = (\eps(\mathfrak{g}), a H_{\eps(\mathfrak{g})}, a \mapsto 1),
\]
while the desired identity 
\[
(\wt{a} \bullet \wt{\eps(\mathfrak{g})}) \bullet \wt{\lam(\mathfrak{g})(a)} = (\eps(\mathfrak{g}), a H_{\eps(\mathfrak{g})}, a \mapsto \lam(\mathfrak{g})(a)) = \mathfrak{g}
\]
follows by (\ref{Eq:Operation}) and Part~(ii) of Definition~\ref{Def:GroupGraphGroups}.

(4) Let $A = (f, h, f') \in E \times H \times E$. Then $\wt{f} \bullet \wt{h} = (f, H_f, 1 \mapsto h)$ by 
(\ref{Eq:Operation}) and Part~(ii) of Definition~\ref{Def:GroupGraphGroups}, and the desired identity $(\wt{f} \bullet \wt{h}) \bullet \wt{f'} = \alpha(A)$ follows now  from (\ref{Eq:Operation}).

(5) By definition, $\eps(\wt{f}) = f$ for all $f \in E$, so that $\varepsilon$ is indeed a retract of the map ${}^\sim:E\rightarrow\mathfrak{G}$; in particular, $\varepsilon$ is surjective. Next, using (\ref{Eq:Operation}) together with Parts~(i) and (ii) of Definition~\ref{Def:GroupGraphGroups}, one sees that, for $\mathfrak{g}\in\mathfrak{G}$ and $h, h'\in H$,
\[
\varepsilon(\tilde{h}\bullet \mathfrak{g}) = \varepsilon(\mathfrak{g}) = \varepsilon(\mathfrak{g}\bullet \tilde{h'}),
\]
so that $\varepsilon$ induces a well-defined surjective map $\hat{\varepsilon}: \widetilde{H}\backslash \mathfrak{G}/\widetilde{H} \rightarrow E$. Setting $C(\mathfrak{g}):= \wt{H} \bullet \mathfrak{g} \bullet \wt{H} \in \wt{H} \sm \fG/
\wt{H}$ for $\mathfrak{g} \in \fG$, we have $C(\mathfrak{g}) = C(\wt{\eps(\mathfrak{g})})$ by Part~(3) of our  proposition, thus 
$\hat{\eps}: \wt{H} \sm \fG/ \wt{H} \rightarrow E$ is a bijection. Also, by Part~(1) of the proposition, 
\[
\hat{\varepsilon}(\overline{C(\mathfrak{g})}) = \hat{\eps}(C(\mathfrak{g}^{- 1})) = \hat{\varepsilon}(C(\widetilde{\varepsilon(\mathfrak{g}^{-1})})) = \varepsilon(\widetilde{\varepsilon(\mathfrak{g}^{-1})}) = \varepsilon(\mathfrak{g}^{-1}) = \overline{\eps(\mathfrak{g})} = 
\overline{\hat{\eps}(C(\mathfrak{g}))},
\]
so that the bijection $\hat{\varepsilon}: \widetilde{H}\backslash \mathfrak{G}/\widetilde{H} \rightarrow E$ respects the bar operation, and we also have
$\hat{\eps}(C(\wt{1})) = e$. Finally, the associative hyperoperation $\circ$ on $E$, as defined in
the proposition, is obtained by transporting the canonical hyperoperation on the double coset space 
$\wt{H} \sm \fG/ \wt{H}$ via the bijection $\hat{\eps}$, making use of Part~(4).
\ep

\begin{de}
\label{Def:FundGroup}
{\em The group $\mathfrak{G}$ associated with a group-like graph of groups $(\mathbb{H}, \iota, \theta, \alpha)$ based on the pair $(E, H)$ according to Part~(1) of Proposition~\ref{Prop:AssGroup} is called the \emph{fundamental group} of $(\mathbb{H}, \iota, \theta, \alpha)$, and we write $\mathfrak{G} = \pi_1(\mathbb{H}, \iota, \theta, \alpha)$.}
\end{de}

\subsection{A presentation for $\pi_1(\mathbb{H}, \iota, \theta, \alpha)$}
\label{Subsec:Presentation}
Let $(\H, \iota, \theta, \alpha)$ be a group-like graph of groups based on the pair $(E, H)$ with its associated fundamental group $\fG = \pi_1(\H, \iota, \theta, \alpha)$ as defined in Section~\ref{Subsec:GrouplikeGraphsGroups}. In order to obtain a combinatorial description of the group $\fG$, we proceed as follows.

First, for each $f \in E$, we choose a left transversal $P_f$ for $H$ modulo its subgroup $H_f$, with $1 \in P_f$. Consequently, according to Part~(3) of Proposition~\ref{Prop:AssGroup}, we obtain a {\em normal form} for the elements of $\fG$: setting, for $\mathfrak{g}\in \mathfrak{G}$, $\rho(\mathfrak{g}) \cap P_{\varepsilon(\mathfrak{g})} = \{\rho_{\mathfrak{g}}\}$ and $\lambda_{\mathfrak{g}} = \lambda(\mathfrak{g})(\rho_{\mathfrak{g}})$, we have 
\begin{equation}
\label{Eq:NormalForm}
\mathfrak{g} = \wt{\rho_{\mathfrak{g}}} \bullet \wt{\eps(\mathfrak{g})} \bullet \wt{\lam_{\mathfrak{g}}}. 
\end{equation}
Next, for each pair $(f, f') \in E \times E$, we choose a set $Q_{f, f'} \sse H$ of pairwise inequivalent representatives for the double cosets in the space $H_f \sm H/ H_{f'}$ with $1 \in Q_{f, f'},$ and with  
$Q_{f', f} = Q_{f, f'}^{- 1}:= \big\{q^{- 1}: q \in Q_{f, f'}\big\}$, provided that $f \neq f'$. The following lemma investigates the behaviour under multiplication of the invariants $\eps(\mathfrak{g}), \rho_{\mathfrak{g}},$ and $\lam_{\mathfrak{g}}$ 
of $\mathfrak{g} \in \fG$. 
\begin{lem}
\label{Lem:InvDepMult}
Let $\mathfrak{g}, \mathfrak{g}'\in \fG,$ let $h \in Q_{\overline{\eps(\mathfrak{g})}, \eps(\mathfrak{g}')}, a \in H_{\eps(\mathfrak{g})}, b \in H_{\eps(\mathfrak{g}')}$ be chosen such that 
$\lam_{\mathfrak{g}} \rho_{\mathfrak{g}'} = \iota_{\eps(\mathfrak{g})}(a) h b,$ and set $\mathfrak{u}:= \alpha(\eps(\mathfrak{g}), h, \eps(\mathfrak{g}'))$. Then we have the following:
\vspace{-2mm}
\begin{enumerate}
\item[(i)] \hspace{2mm} $\eps(\mathfrak{g} \bullet \mathfrak{g}') = \eps(\mathfrak{u}),$
\vspace{2mm}
\item[(ii)]\hspace{2mm} $\rho_{\mathfrak{g} \bullet \mathfrak{g}'} \equiv \rho_{\mathfrak{g}} \cdot a \cdot \rho_{\mathfrak{u}}\,{\rm mod}\,H_{\eps(\mathfrak{u})},$ 
\vspace{2mm}
\item[(iii)]\hspace{2mm} $\lam_{\mathfrak{g} \bullet \mathfrak{g}'} = \iota_{\eps(\mathfrak{u})}(\rho_{\mathfrak{g}\bullet \mathfrak{g}'}^{-1} \rho_{\mathfrak{g}}  a  
\rho_{\mathfrak{u}}) \cdot \lam_{\mathfrak{u}} \cdot \iota_{\eps(\mathfrak{g}')}(b) \cdot \lam_{\mathfrak{g}'}$.
\end{enumerate}
\end{lem}
\begin{proof}
Let $A = (\eps(\mathfrak{g}), h, \eps(\mathfrak{g}'))$ and $B = (\eps(\mathfrak{g}), \lambda_{\mathfrak{g}} \rho_{\mathfrak{g}'}, \eps(\mathfrak{g}'))$. Thus 
$\alpha(A) = \mathfrak{u}$, and, by Lemma~\ref{Lem:alpha}, we obtain 
\[
\alpha(B) = \big(\eps(\mathfrak{u}),\, a \rho(\mathfrak{u}),\, a \rho_{\mathfrak{u}} \mapsto \lam_{\mathfrak{u}} \cdot \iota_{\eps(\mathfrak{g}')}(b)\big). 
\]
According to (\ref{Eq:Operation}), we get 
\[
\mathfrak{g} \bullet \mathfrak{g}' = \big(\varepsilon(\mathfrak{u}),\, \rho_{\mathfrak{g}} a \rho(\mathfrak{u}),\,\rho_{\mathfrak{g}} a \rho_{\mathfrak{u}} \mapsto \lambda_{\mathfrak{u}}\cdot \iota_{\varepsilon(\mathfrak{g}')}(b) \cdot \lambda_{\mathfrak{g}'}\big),
\]
from which Assertions~(i) and (ii) follow immediately. Also,
\begin{align*}
\lambda_{\mathfrak{g}\bullet \mathfrak{g}'} &= \lambda(\mathfrak{g}\bullet \mathfrak{g}')(\rho_{\mathfrak{g}\bullet \mathfrak{g}'})\\[1mm] 
&= \lambda(\mathfrak{g}\bullet \mathfrak{g}')(\rho_{\mathfrak{g}} a \rho_{\mathfrak{u}} \cdot (\rho_{\mathfrak{g}\bullet \mathfrak{g}'}^{-1} \rho_{\mathfrak{g}} a \rho_{\mathfrak{u}})^{-1})\\[1mm] 
&= \iota_{\varepsilon(\mathfrak{u})}(\rho_{\mathfrak{g}\bullet \mathfrak{g}'}^{-1} \rho_{\mathfrak{g}} a \rho_{\mathfrak{u}})\cdot \lambda_{\mathfrak{u}}\cdot \iota_{\varepsilon(\mathfrak{g}')}(b)\cdot \lambda_{\mathfrak{g}'}, 
\end{align*}
whence (iii).
\end{proof}
Next, we choose an \textit{orientation} of $(E,{}^-)$, that is, a subset $E_+\subseteq E$ such that
\[
E = E_+ \amalg \big\{\bar{f}: f\in E_+ \mbox{ and } f\neq \bar{f}\big\};
\] 
in particular, $e = \bar{e}\in E_+$. We denote by $F = F_X$ the free group with basis 
\[
X = \big\{x_f: f \in E_+ \setminus \{e\}\big\};
\]
that is, the elements of $X$ are in bijective correspondence with the elements of the set $E_+':= E_+\setminus\{e\}$, and we let $\wt{\fG} = H\ast F$ be the free product 
of the groups $H$ and $F$. According to Part (3) of Proposition~\ref{Prop:AssGroup}, the group $\fG$ is generated by the 
union $\wt{H} \cup \wt{E}$. Moreover, by Part~(1) of Proposition~\ref{Prop:AssGroup}, the definition of the map ${}^\sim\!: E \rightarrow \mathfrak{G}$, and Part~(ii) of Definition~\ref{Def:GraphGroups}, we have $\wt{f}^{- 1} = \wt{\bar{f}}$ for $f \in E$ and $f \neq \bar{f},$ so that 
$\fG$ is generated by $\wt{H} \cup \wt{E_+'}$. We note that , in general,
\begin{equation}
\label{Eq:ftildeInverse}
\tilde{f}^{-1} = \tilde{\bar{f}} \bullet \widetilde{\theta_f}^{-1},\quad f\in E.
\end{equation}
We have a canonical epimorphism $\pi: \wt{\fG} \rightarrow \fG$ 
with $\pi(h) = \wt{h}$ for $h \in H$, and $\pi(x_f) = \wt{f}$ for $f \in E_+'$. It is convenient to set $x_e:= 1$ and 
$x_f:= x_{\bar{f}}^{-1}$ for $f\in E\setminus E_+$, so that we have $\pi(x_f) = \wt{f}$ for all $f\in E$. In order to obtain the 
desired presentation of $\fG$ in terms of the subgroup $H$ and the set $E_+$, we have to provide a system of generators 
for the kernel $K=\Ker(\pi)$ as a normal subgroup of $\wt{\fG}$.

Denote by $\sim_1$ the congruence on the group $\wt{\fG}$ generated by the family of pairs
\[
(h x_f, x_f \iota_f(h)),\quad(f\in E_+',\, h\in H_f).
\]
Note that we have
\[
h x_f \sim_1 x_f\, \iota_f(h),\quad(f\in E,\,h\in H_f).
\]
Indeed, this is clear for $f=e$, and if $f\in E\setminus E_+$ and $h\in H_f$, then $\bar{f}\in E_+'$ and, by definition of 
$\sim_1$,
\[
\iota_f(h^{-1}) x_{\bar{f}} \sim_1 x_{\bar{f}} \iota_{\bar{f}}(\iota_f(h^{-1})) = x_{\bar{f}} h^{-1},
\]
hence 
\[
h x_f = h x_{\bar{f}}^{-1} \sim_1 x_{\bar{f}}^{-1} \iota_f(h) = x_f \iota_f(h),
\]
as desired. The quotient group $\wh{\fG}:= \wt{\fG}/\!\!\sim_1$ is an HNN-extension with base group $H$,  stable letters $x_f$ for $f\in E_+'$, associated 
subgroups $H_f, H_{\bar{f}}$, and associated isomorphisms $\iota_f: H_f\rightarrow H_{\bar{f}}$. By the normal form theorem for HNN-extensions, the projection 
$\wh{\pi}: \wt{\fG} \rightarrow \wh{\fG}$ is the identity on the factors $H$ and $F$ of the free product 
$\wt{\fG}=H\ast F$, so we may use the same symbols for the images via $\wh{\pi}$ of the elements in $H\cup F$; cf.\, for 
instance, Theorem~3 in \cite[Chap.~VI]{Baumslag}. We also note that, by the definition (\ref{Eq:Operation}) of the group
operation $\bullet$, the epimorphism $\pi: \wt{\fG} \rightarrow \fG$ factors through $\wh{\fG}$, so that $\fG$ is a quotient of $\wh{\fG}$. More explicitly, combining Part (iii) of Definition~\ref{Def:GroupGraphGroups} with Lemma~\ref{Lem:alpha}, we find that
\begin{equation}
\label{Eq:ffbar}
\alpha(\bar{f}, h, f) = \widetilde{\theta_{\bar{f}} \iota_f(h)},\quad (f\in E,\, h\in H_f).
\end{equation}

Next, denote by $\sim_2$ the congruence on $\wh{\fG}$ generated by the family of pairs
\[
(x_f h x_{f'}, \rho_{\mathfrak{g}} x_{\varepsilon(\mathfrak{g})} \lambda_{\mathfrak{g}})\quad (f,f' \in E',\,h\in Q_{\bar{f}, f'},\, \mathfrak{g}= 
\alpha(f, h, f')),
\] 
where $E':=E\setminus\{e\}$, and let $\sim$ be the congruence on $\wt{\fG}$ generated by the union of the congruence $\sim_1$ and the lifting via 
$\wh{\pi}$ of the congruence $\sim_2$ to $\wt{\fG}$. Note that, by Parts~(i) and (ii) of Definition~\ref{Def:GroupGraphGroups}, we have
\[
x_f h x_{f'} \sim \rho_{\mathfrak{g}} x_{\varepsilon(\mathfrak{g})} \lambda_{\mathfrak{g}} \mbox{ for }f,f'\in E,\, h\in Q_{\bar{f}, f'}, \mbox{ and }\mathfrak{g} = \alpha(f,h,f').
\]
By Part~(4) of Proposition~\ref{Prop:AssGroup} plus Equation~(\ref{Eq:NormalForm}), the epimorphism 
$\pi: \wt{\fG}\rightarrow \fG$ factors through the quotient $\wt{\fG}/\!\!\sim\,\, \cong \wh{\fG}/\!\!\sim_2$. Moreover, 
we have the following statement refining \cite[Proposition 4.4.(ii)]{Ba1}.
\begin{te}
\label{Thm:Presentation}
The canonical projection 
$\varphi: \wt{\fG}/\!\!\sim\,\rightarrow\, \fG$ 
induced by $\pi: \wt{\fG} \rightarrow \fG$ is an isomorphism. Consequently, the group $\fG = \pi_1(\mathbb{H}, \iota, \theta, \alpha)$ is generated by the free product 
$\wt{\fG} = H\ast F$ modulo the relations of type {\em (I)}
\begin{equation}
\label{Eq:Type(I)}
h x_f = x_f \iota_f(h),\quad(f\in E_+',\, h\in H_f),
\end{equation}
together with the relations of type {\em (II)}
\begin{equation}
\label{Eq:Type(II)}
x_f h x_{f'} = \rho_{\mathfrak{g}} x_{\varepsilon(\mathfrak{g})} \lambda_{\mathfrak{g}},\quad(f, f'\in E',\, h\in Q_{\bar{f},f'},\, \mathfrak{g}= \alpha(f, h, f')).
\end{equation}
\end{te} 
\begin{proof}
First, note that, by construction, the relations of types (I) and (II) hold in $\wt{\fG}/\!\!\sim$ for arbitrary 
$f,f'\in E$. Next, making use of the invariants $\rho_{\mathfrak{g}}$, $\varepsilon(\mathfrak{g})$, and $\lambda_{\mathfrak{g}}$ of an element $\mathfrak{g}\in\mathfrak{G}$, define a map $\psi: \fG \rightarrow \wt{\fG}/\!\!\sim$ via
\[
\mathfrak{g} \mapsto \psi(\mathfrak{g}):= \rho_{\mathfrak{g}} x_{\varepsilon(\mathfrak{g})} \lambda_{\mathfrak{g}}\mbox{ modulo }\sim.
\]
It suffices to show that $\psi$ is a homomorphism; that is, that
\[
L_{\mathfrak{g},\mathfrak{g}'}:= \rho_{\mathfrak{g}} x_{\varepsilon(\mathfrak{g})} \lambda_{\mathfrak{g}} \rho_{\mathfrak{g}'} x_{\varepsilon(\mathfrak{g}')} \lambda_{\mathfrak{g}'}\, \sim\, 
\rho_{\mathfrak{g} \bullet \mathfrak{g}'} x_{\varepsilon(\mathfrak{g} \bullet \mathfrak{g}')} \lambda_{\mathfrak{g} \bullet \mathfrak{g}'} =: R_{\mathfrak{g},\mathfrak{g}'}
\]
holds for all $\mathfrak{g},\mathfrak{g}'\in \fG$. Once this is accomplished, it is easily checked on generators that $\psi\circ \varphi = 1_{\tilde{\mathfrak{G}}/\sim}$, whence injectivity of $\varphi$. 

Given $\mathfrak{g}, \mathfrak{g}'\in \fG$, let $h \in Q_{\overline{\varepsilon(\mathfrak{g})}, \varepsilon(\mathfrak{g}')},$ $a\in H_{\varepsilon(\mathfrak{g})},$ and  
$b\in H_{\varepsilon(\mathfrak{g}')}$ be such that $\lambda_{\mathfrak{g}} \rho_{\mathfrak{g}'} = \iota_{\varepsilon(\mathfrak{g})}(a) h b$. Then
\[
L_{\mathfrak{g},\mathfrak{g}'} = \rho_{\mathfrak{g}} x_{\varepsilon(\mathfrak{g})}\, \iota_{\varepsilon(\mathfrak{g})}(a) h b\, x_{\varepsilon(\mathfrak{g}')} \lambda_{\mathfrak{g}'}.
\]
Since
\[
x_{\varepsilon(\mathfrak{g})}\, \iota_{\varepsilon(\mathfrak{g})}(a) \sim a x_{\varepsilon(\mathfrak{g})}
\]
as well as
\[
b x_{\varepsilon(\mathfrak{g}')} \sim x_{\varepsilon(\mathfrak{g}')}\,\iota_{\varepsilon(\mathfrak{g}')}(b)
\]
by relations of type (I), it follows that
\[
L_{\mathfrak{g},\mathfrak{g}'} \,\sim \rho_{\mathfrak{g}} a x_{\eps(\mathfrak{g})} h x_{\eps(\mathfrak{g}')} \iota_{\eps(\mathfrak{g}')}(b) \lam_{\mathfrak{g}'}.
\]
Set $\mathfrak{u}:= \alpha(\eps(\mathfrak{g}), h, \eps(\mathfrak{g}')) = \wt{\eps(\mathfrak{g})} \bullet \wt{h} \bullet \wt{\eps(\mathfrak{g}')}$. Since 
\[
x_{\varepsilon(\mathfrak{g})} h x_{\varepsilon(\mathfrak{g}')} \,\sim\, \rho_{\mathfrak{u}} x_{\varepsilon(\mathfrak{u})} \lambda_{\mathfrak{u}}
\] 
by a type (II) relation, we deduce that
\begin{equation}
\label{Eq:LEquiv}
L_{\mathfrak{g},\mathfrak{g}'} \,\sim\, \rho_{\mathfrak{g}} a \rho_{\mathfrak{u}} x_{\varepsilon(\mathfrak{u})} \lambda_{\mathfrak{u}} \iota_{\varepsilon(\mathfrak{g}')}(b) \lambda_{\mathfrak{g}'}.
\end{equation}
Further, by Parts~(i) and (ii) of Lemma~\ref{Lem:InvDepMult}, we have
$\eps(\mathfrak{g} \bullet \mathfrak{g}') = \eps(\mathfrak{u})$ as well as $\rho_{\mathfrak{g} \bullet \mathfrak{g}'}^{-1}\, \rho_{\mathfrak{g}} a \rho_{\mathfrak{u}} \in H_{\varepsilon(\mathfrak{u})}$, so that
\begin{equation}
\label{Eq:AuxRel}
(\rho_{\mathfrak{g} \bullet \mathfrak{g}'}^{-1} \rho_{\mathfrak{g}} a \rho_{\mathfrak{u}}) x_{\varepsilon(\mathfrak{u})}\, \sim\, x_{\varepsilon(\mathfrak{u})} 
\iota_{\varepsilon(\mathfrak{u})}(\rho_{\mathfrak{g} \bullet \mathfrak{g}'}^{-1} \rho_{\mathfrak{g}} a \rho_{\mathfrak{u}})
\end{equation}
by a type (I) relation. Making use of (\ref{Eq:AuxRel}), as well as Parts~(i) and (iii) of Lemma~\ref{Lem:InvDepMult}, 
we now find that
\begin{align*}
R_{\mathfrak{g},\mathfrak{g}'} &= \rho_{\mathfrak{g} \bullet \mathfrak{g}'} x_{\varepsilon(\mathfrak{g} \bullet \mathfrak{g}')} \lambda_{\mathfrak{g} \bullet \mathfrak{g}'}\\[2mm]
&= \rho_{\mathfrak{g} \bullet \mathfrak{g}'} x_{\varepsilon(\mathfrak{u})} \iota_{\varepsilon(\mathfrak{u})}(\rho_{\mathfrak{g} \bullet \mathfrak{g}'}^{-1} \rho_{\mathfrak{g}} a \rho_{\mathfrak{u}}) \lambda_{\mathfrak{u}} \iota_{\varepsilon(\mathfrak{g}')}(b) \lambda_{\mathfrak{g}'}\\[2mm]
&\sim \rho_{\mathfrak{g} \bullet \mathfrak{g}'} (\rho_{\mathfrak{g} \bullet \mathfrak{g}'}^{-1} \rho_{\mathfrak{g}} a \rho_{\mathfrak{u}}) x_{\varepsilon(\mathfrak{u})} \lambda_{\mathfrak{u}} 
\iota_{\varepsilon(\mathfrak{g}')}(b) \lambda_{\mathfrak{g}'}\\[2mm] 
&= \rho_{\mathfrak{g}} a \rho_{\mathfrak{u}} x_{\varepsilon(\mathfrak{u})} \lambda_{\mathfrak{u}} \iota_{\varepsilon(\mathfrak{g}')}(b) \lambda_{\mathfrak{g}'},
\end{align*}
the last displayed expression being equivalent to $L_{\mathfrak{g},\mathfrak{g}'}$ by (\ref{Eq:LEquiv}), as desired.  
\end{proof}

\subsection{Reducing the number of type (II) relations}

In specific cases, we may discard some of the type (II) relations using the concrete combinatorial structure of the given polygroup $E$. In 
any case, we have the following.
\begin{lem}
\label{Lem:ReduceII}
In the context of Theorem~{\em \ref{Thm:Presentation},} choose a total order $\leq$ on the set $E' = E\setminus\{e\}$. Then
 the relations of type {\em (II)} with the pair $(f, f')$ satisfying $\bar{f} > f'$ can be derived from type {\em (I)} relations plus the type {\em (II)} relations with $\bar{f}\leq f'$.
\end{lem}

\begin{proof}
Let $f, f' \in E'$ be such that $\bar{f} > f'$, and let
$h \in Q_{\bar{f}, f'}\,$. As $\bar{f} \neq f'$, it follows that $h^{- 1} \in Q_{f', \bar{f}}$, so that we have an accepted type (II) relation 
\begin{equation}
\label{Eq:AccTypeII}
x_{\bar{f'}} h^{- 1} x_{\bar{f}} = \rho_{\mathfrak{g}} x_{\varepsilon(\mathfrak{g})} \lam_{\mathfrak{g}}, \quad \mathfrak{g} = \alpha(\bar{f'}, h^{-1}, \bar{f}).
\end{equation}
On the other hand, making use of Equations~(\ref{Eq:NormalForm}) and (\ref{Eq:ftildeInverse}), Part~(4) of Proposition~\ref{Prop:AssGroup}, plus the fact that, by Definition~\ref{Def:GraphGroups}, $\theta_f = \theta_{\bar{f}}$ for $f\in E$, we obtain
\begin{align}
\mathfrak{g}'&:= \alpha(f, h, f')\notag\\[1mm] 
&= \wt{f} \bullet \wt{h} \bullet \wt{f'}\notag\\[1mm]
&= (\wt{f} \bullet \wt{\bar{f}}) \bullet (\wt{\bar{f}}^{- 1} \bullet \wt{h} \bullet \wt{\bar{f'}}^{- 1}) \bullet 
(\wt{\bar{f'}} \bullet \wt{f'})\notag\\[1mm]
&= \wt{\theta_f} \bullet \mathfrak{g}^{- 1} \bullet \wt{\theta_{f'}}\notag\\[1mm]
&= \wt{\theta_f} \bullet \wt{\lam_{\mathfrak{g}}}^{- 1} \bullet \wt{\overline{\varepsilon(\mathfrak{g})}} \bullet \wt{\theta_{\varepsilon(\mathfrak{g})}}^{- 1} \bullet
\wt{\rho_{\mathfrak{g}}}^{- 1} \bullet \wt{\theta_{f'}}.\label{Eq:gprimeDecomp}
\end{align}
We now apply Lemma~\ref{Lem:InvDepMult} to the decomposition
\[
\mathfrak{g}' = (\widetilde{\theta_f\lambda_{\mathfrak{g}}^{-1}}\bullet \widetilde{\overline{\varepsilon(\mathfrak{g})}}) \bullet \widetilde{\theta_{\varepsilon(\mathfrak{g})}^{-1} \rho_{\mathfrak{g}}^{-1} \theta_{f'}} = (\overline{\varepsilon(\mathfrak{g})}, \theta_f \lambda_{\mathfrak{g}}^{-1} H_{\overline{\varepsilon(\mathfrak{g})}}, \theta_f \lambda_{\mathfrak{g}}^{-1} \mapsto 1) \bullet (e, H, 1 \mapsto \theta^{-1}_{\varepsilon(\mathfrak{g})} \rho_{\mathfrak{g}}^{-1} \theta_{f'}) 
\]
of $\mathfrak{g}'$ obtained in (\ref{Eq:gprimeDecomp}). In the notation of that lemma, we may take $a = b = h =1$,  so that, by Part~(ii) of Definition~\ref{Def:GroupGraphGroups}, 
\[
\mathfrak{u} = \alpha(\overline{\varepsilon(\mathfrak{g})}, 1, e) = (\overline{\varepsilon(\mathfrak{g})}, H_{\overline{\varepsilon(\mathfrak{g})}}, 1 \mapsto 1). 
\]
From Parts (i) and (iii) of Lemma~\ref{Lem:InvDepMult}, we find that
\[
\varepsilon(\mathfrak{g}') = \varepsilon(\mathfrak{u}) = \overline{\varepsilon(\mathfrak{g})}
\]
and that 
\[
\lam_{\mathfrak{g}'} = \iota_{\varepsilon(\mathfrak{g}')}(\rho_{\mathfrak{g}'}^{- 1} \theta_f
\lam_{\mathfrak{g}}^{- 1}) \theta_{\varepsilon(\mathfrak{g})}^{- 1} \rho_{\mathfrak{g}}^{- 1} \theta_{f'}.
\]
Finally, using the accepted type (II) relations
$x_t x_{\bar{t}} = \theta_t$ for $t \in E$, Relation (\ref{Eq:AccTypeII}), and a suitable type (I) relation, it follows that
\begin{align*}
x_f h x_{f'} &= (x_f x_{\bar{f}}) (x_{\bar{f'}} h^{- 1} x_{\bar{f}})^{- 1} (x_{\bar{f'}} x_{f'})\\[1mm]
&= \theta_f (\rho_{\mathfrak{g}} x_{\varepsilon(\mathfrak{g})} \lam_{\mathfrak{g}})^{- 1} \theta_{f'}\\[1mm]
&= \theta_f \lam_{\mathfrak{g}}^{- 1} x_{\overline{\varepsilon(\mathfrak{g})}} \theta_{\varepsilon(\mathfrak{g})}^{- 1} \rho_{\mathfrak{g}}^{- 1} \theta_{f'}\\[1mm]
&= \rho_{\mathfrak{g}'} ((\rho_{\mathfrak{g}'}^{- 1} \theta_f \lam_{\mathfrak{g}}^{- 1}) x_{\varepsilon(\mathfrak{g}')}) \theta_{\varepsilon(\mathfrak{g})}^{- 1} \rho_{\mathfrak{g}}^{- 1} \theta_{f'}\\[1mm]
&= \rho_{\mathfrak{g}'} (x_{\varepsilon(\mathfrak{g}')} \iota_{\varepsilon(\mathfrak{g}')}(\rho_{\mathfrak{g}'}^{- 1} \theta_f \lam_{\mathfrak{g}}^{- 1})) \theta_{\varepsilon(\mathfrak{g})}^{- 1} \rho_{\mathfrak{g}}^{- 1} \theta_{f'}\\[1mm]
&= \rho_{\mathfrak{g}'} x_{\varepsilon(\mathfrak{g}')} \lam_{\mathfrak{g}'}, 
\end{align*}
as desired.
\end{proof}

\subsection{Solving the extension problem}
\label{Subsec:ExtensionProblem}
Group-like graphs of groups and their fundamental groups arise naturally in the study of group actions on groupoids. Let $\Phi: G \rightarrow\mathrm{Aut}({\bf X})$ be an action of the group $G$ on the connected groupoid ${\bf X}$. Then, in a manner analogous to classical Bass-Serre theory for group actions on (simplicial) graphs, one can assodiate with $\Phi$ a group-like graph of groups $(\mathbb{H}, \iota, \theta, \alpha)$ such that there exists a short exact sequence
\[
1 \longrightarrow N \longrightarrow \Pi \longrightarrow G \longrightarrow 1,
\]
where $\Pi = \pi_1(\mathbb{H}, \iota, \theta, \alpha)$ and $N \cong \pi_1({\bf X})$; cf.\ \cite{Ba1} and \cite{BM1}. In this setting of a \emph{structure theorem for group actions on groupoids}, Theorem~\ref{Thm:PresGpAct} arises as the special case where the action is of simplicial type (i.\,e., the groupoid ${\bf X}$ is a blow-up of a non-empty set). Our solution of the Schreier-type extension problem for group extensions of polygroups by groups below in particular provides an alternative approach to Theorem~\ref{Thm:PresGpAct}; see Section~\ref{Subsec:ProofMT}.

Given a group $H$ and a polygroup $(E; \circ,\, \bar{}\,, e),$ let $\cG\cE\cX\cT(E, H)$ and ${\rm GEXT}(E, H)$ be
as defined in \ref{Subsec:GroupExt}. We come now to the description of ${\rm GEXT}(E, H)$ via group-like graphs of 
groups as {\em generalized factor systems}.

Let ${\cG\cF\cS}(E, H)$ be the set of those group-like graphs of groups $(\H, \iota, \theta, \alpha)$ based on the pair
$(\Gam(E), H)$, such that the hyperoperation of the polygroup $E$ is connected with the operations $\alpha, \varepsilon$ of $(\mathbb{H}, \iota, \theta, \alpha)$ via the identity 
\begin{equation}
\label{Eq:GFS}
f \circ f' = \big\{\eps(\alpha(f, h, f')):\,h \in H\big\},\quad (f, f'\in E).
\end{equation}
$\cG\cF\cS(E, H)$ is a groupoid, possibly empty, with morphisms defined as follows.
\begin{de}
\label{Def:IsoGFS}
By a morphism $(\H, \iota, \theta, \alpha) \lra (\H', \iota', \theta', \alpha')$ between two group-like graphs of groups $(\mathbb{H}, \iota, \theta, \alpha)$ and $(\mathbb{H}', \iota', \theta', \alpha')$ based on the same pair $(E, H)$ we mean a pair $(c, d)$ consisting of
maps $c: E \lra H,$ $f \mapsto c_f,$ and $d: E \lra H, f \mapsto d_f,$ such that

\vspace{-2mm}
\ben
\item[\rm (i)] $c_e = d_e = 1,$ $d_f = c_{\bar{f}}^{- 1}$ for $f \neq \bar{f},$ and $d_f c_f \in H_f$ for $f = \bar{f},$

\vspace{2mm}

\item[\rm (ii)] $H'_f = c_f H_f c_f^{- 1},$ $\iota'_f(c_f h c_f^{- 1}) = d_f^{- 1} \iota_f(h) d_f$ for
$h \in H_f,$ and $\theta'_f = c_f \theta_f \iota_f(d_f c_f) d_f$ for $f = \bar{f},$

\vspace{2mm}
 
\item[\rm (iii)] For any triple $A = (f, h, f') \in E \times H \times E,$ setting $A' = (f, d_f^{- 1} h c_{f'}^{- 1}, 
f')$, we have
\begin{align*}
\eps'(\alpha'(A')) &= \eps(\alpha(A)),\\[1mm]
\rho'(\alpha'(A')) &= c_f\, \rho(\alpha(A))\, c_{\eps(\alpha(A))}^{- 1},\\[1mm]
\lam'(\alpha'(A'))(c_f\, x\, c_{\eps(\alpha(A))}^{- 1}) &= d_{\eps(\alpha(A))}^{- 1}\,\lam(\alpha(A))(x)\, d_{f'},
\quad(x \in \rho(\alpha(A))).
\end{align*} 
\een
\end{de}

The composition rule for composable morphisms $(c, d)$ and $(c', d')$ is $(c', d') \circ (c, d):= (c'', d'')$, where
$c''_f = c'_f c_f, d''_f = d_f d'_f$ for $f \in E$; in particular, 
\[
(c, d): (\H, \iota, \theta, \alpha) \to (\H', \iota', \theta', \alpha')\,\mbox{and}\,(f \mapsto c_f^{- 1}, 
f \mapsto d_f^{- 1}): (\H', \iota', \theta', \alpha') \to (\H, \iota, \theta, \alpha)
\] 
are isomorphisms inverse to each other. We denote by $GFS(E, H)$ the set of isomorphism classes (connected components)
of the groupoid $\cG\cF\cS(E, H)$.
\begin{te}
\label{Thm:SolExtP}
Let $H$ be a group, and let $(E, \circ,\,{}^-, e)$ be a polygroup. Then 
there exists a full and essentially surjective\footnote{Such a functor is also called ``representative'' or ``dense'' in the literature; cf.\, for instance \cite[Chap.~1, Note~3.11]{PP} and \cite[Def.~12.5]{HS}.} covariant functor 
\[
\delta: \cG\cF\cS(E, H) \lra \cG\cE\cX\cT(E, H).
\] 
In particular, the induced map $\wh{\delta}: {\rm GFS}(E, H) \lra {\rm GEXT}(E, H)$ is a bijection. 
\end{te}
\bp
We define a covariant functor $\delta:\cG\cF\cS(E, H) \lra \cG\cE\cX\cT(E, H)$ as follows. An object 
$(\H, \iota, \theta, \alpha)$ of the category $\cG\cF\cS(E, H)$ is sent by $\delta$ to the triple 
\[
\delta(\H, \iota, \theta, \alpha):= (\fG,\,\,\wt{}: H \rightarrow \fG,\, 
\eps: \fG \rightarrow E),\quad \mathfrak{G} = \pi_1(\mathbb{H}, \iota, \theta, \alpha),
\] 
which is an object of the category $\cG\cE\cX\cT(E, H)$ in view of  Part~(5) of Proposition~\ref{Prop:AssGroup} and Condition~(\ref{Eq:GFS}). 

A morphism $(c, d): (\H, \iota, \theta, \alpha) \lra (\H', \iota', \theta', \alpha')$ of the groupoid $\cG\cF\cS(E, H)$
is sent by $\delta$ to the morphism 
\[
\delta(c, d): (\fG,\,\,\wt{}\,:H \rightarrow \fG,\, \eps: \fG \rightarrow E) \lra (\fG',\,\,\wt{\,}\,\,'\!\!: H \rightarrow \fG',\, \eps'\!: \fG' \rightarrow E)
\] 
of the groupoid $\cG\cE\cX\cT(E, H)$, defined by 
\[
\delta(c, d)(f, a H_f, a \mapsto b):= (f, a c_f^{- 1} H_f', ac_f^{- 1} \mapsto d_f^{- 1} b),\quad (a, b\in H).
\]
This definition of $\delta(c,d)$ invites a fair amount of comment.

a) \textit{$\delta(c,d): \mathfrak{G} \rightarrow \mathfrak{G}'$ is well defined.} Indeed, let 
\[
(f, \tilde{a} H_f, \tilde{a} \mapsto \iota_f(h)^{-1} b)
\]
be a second description for the group element $\mathfrak{g} = (f, aH_f, a \mapsto b)\in\mathfrak{G}$, where $\tilde{a} = ah$ with some $h\in H_f$. Then 
\[
\tilde{a} c_f^{-1} H'_f = a h c_f^{-1} H'_f = a c_f^{-1}\cdot c_f h c_f^{-1} H'_f = a c_f^{-1} H'_f 
\]
since $c_f h c_f^{-1}\in H'_f$ by Part~(ii) of Definition~\ref{Def:IsoGFS}. Also, again using Part~(ii) of Definition~\ref{Def:IsoGFS}, the rigid map $\lambda: a c_f^{-1} H'_f \rightarrow H$ given by $ac_f^{-1} \mapsto d_f^{-1}b$ is seen to send $\tilde{a} c_f^{-1} = a c_f^{-1} \cdot c_f h c_f^{-1}$ to 
\[
\iota'_f(c_f h c_f^{-1})^{-1} d_f^{-1} b = (d_f^{-1} \iota_f(h) d_f)^{-1} d_f^{-1} b = d_f^{-1} \iota_f(h)^{-1} b,  
\]
thus agrees with the rigid map $\tilde{a} c_f^{-1} H'_f \rightarrow H$ given by $\tilde{a} c_f^{-1} \mapsto d_f^{-1} \iota_f(h)^{-1} b$, as desired.

b) \textit{$\delta(c,d)$ is injective.} Let $\mathfrak{g}_1 = (f_1, a_1H_{f_1}, a_1 \mapsto b_1)$ and $\mathfrak{g}_2 = (f_2, a_2H_{f_2}, a_2 \mapsto b_2)$ be elements of $\mathfrak{G}$, and suppose that
\begin{equation}
\label{Eq:deltaInj}
(f_1, a_1 c_{f_1}^{-1} H'_{f_1}, a_1 c_{f_1}^{-1} \mapsto d_{f_1}^{-1} b_1) = (f_2, a_2 c_{f_2}^{-1} H'_{f_2}, a_2 c_{f_2}^{-1} \mapsto d_{f_2}^{-1} b_2).
\end{equation}
We have to show that $\mathfrak{g}_1 = \mathfrak{g}_2$. Comparing first components in (\ref{Eq:deltaInj}), we see that $f_1=f_2$ while, comparing second components, we observe that $a_2 c_{f_1}^{-1} = a_1 c_{f_1}^{-1} h$ for some $h\in H_{f_1}'$. Thus,
\[
a_2 H_{f_2} = a_1 c_{f_1}^{-1}h c_{f_1} H_{f_1} = a_1 H_{f_1},
\]  
since $c_{f_1}^{-1} h c_{f_1} \in H_{f_1}$. Also, comparing third components in (\ref{Eq:deltaInj}), we find that 
\[
\iota_{f_1}'(h)^{-1} d_{f_1}^{-1} b_1 = d_{f_1}^{-1} b_2.
\]
Hence, the rigid map $a_1 H_{f_1} \rightarrow H$ given by $a_1\mapsto b_1$ sends $a_2 = a_1\cdot c_{f_1}^{-1} h c_{f_1}$ to 
\[
\iota_{f_1}(c_{f_1}^{-1} h c_{f_1})^{-1} b_1 = (d_{f_1} \iota_{f_1}'(h) d_{f_1}^{-1})^{-1} b_1 = d_{f_1}\cdot \iota_{f_1}'(h)^{-1} d_{f_1}^{-1} b_1 = b_2,
\]
as required.

c) \textit{$\delta(c,d)$ is a homomorphism.} For $i=1,2$, let $\mathfrak{g}_i = (f_i, a_i H_{f_i}, a_i \mapsto b_i) \in \fG$, and set $\mathfrak{g}_i':= \delta(c, d)(\mathfrak{g}_i) = (f_i, a_i c_{f_i}^{- 1} H_{f_i}',
a_i c_{f_i}^{- 1} \mapsto d_{f_i}^{- 1} b_i) \in \fG'$. By (\ref{Eq:Operation}), we get
\[
\mathfrak{g}_1 \bullet \mathfrak{g}_2 = (\eps(\alpha(A)), a_1 \rho(\alpha(A)), a_1 x \mapsto \lam(\alpha(A))(x) b_2), 
\]
where $A = (f_1, b_1 a_2, f_2),$ and
\begin{equation}
\label{Eq:PrimeProds}
\mathfrak{g}_1' \bullet' \mathfrak{g}_2' = (\eps'(\alpha'(A')), a_1 c_{f_1}^{- 1} \rho'(\alpha'(A')), a_1 c_{f_1}^{- 1} y \mapsto 
\lam'(\alpha'(A'))(y) d_{f_2}^{- 1} b_2),
\end{equation}
where $A' = (f_1, d_{f_1}^{- 1} b_1 a_2 c_{f_2}^{- 1}, f_2)$. Setting $\rho(\alpha(A)) = p H_{\varepsilon(\alpha(A))}$ with some suitable $p\in H$, we thus find that
\begin{equation}
\label{Eq:deltaImage}
\delta(c,d)(\mathfrak{g}_1\bullet \mathfrak{g}_2) = \big(\varepsilon(\alpha(A)),\, a_1 p c_{\varepsilon(\alpha(A))}^{-1} H'_{\varepsilon(\alpha(A))},\, a_1 p c_{\varepsilon(\alpha(A))}^{-1} \mapsto dp_{\varepsilon(\alpha(A))}^{-1} \lambda(\alpha(A))(p) b_2\big),
\end{equation}
which now has to be compared with (\ref{Eq:PrimeProds}). Noting that the triples $A, A'$ fit into the context of Part (iii) of Definition~\ref{Def:IsoGFS}, we in particular infer that
\[
\varepsilon'(\mathfrak{g}_1' \bullet' \mathfrak{g}_2') = \varepsilon'(\alpha'(A')) = \varepsilon(\alpha(A)) = \varepsilon'(\delta(c,d)(\mathfrak{g}_1\bullet \mathfrak{g}_2))
\]
and that
\begin{align*}
\rho'(\mathfrak{g}_1' \bullet' \mathfrak{g}_2') &= a_1 c_{f_1}^{-1} \rho'(\alpha'(A'))\\[1mm] 
&= a_1 c_{f_1}^{-1} (c_{f_1} \rho(\alpha(A)) c_{\varepsilon(\alpha(A))}^{-1})\\[1mm] 
&= a_1 p H_{\varepsilon(\alpha(A))} c_{\varepsilon(\alpha(A))}^{-1}\\[1mm] 
&= a_1 p c_{\varepsilon(\alpha(A))}^{-1} H_{\varepsilon(\alpha(A))}'\\[1mm] 
&= \rho'(\delta(c,d)(\mathfrak{g}_1\bullet \mathfrak{g}_2)),
\end{align*}
where we have made use of Part~(ii) of Definition~\ref{Def:IsoGFS} in Step~4. Finally, again using Part~(iii) of Definition~\ref{Def:IsoGFS}, we see that the element
\[
a_1 p c_{\varepsilon(\alpha(A))}^{-1} = a_1 c_{f_1}^{-1} \cdot c_{f_1} p c_{\varepsilon(\alpha(A))}^{-1} \in a_1 c_{f_1}^{-1} \rho'(\alpha'(A'))
\]
is mapped under $\lambda'(\mathfrak{g}_1' \bullet' \mathfrak{g}_2')$ to
\[
\lambda'(\alpha'(A'))(c_{f_1} p c_{\varepsilon(\alpha(A))}^{-1}) d_{f_2}^{-1} b_2 = d_{\varepsilon(\alpha(A))}^{-1} \lambda(\alpha(A))(p) d_{f_2} d_{f_2}^{-1} b_2 = d_{\varepsilon(\alpha(A))}^{-1} \lambda(\alpha(A))(p) b_2,  
\]
so that $\lambda'(\mathfrak{g}_1' \bullet' \mathfrak{g}_2') = \lambda'(\delta(c,d)(\mathfrak{g}_1\bullet \mathfrak{g}_2))$. It follows that
\[
\delta(c, d)(\mathfrak{g}_1 \bullet \mathfrak{g}_2) = \delta(c,d)(\mathfrak{g}_1) \bullet' \delta(c,d)(\mathfrak{g}_2), \quad(\mathfrak{g}_1, \mathfrak{g}_2\in\mathfrak{G}),
\]
as desired. Surjectivity of $\delta(c,d)$ is obvious, while the equations $\delta(c,d)\circ \widetilde{} = \widetilde{}\,\,'$ and $\varepsilon'\circ \delta(c,d) = \varepsilon$ follow from Parts~(i) and (ii) of Definition~\ref{Def:IsoGFS} (more precisely the facts that $c_e = d_e = 1$ and $H_e' = c_e H_e c_e^{-1} =  H$) plus the definition of the map $\delta(c,d)$. Thus, $\delta(c,d)$ is indeed a morphism in the category $\mathcal{GEXT}(E,H)$ from $(\mathfrak{G},\, \widetilde{}, \varepsilon)$ to $(\mathfrak{G}',\, \widetilde{}\,\,', \varepsilon')$ as claimed.  Functoriality of $\delta$ is clear.
 
Next, we show that the functor $\delta$ is {\em full}; in particular,
the induced map $\wh{\delta}: {\rm GFS}(E, H) \lra {\rm GEXT}(E, H)$ is injective. Let 
$\fH = (\H, \iota, \theta, \alpha), \fH' = (\H', \iota', \theta', \alpha'),$ be objects in the category $\cG\cF\cS(E, H)$, and let $\delta(\fH) = (\fG,\,\wt{}\,, \eps), 
\delta(\fH') = (\fG',\,\wt{}\,\,', \eps')$ be their images under the functor $\delta$. Let $\varphi: \fG \rightarrow \fG'$ be a group isomorphism satisfying the compatibility conditions 
$\varphi \circ\,\wt{} = \wt{}\,\,'$ and $\eps' \circ \varphi = \eps$; that is, $\varphi:\delta(\mathfrak{H})  \rightarrow \delta(\mathfrak{H}')$ is a morphism in the category $\mathcal{GEXT}(E,H)$. In order to lift $\varphi$ to a morphism $(c, d): \fH \rightarrow \fH',$ we proceed as follows. Let $f \in E$. As
$\eps(\wt{f}) = f = \eps'(\varphi(\wt{f}))$, there exist $c_f, d_f \in H$ such that $\varphi(\wt{f}) = 
(f, c_f^{- 1} H_f', c_f^{- 1} \mapsto d_f^{- 1})$. Clearly, we may choose the $c_f$'s and $d_f$'s such that the 
normalizing condition (i) from Definition~\ref{Def:IsoGFS} is satisfied, and we check that Conditions (ii) and (iii) hold, to conclude that $(c, d) \in {\rm Mor}(\fH, \fH')$. For instance, the identities in Part~(iii) of Definition~\ref{Def:IsoGFS} result from the fact that $\varphi$ is a homomorphism applied to, say, elements $\mathfrak{g}_1 = (f_1, H_{f_1}, 1\mapsto h)$ and $\mathfrak{g}_2 = \widetilde{f_2}$. We omit the straightforward but somewhat tedious computations. Since, by construction, the group homomorphisms $\delta(c,d)$ and $\varphi$ agree on the generating system $\widetilde{H} \cup \widetilde{E}$ of $\mathfrak{G}$, we have  $\delta(c, d) = \varphi$ as desired. 

Finally, we have to show that $\delta$ is {\em essentially surjective}; that is, given any  object 
\[
\mathbb{G} = \big(G,\, j: H \rightarrow G,\, p: G\rightarrow E\big) 
\]
of $\cG\cE\cX\cT(E, H)$, there exists an object $\fH$  of $\cG\cF\cS(E, H)$ such that $\mathbb{G} \cong \delta(\fH)$; 
in particular, the induced map $\wh{\delta}$ is surjective. In order to exhibit such an object $\fH$,  we choose a section $\sigma: E \rightarrow G$ of the projection map $p: G \rightarrow E$ such that $\si(e) = 1, \si(\bar{f}) = \si(f)^{- 1}$ for $f \neq \bar{f}$, 
and $\si(f)^2 \in j(H)$ for $f = \bar{f}$. We then define a graph of groups $(\H, \iota, \theta)$ based on 
$(\Gam(E), H)$ via
\begin{align*}
H_f &:= j^{- 1}\big(j(H) \cap \si(f) j(H) \si(f)^{- 1}\big),\quad (f\in E),\\[1mm]
\iota_f(h) &:= j^{- 1}\big(\si(f)^{- 1} j(h) \si(f)\big),\quad(f\in E,\, h \in H_f),\\[1mm]
\theta_f &:= j^{- 1}\big(\si(f) \si(\bar{f})\big),\quad (f\in E).
\end{align*}
It is straightforward to check Conditions~(i)--(iii) in Definition~\ref{Def:GraphGroups}. Moreover, since the surjective map 
\[
H \times E \times H \lra G,\quad (h_1, f, h_2) \mapsto j(h_1) \si(f) j(h_2),
\]
determines the equivalence relation (\ref{Eq:Equiv}) on the product $H \times E \times H$, we obtain a well-defined bijection 
\[
\zeta: G \lra \fG,\quad j(h_1) \si(f) j(h_2) \mapsto (f, h_1 H_f, h_1 \mapsto h_2).
\] 
By construction, we have $\zeta\circ j = \,\tilde{}$ and $\varepsilon\circ \zeta = p$.
 
In order to extend the graph of groups $(\H, \iota, \theta)$, in accordance with Definition~\ref{Def:GroupGraphGroups}, to a
group-like graph of groups, we have to define a suitable  map $\alpha: E \times H \times E \rightarrow \fG$. For a triple 
$A = (f, h, f') \in E \times H \times E$, we choose $a = a_A, b = b_A \in H$ such that 
\[
\si(f) j(h) \si(f') = j(a) \si(p(\si(f) j(h) \si(f'))) j(b),
\] 
and define $\alpha(A) \in \fG$ via
\begin{align*}
\eps(\alpha(A)) &:= p(\si(f) j(h) \si(f')),\\[1mm] 
\rho(\alpha(A)) &:= a H_{\eps(\alpha(A))},\\[1mm] 
\lam(\alpha(A))(ap) &:= \iota_{\varepsilon(\alpha(A))}(p)^{-1} b,\quad(p\in H_{\varepsilon(\alpha(A))}),
\end{align*}
noting that this definition does not depend on the choice of the elements $a$ and $b$. One easily checks Conditions
(i)--(iv) of Definition~\ref{Def:GroupGraphGroups} to conclude that $\fH = (\H, \iota, \theta, \alpha)$ is a group-like graph of groups based on the pair $(\Gamma(E), H)$. In addition, $\fH$ also satisfies (\ref{Eq:GFS}) since, by assumption, the projection map $p: G \rightarrow E$ 
induces an isomorphism $j(H) \sm G/ j(H) \lra E$ of polygroups. Consequently, $\fH$ is an object of the category $\cG\cF\cS(E, H)$. Finally,
using the definition (\ref{Eq:Operation}) of the group operation $\bullet$ on $\fG$, it is straightforward to see that the bijective
map $\zeta: G \rightarrow \fG$ defined above is a group homomorphism, and thus a morphism in the category $\cG\cE\cX\cT(E, H)$ from $\mathbb{G}$ to $\delta(\fH)$, whence the required isomorphism $ \mathbb{G} \cong \delta(\mathfrak{H})$.
\ep
\subsection{Proof of Theorem~\ref{Thm:PresGpAct}}
\label{Subsec:ProofMT}
Given a group $G$ and a subgroup $H \leq G$, we denote by $E$ the canonical polygroup with support $H \sm G/ H$. Thus, 
\[
\mathbb{G}:= (G,\, j: H \rightarrow G,\, p: G \rightarrow E) 
\]
is an object of the category $\cG\cE\cX\cT(E, H)$,  
where $j$ is the natural inclusion map, and $p$ is the canonical projection given by $g \mapsto C(g):= H g H$. By Theorem~\ref{Thm:SolExtP} there exists some object 
$\cH = (\H, \iota, \theta, \alpha)$ in  $\cG\cF\cS(E, H)$, such that 
\[
\mathbb{G} \cong \delta(\fH) = (\fG,\,\wt{}\,:H \rightarrow \fG, \eps: \fG \rightarrow E), 
\]
where $\mathfrak{G} = \pi_1(\H, \iota, \theta, \alpha)$. Theorem~\ref{Thm:PresGpAct} is now 
an immediate consequence of Theorem~\ref{Thm:Presentation}, the proof of Theorem~\ref{Thm:SolExtP} (more precisely, the construction of the group-like graph of groups $\mathcal{H}$), and Lemma~\ref{Lem:ReduceII}.

\end{document}